\documentclass[12pt]{amsart}
\usepackage{a4wide}
\usepackage{times}
\usepackage{mathtools, amssymb}
\usepackage{graphicx,xspace}
\usepackage{epsfig}
\usepackage[usenames,dvipsnames]{xcolor}
\usepackage{tikz}
\usepackage[T1]{fontenc}
\usepackage[utf8]{inputenc}
\usepackage{bm}
\usepackage{cancel}

\usepackage{hyperref,pifont}

\usepackage[capitalize]{cleveref}

\theoremstyle{plain}

\newtheorem{lemma}{Lemma}[section]
\newtheorem{theorem}[lemma]{Theorem}
\newtheorem{corollary}[lemma]{Corollary}
\newtheorem{proposition}[lemma]{Proposition}
\newtheorem{definition}[lemma]{Definition}
\newtheorem*{hypothesis}{Hypothesis}

\theoremstyle{remark}
\newtheorem{remark}[lemma]{Remark}
\newtheorem{example}[lemma]{Example}
\newtheorem{observation}[lemma]{Observation}

\def\eps{\varepsilon}
\def\si{\sigma}         
\def\Sn{\mathfrak{S}}   

\def\xx{\bm{x}}
\def\yy{\bm{y}}
\def\XX{\vec{\mathbf{x}}}
\def\YY{\vec{\mathbf{y}}}

\def\TT{\mathbb{T}}

\def\CCC{\mathcal{C}}  
\def\EEE{\mathcal{E}}   
\def\SSS{\mathcal{S}}   
\def\TTT{\mathcal{T}}   
\def\HHH{\mathcal{H}}   
\def\DDD{\mathcal{D}}   
\def\gauche{\mathrm{left}}
\def\droite{\mathrm{right}}
\def\proba{\mathbb{P}}

\def\patterntree{{t_0}}     
\def\One{\bm{1}}
\def\nonp{{\scriptscriptstyle{\mathrm{not}} {\oplus}}}
\def\nonm{{\scriptscriptstyle{\mathrm{not}} {\ominus}}}

\def\O{\mathcal{O}}

\def\Si{\mathsf{Si}}
\def\QE{\mathsf{QE}}
\def\Sep{\mathcal T_{\text{sep}}}

\def\SsEnsemble{\mathcal{I}}

\def\tkin{\bm t^{[k]}_{i,n}}
\def\tkinhead{\bm t^{[k,\text{\footnotesize head}]}_{i,n}}
\newcommand{\Internal}[1]{\mathrm{Int}(#1)}
\newcommand{\Leaves}[1]{\mathrm{Lf}(#1)}
\newcommand{\transpose}[1]{\prescript{\intercal}{}{#1}}

\DeclareMathOperator{\Id}{Id}

\DeclareMathOperator{\perm}{perm}
\DeclareMathOperator{\Perm}{Perm}

\DeclareMathOperator{\pat}{pat}

\DeclareMathOperator{\Cat}{Cat}

\DeclareMathOperator{\Reduc}{Red}

\DeclareMathOperator{\Crit}{Crit}
\DeclareMathOperator{\SpectralRadius}{SR}
\DeclareMathOperator{\diag}{diag}
\DeclareMathOperator{\Com}{Com}
\newcommand{\InducedPerm}{\mathbf{Perm}}

\newcommand{\set}[1]{\left\{#1\right\}}
\newcommand\restr[2]{{%
		\left.\kern-\nulldelimiterspace %
		#1 %
		\right|_{#2} %
	}}

\def\Av{\mathrm{Av}}

\title[Scaling limits of permutation classes]{Scaling limits of permutation classes \\ with a finite specification: a dichotomy}

\author[F. Bassino]{Frédérique Bassino}
       \address[FB]{Université Paris 13, Sorbonne Paris Cité, LIPN, CNRS UMR 7030, F-93430 Villetaneuse, France}
       \email{bassino@lipn.univ-paris13.fr}

 \author[M. Bouvel]{Mathilde Bouvel}
 \author[V. Féray]{Valentin Féray}
       \address[MB,VF]{Institut für Mathematik, Universität Zürich, Winterthurerstr. 190, CH-8057 Zürich, Switzerland}
       \email{mathilde.bouvel@math.uzh.ch}
       \email{valentin.feray@math.uzh.ch}

 \author[L. Gerin]{Lucas Gerin}
       \address[LG]{CMAP, \'Ecole Polytechnique, CNRS, Route de Saclay, 91128 Palaiseau Cedex, France}
       \email{gerin@cmap.polytechnique.fr}

  \author[M. Maazoun]{Mickaël Maazoun}
   \address[MM]{École Normale Supérieure de Lyon, UMPA UMR 5669 CNRS, 46 allée d’Italie, 69364 Lyon Cedex 07, France}
   \email{mickael.maazoun@ens-lyon.fr}

 \author[A. Pierrot]{Adeline Pierrot}
 \address[AP]{LRI, Université Paris-Sud, Bat. 650 Ada Lovelace, 91405 Orsay Cedex, France}
       \email{adeline.pierrot@lri.fr}



\begin{document}

\begin{abstract}

We consider uniform random permutations in classes having a finite combinatorial specification for the substitution decomposition. These classes include (but are not limited to) all permutation classes with a finite number of simple permutations.  Our goal is to study their limiting behavior in the sense of permutons.
  
The limit depends on the structure of the specification restricted to families with the largest growth rate.
When it is strongly connected, two cases occur.
If the associated system of equations is linear, the limiting permuton is a deterministic $X$-shape.
Otherwise, the limiting permuton is the Brownian separable permuton, a random object that already appeared as the limit of most substitution-closed permutation classes, among which the separable permutations.
Moreover these results can be combined to study some non strongly connected cases. 

To prove our result, we use a characterization of the convergence
of random permutons by the convergence of random subpermutations.
Key steps are the combinatorial study, via substitution trees,
of families of permutations
with marked elements inducing a given pattern,
and the singularity analysis of the corresponding generating functions.
\end{abstract}

\maketitle

\tableofcontents

\section{Introduction}

\subsection{Context and background} 
In this paper we consider sets of permutations (of all sizes), called {\em classes},
which are classical objects in enumerative combinatorics \cite{VatterSurvey}.
By definition, a permutation class is a set of permutations downward closed
with respect to a natural notion of substructures, called patterns
(see \cref{ssec:basics_patterns} for the relevant definitions).
The general question we are interested in is the description of the asymptotic properties of
a uniform random permutation of large size in a class.
The literature on the subject has developed quickly in the past few years
with a variety of approaches,
see for example \cite{Nous2,Jacopo,HoffmanBrownian1,Janson321,MadrasPehlivan,MinerPak}.
A detailed presentation of this literature can be found for example in \cite[Section~1.1]{Nous1}.

Permutation classes are most often studied with an enumerative perspective, 
and among the combinatorial tools introduced to enumerate permutation classes
is the so-called {\em substitution decomposition}. 
We present briefly this notion here in an informal way,
precise statements will be given in \cref{sec:framework}.

We see a permutation $\sigma$ (of size $n$) as its {\em diagram}, \emph{i.e.}  a square grid with dots
at coordinates $(i,\sigma(i))$ (for $i$ in $\{1,\dots,n\}$).
For $\theta$ a permutation of size $d$, the substitution $\theta[\pi^{(1)},\dots,\pi^{(d)}]$
is obtained by inflating each point $\theta(i)$
of $\theta$ by a square containing the diagram of $\pi^{(i)}$, see \cref{fig:Subst_Intro}. 
\begin{figure}[htbp]
\[\includegraphics[width=95mm]{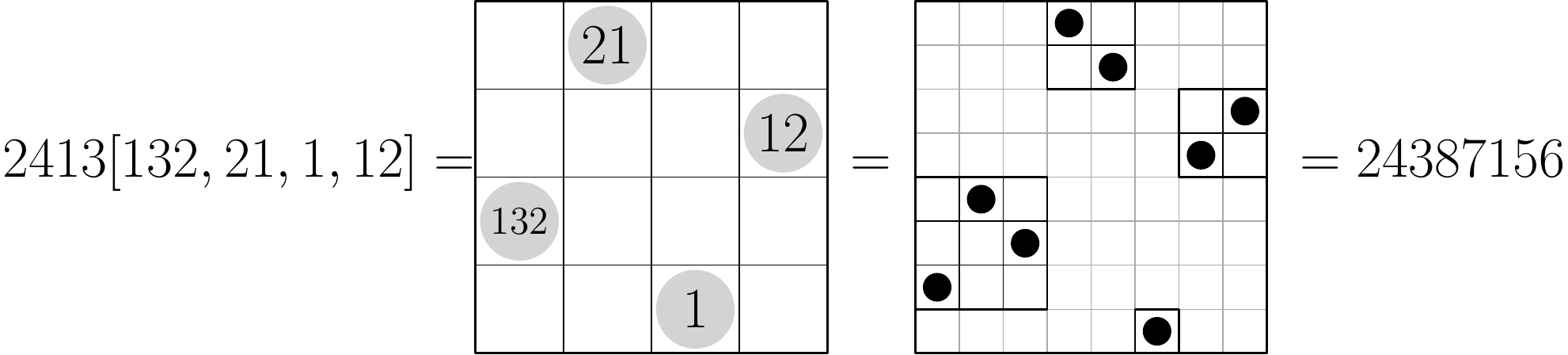}\]
  \caption{Example of substitution of permutations.}
  \label{fig:Subst_Intro}
\end{figure}

Each permutation can be decomposed in a canonical way
as successive substitutions, starting from the indecomposable elements,
which are called {\em simple permutations} (defined in~\cite{AlbertAtkinsonKlazar}).
This allows to encode bijectively permutations by trees, called {\em substitution trees}.
In the sequel, classes of permutations are identified with the set of their substitution trees,
and therefore denoted by $\mathcal T$.
We are interested in classes $\mathcal T$ with a nice recursive description, 
namely a {\em finite} system of combinatorial equations for $\mathcal T$,
called {\em specification}.

To fix the ideas, we explain how such a specification 
can be obtained for the famous class of {\em separable permutations}.
One way to define the class $\Sep$ of separable permutations
is as the smallest set of permutations containing $1$ and stable by taking substitutions
in $12$ and $21$.
Therefore $\Sep$ satisfies 
\[\Sep=\{\bullet\}\ \uplus\ 12[\Sep,\Sep]\ \uplus\ 21[\Sep,\Sep].\]
This defines recursively the elements of $\Sep$, this is however {\em not} a combinatorial
specification, since 
some separable permutations have several decompositions witnessing their membership to $12[\Sep,\Sep]$ (or to $21[\Sep,\Sep]$). 
To express $\Sep$ in a way that allows only one decomposition of any separable permutation 
(and make the tree decomposition unique), we need to consider
the subsets $\Sep^{\nonp}$ (resp. $\Sep^{\nonm}$) consisting in separable permutations
that cannot be written as $12[\pi^{(1)},\pi^{(2)}]$ (resp. $21[\pi^{(1)},\pi^{(2)}]$).
It can easily be shown that these three families satisfy the following
combinatorial specification
\begin{equation}
\begin{cases}
\quad\Sep &= \{\bullet\} \  \biguplus \ 
\oplus[\Sep^\nonp,\Sep]
\  \biguplus \ 
\ominus[\Sep^\nonm,\Sep];
\\
\ \Sep^{\nonp}  &= \{\bullet\} \  \biguplus \ 
\ominus[\Sep^\nonm,\Sep];
\\
\ \Sep^{\nonm}  &= \{\bullet\} \  \biguplus \ 
\oplus[\Sep^\nonp,\Sep].
\end{cases}
\end{equation}

This example is a particular case of a more general family of permutation classes,
that of {\em substitution-closed classes}.
All these classes have combinatorial specifications with three equations
(given below in \cref{eq:SpecifClosedClasses}).
In \cite{Nous2},
we obtained all the possible limiting shapes for such classes
with a unified combinatorial approach and a careful generating function analysis.

Another sufficient condition for having a specification
 is that the class contains finitely many {\em simple permutations}. 
It was proved by \cite{AA05} that such a class $\mathcal T$ always has an algebraic generating function,
and  \cite{BBPPR} provides  an algorithmic way to compute
a specification for $\mathcal T$.
Unlike for substitution-closed classes, the number of equations is not fixed
(and grows quickly in examples), making a unified analysis much harder.
We also note that a class $\mathcal T$ may admit such a finite specification,
while containing infinitely many simple permutations.
This is the case of the class of {\em pin-permutations} \cite{BHV,PinPerm}.

A combinatorial specification for the class $\mathcal T$ provides in an automatic way a random sampler \cite{FlajoletZimmerman,Boltzmann}
of objects in $\mathcal T$.
We show in \cref{fig:PleinDeSimus} large permutations in several classes
obtained in this way (using Boltzmann generators).
Permutations are here represented by their diagrams.
As we see on these examples, various qualitative asymptotic behaviors occur.
The results of the present paper apply in particular to each of these four cases,
giving an explicit limit shape result.
\begin{figure}[htbp]
\begin{tabular}{cccc}
 \includegraphics[width = 0.231\linewidth]{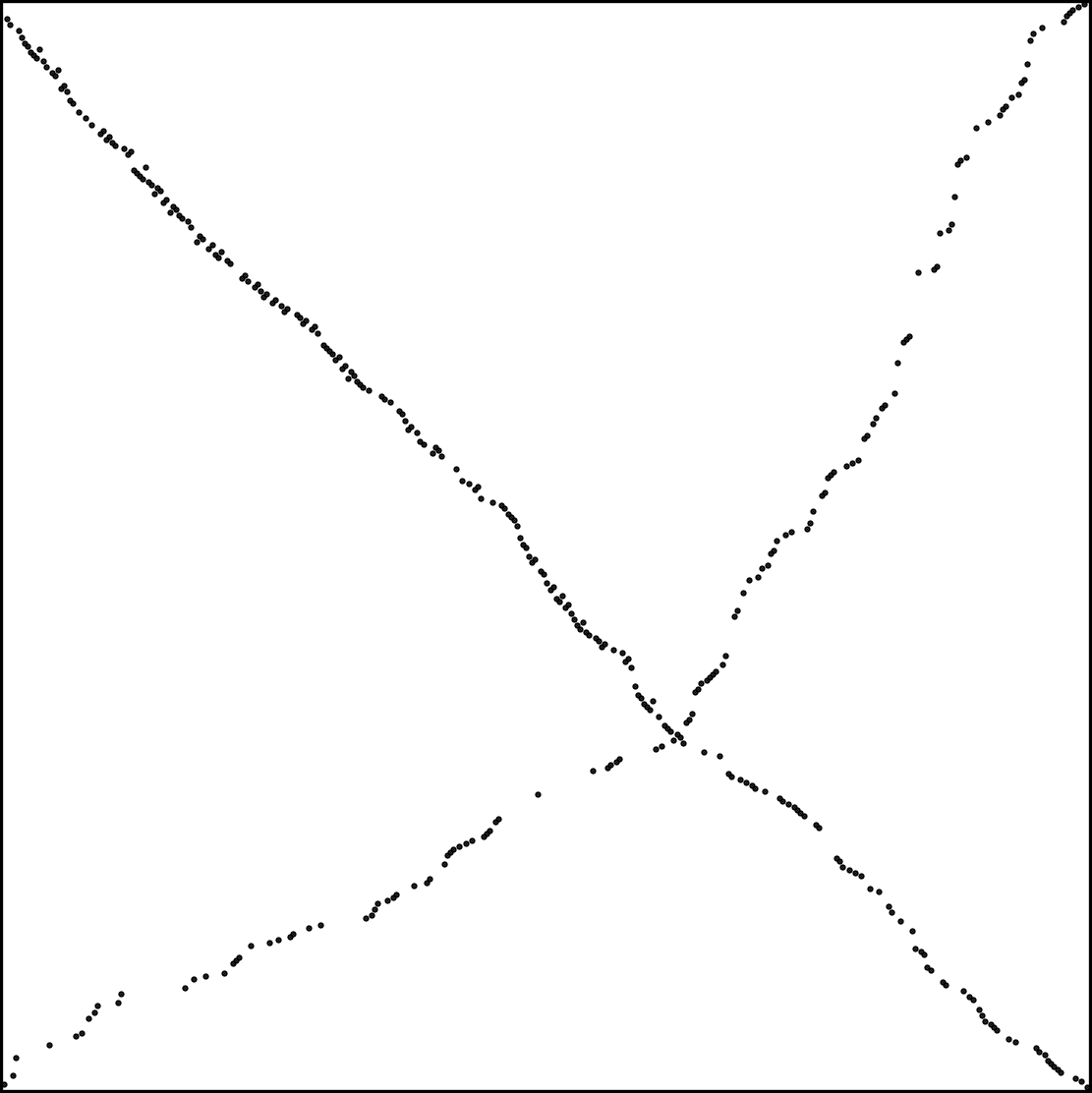} & \includegraphics[width = 0.238\linewidth]{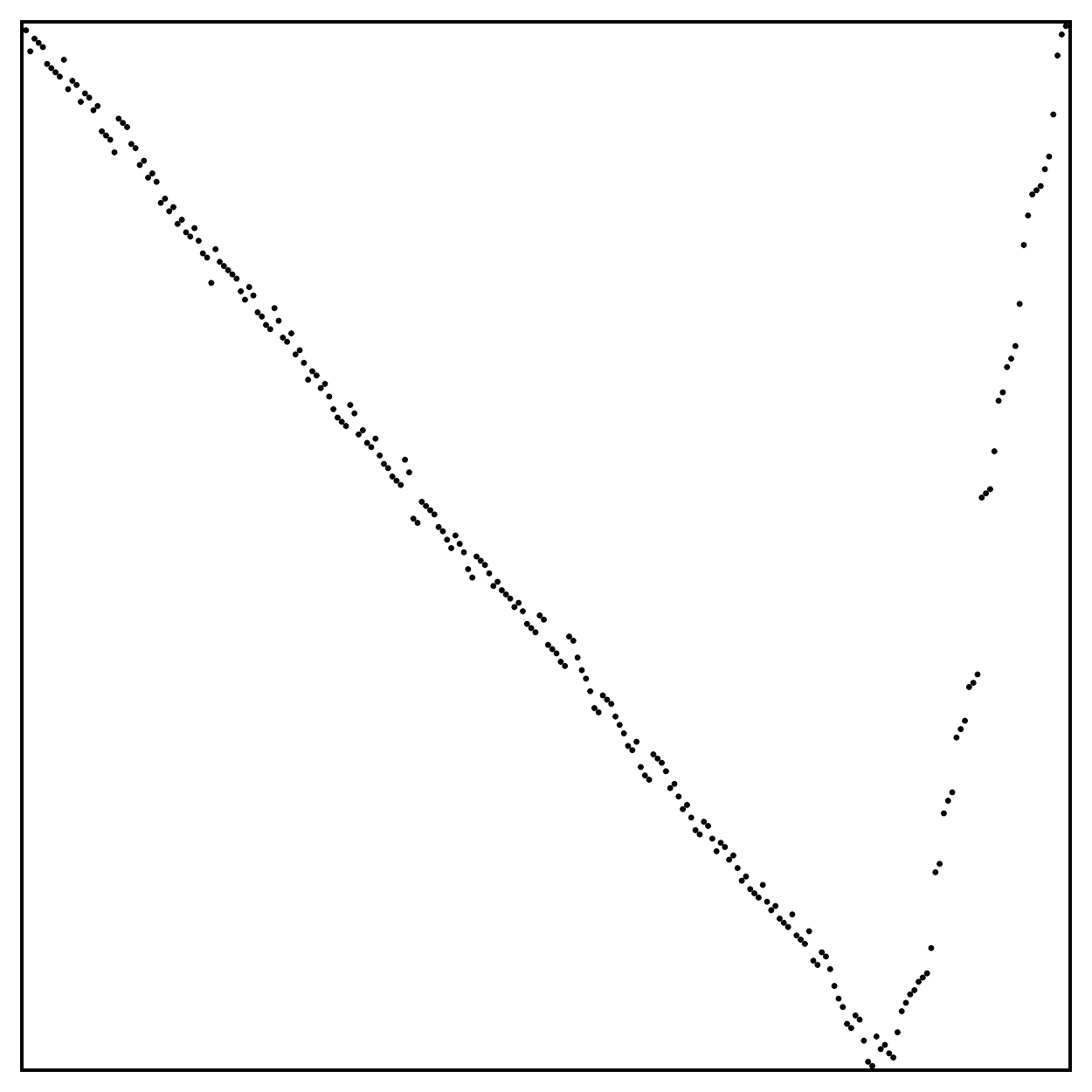}& \includegraphics[width = 0.24\linewidth]{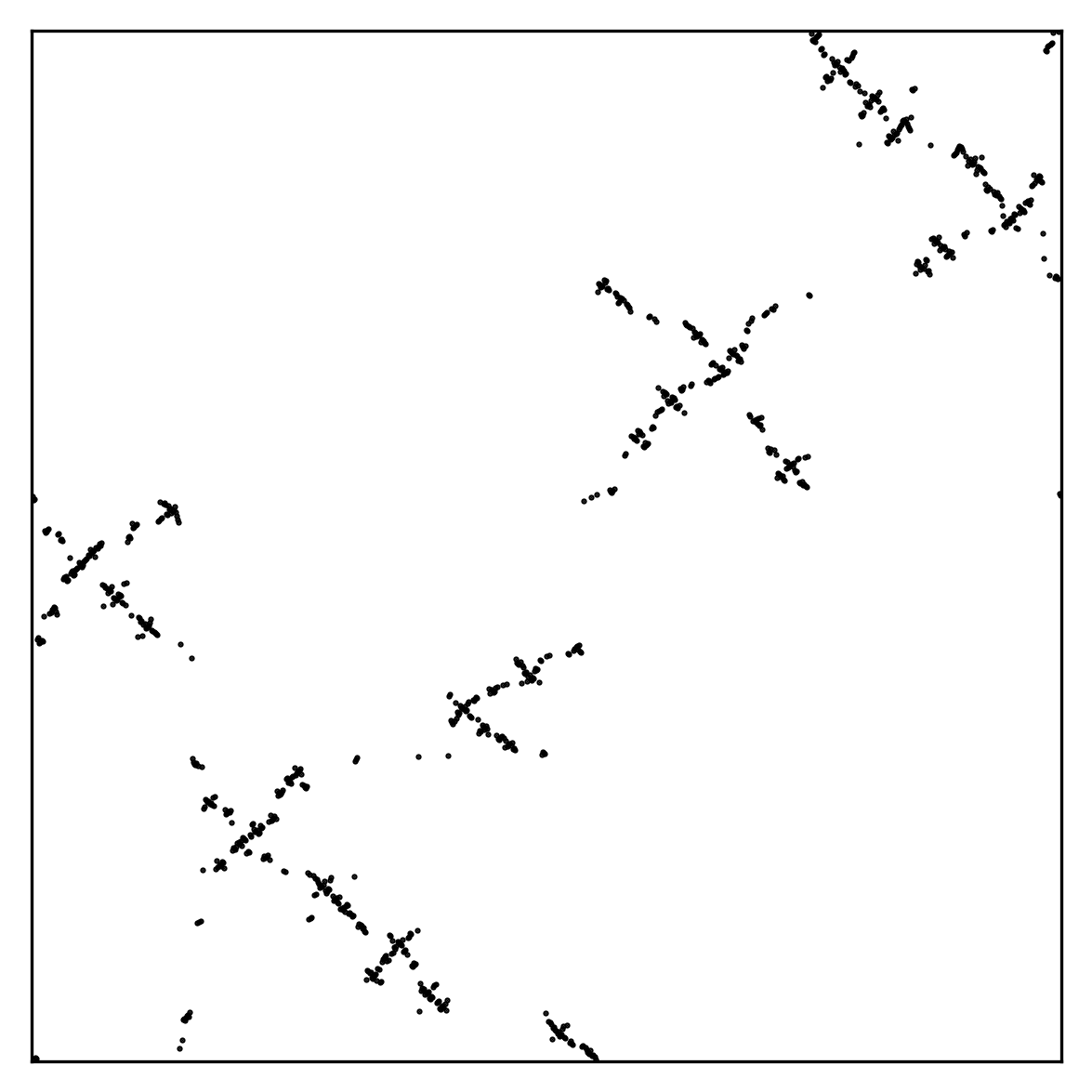} & \includegraphics[width = 0.238\linewidth]{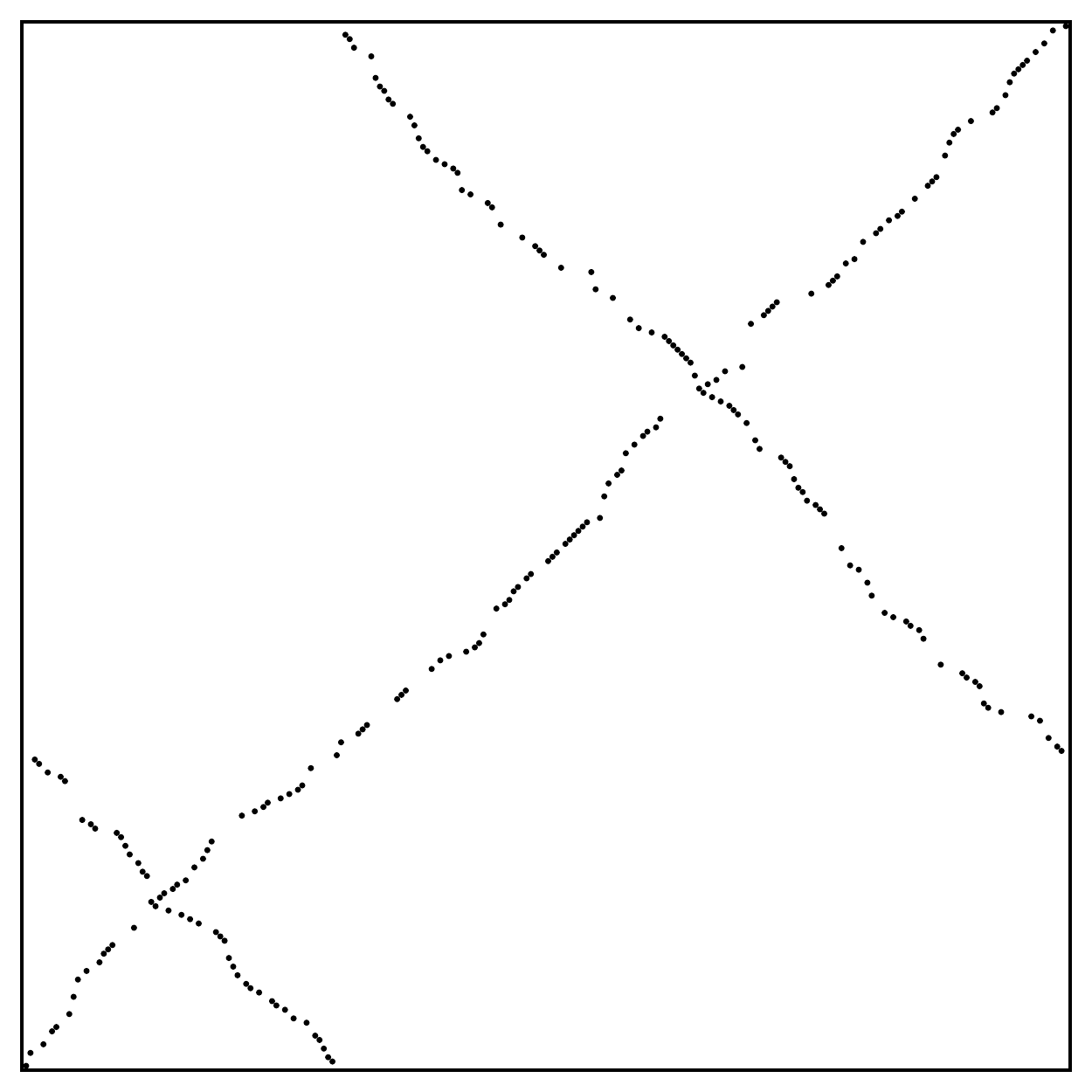}\\
 (a) & (b) & (c) & (d)
\end{tabular}
\caption{Large uniform random permutations in four different finitely specified classes. These four cases are covered by the present paper.}
\label{fig:PleinDeSimus}
\end{figure}

Our limiting results are phrased in the framework of {\em permutons}, 
which can be thought of as infinite rescaled permutations.
A permuton is a measure on $[0,1]^2$, whose projections on the horizontal and vertical axes
are the uniform measure on $[0,1]$.
Every permutation defines a permuton, by considering its rescaled diagram.
The set of permutons is endowed with the weak convergence topology of measures,
providing a natural notion of convergence for permutations. 
We review this setting in further details in \cref{ssec:permutons_intro}.

\subsection{Presentation of the results}
We consider a permutation class $\mathcal T$ with a specification.
This specification involves several families of permutations 
$\mathcal T_0=\mathcal T$, $\mathcal T_1$, \dots, $\mathcal T_d$.
Among these families, the ones with the smallest radius of convergence play a prominent role
in the asymptotics; we call such families {\em critical}.
In our case, the class $\mathcal T$ is always critical 
and we assume that the other critical families are $\mathcal T_1$, \dots, $\mathcal T_c$ for some $c \le d$.

An important information to study  $\mathcal T$ through its specification is to know
which families appear in the equation defining each $\mathcal T_i$ in the specification.
This is traditionally encoded in a directed graph with vertex set $\{\mathcal T_0,\dots,\mathcal T_d\}$,
called {\em dependency graph} of the specification.
A standard assumption to study combinatorial specifications
is that this graph is strongly connected (see \cite[Thm. VII.6, p. 489]{Violet}, \cite[Thm. 2.33]{Drmota} or \cite[Lemma 2]{BanderierDrmota}), 
implying in particular that all families are critical.
This assumption is too strong in our context.
We shall instead assume that the dependency graph {\em restricted to the critical families}
is strongly connected.
We will discuss later some methods to relax this assumption.

Under the strong connectivity assumption above, 
there are two possible asymptotic behaviors for a uniform random permutation $\bm \sigma_n$ in $\mathcal T$. 
\begin{itemize}
\item Either the combinatorial equation defining each {\em critical} family $\mathcal T_i$ 
is linear in every {\em critical} family (it may depend nonlinearly on non-critical families). 
This is referred to as the {\em essentially linear case}. 
In this case, we prove in \cref{Th:linearCase} the convergence of $\bm \sigma_n$ in distribution
towards a deterministic permuton,
that has a shape of an $X$, \emph{i.e.} is supported by four line segments from the corners of $[0,1]^2$
to a common central point.
This permuton depends on the class $\mathcal T$ only through a quadruple $\bm p$ 
whose components are in $[0,1]$, sum up to $1$ and indicate the mass of the four line segments
(thus determining the coordinates of the central point). 
The simulations (a) and (b) of \cref{fig:PleinDeSimus} fit in this framework.
In the second case, the limiting $X$-permuton is in some sense degenerate:
only two components of its quadruple $\bm p$ are nonzero,
explaining the $V$-shape.
The statements regarding those two classes may be found in Sections  \ref{sec:ClasseXTilde} and \ref{sec:ClasseV}.
\item The other possibility (called {\em essentially branching case}) 
is that the equation defining some {\em critical} family $\mathcal T_i$ 
involves a product of at least two {\em critical} families (which may be the same).
In this case, we prove in \cref{Th:branchingCase} that $\bm \sigma_n$ converges 
in distribution towards a biased Brownian separable permuton,
as introduced in \cite{Nous2,MickaelConstruction}.
In this case, the limit depends only on $\mathcal T$ through a single real parameter $p \in [0,1]$.
The simulation (c) of \cref{fig:PleinDeSimus} illustrates this behavior, and the corresponding formal statement regarding this class may be found in \cref{sec:Exemple_branching}.
 \end{itemize}
Unlike the $X$-permuton, the Brownian separable permuton already appeared
in our previous works \cite{Nous2,Nous1} as a universal limit 
of substitution-closed permutation classes.
The second item above shows that the universality class
of the Brownian separable permuton
extends further than the substitution-closed classes.
The first item reveals another (new) universality class,
with a simple limiting object: the $X$-permuton.

As the readers will have noticed,
the simulation (d) of \cref{fig:PleinDeSimus} does not fit in any of the two above situations.
The reason is that the dependency graph of the underlying specification restricted to the critical families
is {\em not} strongly connected. 

Our main results (\cref{Th:linearCase} and \cref{Th:branchingCase}) do not apply to
the not strongly connected case.
However, in \cref{Sec:CouteauSuisse}, we describe a strategy
to reduce the study of such cases to the strongly connected one.
This strategy applies in particular to the class in the simulation (d) above,
and the limit in this case is a juxtaposition of two $X$-permutons of random relative sizes.
This statement is proved in \cref{ssec:compound}.

\subsection{Relation with our previous works}

The present paper is the third article in the line that we started with \cite{Nous1}. We first obtained the asymptotic behavior of separable permutations (separable permutations form the iconic class of the branching case). In \cite{Nous2} we proposed a first extension towards substitution-closed classes. We identified three distinct asymptotic behaviors according to some technical conditions (H1), (H2) and (H3) related to the generating function of the family of simple permutations in the class (we refer to \cite{Nous2} for precise statements). 

We propose in the present paper another extension,
namely to permutation classes with a finite specification.
This contains the case of substitution-closed classes, as we will see p.\pageref{eq:SpecifClosedClasses}.
We restrict ourselves to specifications satisfying an analytic condition -- that we denote (AR) --,
which informally says that the equations appearing in this system are all analytic at the radius of convergence.
In the case of substitution-closed classes, this is equivalent to condition (H1).

\begin{center}
\includegraphics[width=0.80\linewidth]{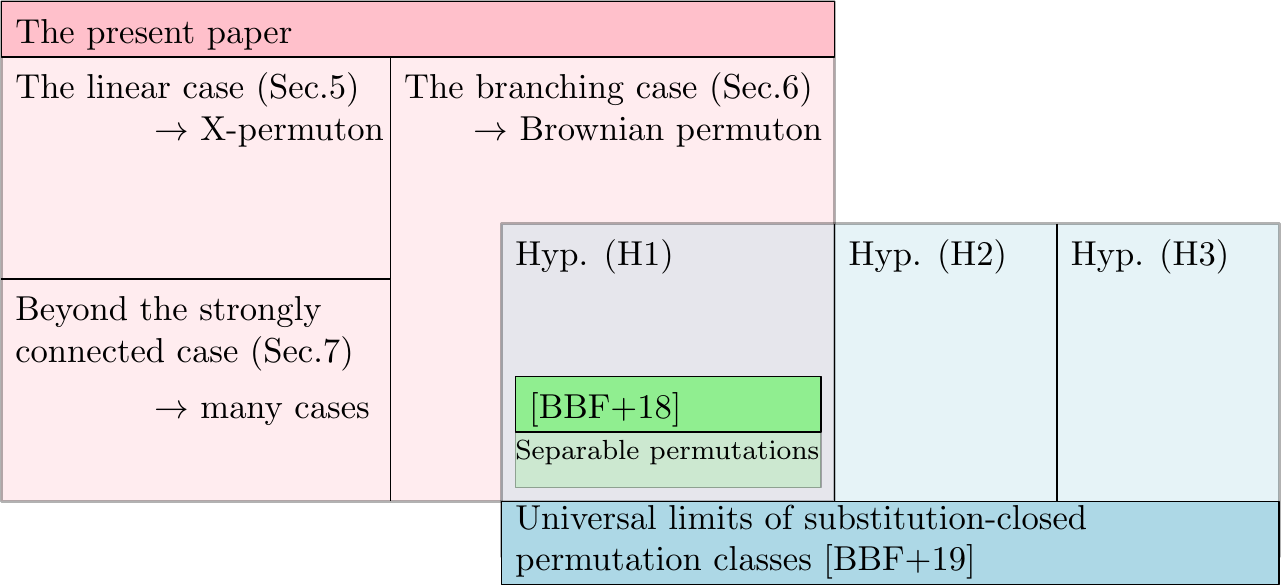}
\end{center}

\subsection{Proof tools: analytic combinatorics of algebraic systems}
\label{ssec:ToolsIntro}
Our main results are convergence results of random permutations in some class $\mathcal T$
in the topology of permutons.
A general result relates such convergence to the convergence, for each $k\ge 1$,
of the substructure, \emph{i.e.} the {\em pattern}, induced by $k$ random elements of the permutation.
The latter can be proved by enumerating, for each $\pi$, the family $\mathcal T_\pi$ of permutations
in $\mathcal T$ with $k$ marked elements inducing the pattern $\pi$.
It turns out that the combinatorial specification for $\mathcal T$ can be refined
to a combinatorial specification for $\mathcal T_\pi$.

We analyze the resulting specifications with tools of analytic combinatorics.
Namely, we classically translate combinatorial specifications into systems of equations
for the associated generating series.
When the equations are analytic on a sufficiently large domain
and when the dependency graph of the system is strongly connected,
two different kinds of behavior might happen:
\begin{itemize}
  \item either the system is linear, and the series 
    have all polar singularities at their radius of convergence \cite{BanderierDrmota};
  \item or the system is called {\em branching}, and the series have all square-root singularities
    (this is known as Drmota-Lalley-Woods theorem in the literature \cite{Violet,Drmota}).
\end{itemize}
We need however to adapt the hypotheses of these theorems to our setting, and more importantly,
to make explicit the coefficients in the first-order asymptotic expansion of the series;
this is done in \cref{sec:complex_analysis}.

We will apply these theorems to the critical series in our (refined) tree-specifications,
considering the non-critical series as parameters.
Once we know the singular behavior of the series,
the transfer theorem of analytic combinatorics \cite{Violet} 
gives us the asymptotic number of elements in $\mathcal T$ and $\mathcal T_\pi$ for all $\pi$.
We deduce from this the probability that $k$ marked elements in a uniform permutation in $\mathcal T$
induce a given pattern $\pi$.
Comparing these probabilities to those in the candidate limiting permutons,
this proves the desired convergence.

In \cref{Subsec:outline_proof}, we present a precise outline of the proof.

\subsection{Probabilistic lens on the linear/branching dichotomy}

Before going into the details of our results, we briefly shed a probabilistic light on the linear/branching dichotomy. The specification of $\TTT$ gives a natural encoding of a random  $t\in\TTT$ as a random multitype tree, whose types are given by $\TTT_0,\TTT_1,\dots,\TTT_d$.

For multitype Galton-Watson trees the research efforts have been mostly concentrated on the case where the matrix of types is \emph{irreducible} (see \cite{GregoryMultitype,Stephenson}), this corresponds in our setting to the subcase where the whole dependency graph is strongly connected.
Under this hypothesis, the linear case is trivial: the tree is just a line and the theory boils down to the analysis of finite irreducible Markov chains.
In the branching case, the behavior is well-understood too: it is shown in \cite{GregoryMultitype} that (critical, finite-variance) multitype Galton–Watson trees counted by their number of nodes converge after rescaling to Aldous's Brownian Continuum Random Tree (CRT).

Without the irreducibility condition, there is no treatment in the literature:
in full generality many different cases could happen.
For instance this may be illustrated by \emph{triangular P\'{o}lya urns} \cite{JansonTriangular},
which model two-type reducible branching processes.

In our setting, where the dependency graph {\em restricted to critical series} is assumed to be strongly connected,
here is what we expect.
The tree contains a subtree starting at the root formed by nodes of critical types,
on which fringe subtrees with nodes of subcritical types (called bushes below) are grafted.
We expect the critical part to be of linear size, while bushes are all of size $O(1)$.
In the essentially linear case, the tree should therefore look like
a long line to which small bushes are grafted,          
while in the essentially branching case we have a tree close to the Brownian CRT.    
This dichotomy is confirmed by simulations, see \cref{fig:SimusArbres}.
This explains why we get in one case a deterministic permuton, 
and in the other case a Brownian object.
\begin{figure}[bhtp]
	\begin{tabular}{cc}
		\includegraphics[height = 7cm]{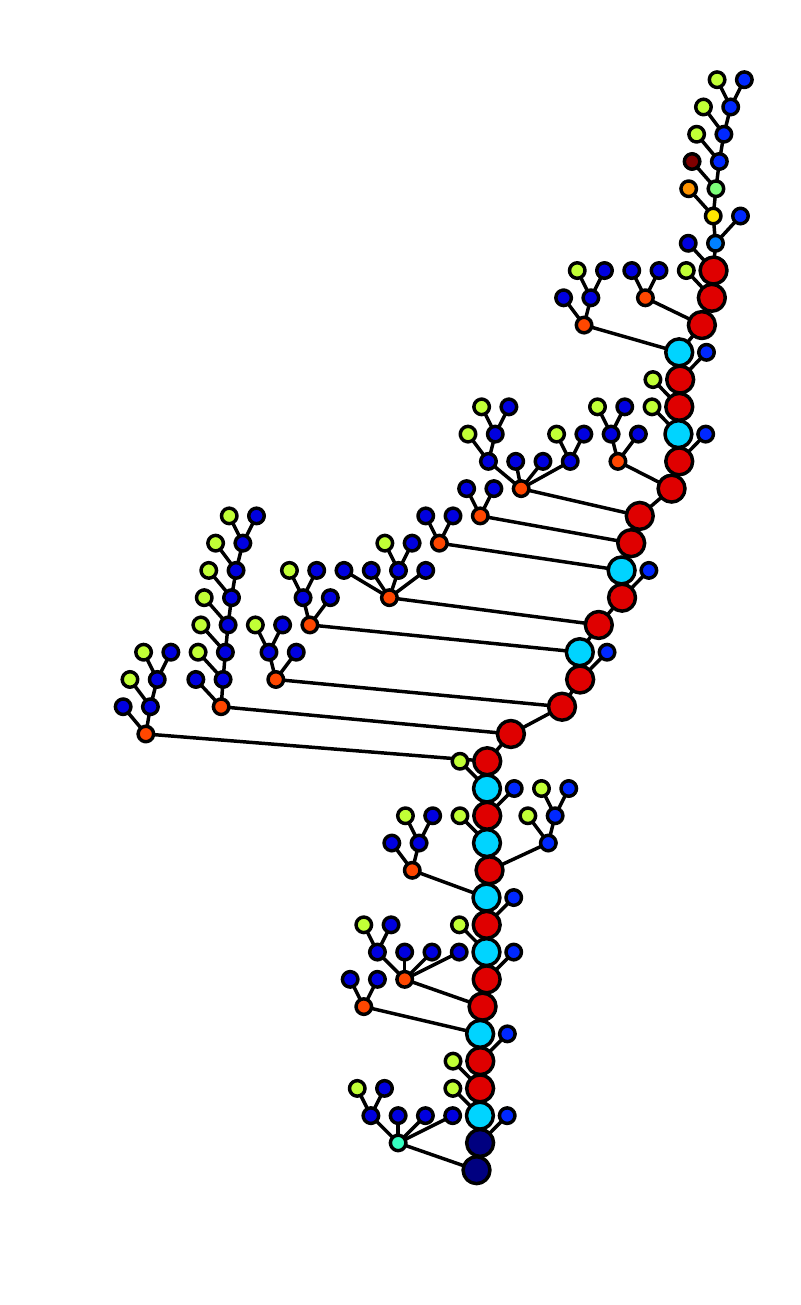}& \includegraphics[height = 7cm]{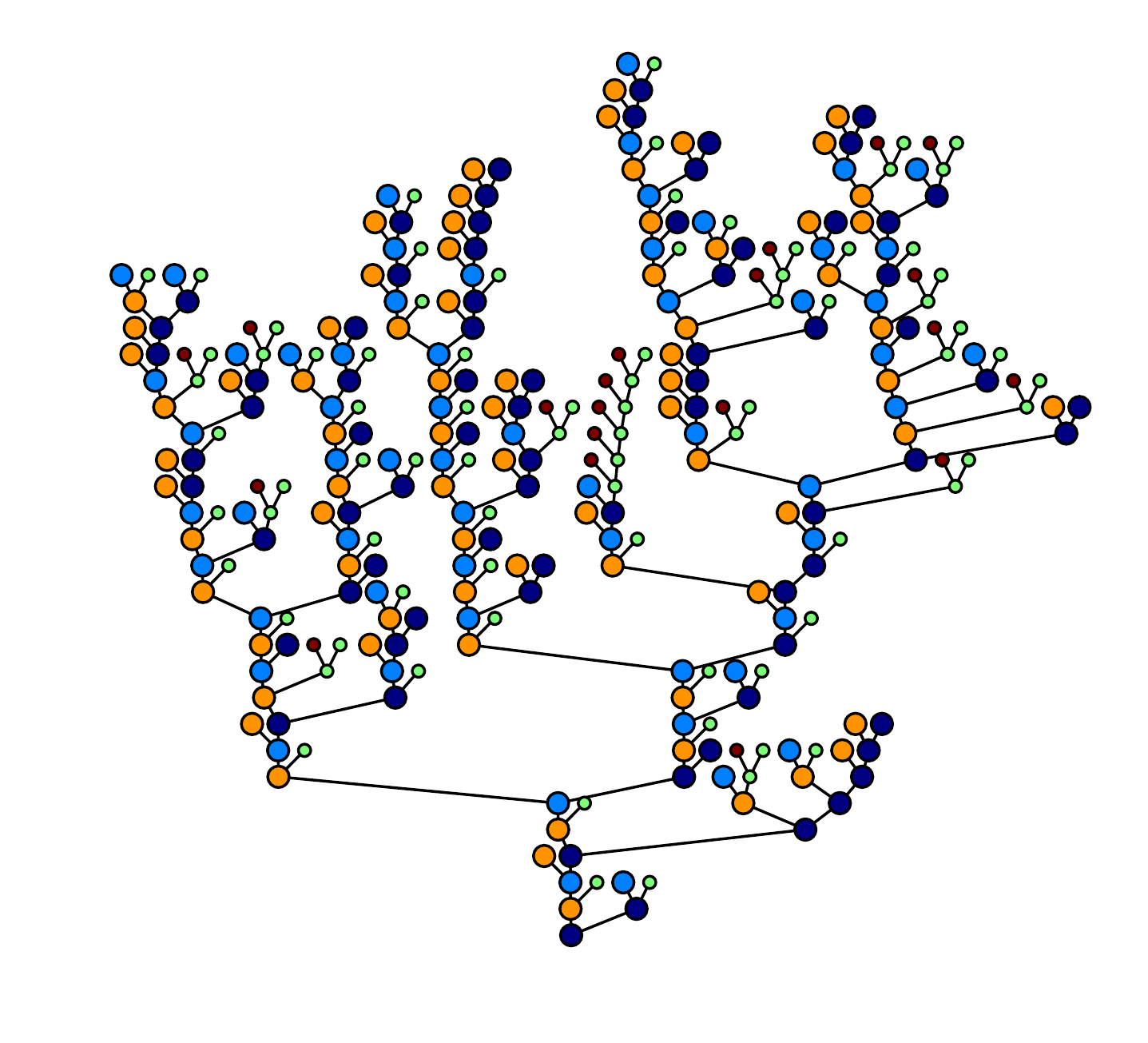} 
	\end{tabular}
	\caption{Substitution trees of uniform random permutations in finitely specified classes. Vertices are colored according to their type in the specification, and critical types have a bigger marker. 
	Left: the essentially linear case (for the class $\Av(2413,1243,2341,41352,531642)$, see \cref{sec:ClasseV}). Right: the essentially branching case (for the class $\Av(132)$, see \cref{sec:Av(132)_debut,sec:Av(132)_suite}). 
		\label{fig:SimusArbres}}
\end{figure}

It might be possible to follow this intuition to prove our results:
first proving convergence results for the (decorated) trees,
and then showing continuity properties of the tree-to-permutation map
to deduce the convergence of the associated permutations.
This raises however many difficulties, like defining a good topology for decorated trees
and proving convergence results for reducible multitype trees in this new topology.
Therefore we have preferred to work directly on permutations,
with combinatorial methods, as explained in \cref{ssec:ToolsIntro}.

We finally mention that the recent paper~\cite{WithBenedikt}, 
which reproves and strengthens the Brownian separable permuton limit result for substitution-closed classes of~\cite{Nous2}, 
uses the above approach of proving convergence results on trees, and then translating them to permutations. 
The approach of~\cite{WithBenedikt} relies on the following fact:
in the context of substitution-closed classes,
thanks to a further encoding, the trees representing 
permutations are distributed as conditioned \emph{monotype} Galton-Watson trees, 
which are better understood than their multitype analogues.
This reduction to monotype trees does not seem to extend 
to the general context of classes with a finite specification studied in the present paper.

\subsection{Simulations and examples}
To apply our results to a specific permutation class,
a finite specification needs to be computed and analyzed,
to identify under which case it falls down,
check the relevant hypotheses,
and compute the parameters of the limiting permuton if applicable.

For classes with a finite number of simple permutations,
we provide an implementation
of the algorithm of \cite{BBPPR} for the computer algebra system Sage. 
This implementation is available on-line \cite{Program}.
It allows to compute the specification of a given class, and to deduce a system of equations for the series enumerating the various families in the specification. 
It can also output a Boltzmann sampler of the class and run it.
Simulations in \cref{fig:PleinDeSimus} were obtained this way.

The next step to apply our results is to identify the critical series.
Unfortunately, as far as we are aware of, there is no automatic way to perform this step.
When the system of equations we obtain is solvable,
it is usually easy to see from the analytic formulas for the generating functions
which are the critical families.
It is also sometimes possible to identify them even in non-solvable cases,
using the dependency graph of the system and estimates on the growth rates of the various families;
see \cref{lem:MonotonieDesRho} for the relationship between critical series and dependency graphs
and \cref{sec:Exemple_branching_annexe} for an example of the identification of critical series
in a non-solvable case.

Once critical series have been identified, the following conditions need to be checked
\begin{enumerate}
\item whether the dependency graph restricted to these critical series
is strongly connected;
\item an aperiodicity condition;
\item whether the system is essentially linear or essentially branching.
\end{enumerate}
This is usually straightforward from definitions.
When items i) and ii) above are fulfilled, our results apply and the limiting permuton is either an $X$-permuton
or a biased Brownian separable permuton, depending on item iii) above.
One still needs to compute the parameter(s).
To this end, the program \cite{Program} contains some useful functions,
in particular evaluating the matrices and eigenvectors appearing in formula 
\eqref{eq:DefProbaCaterpillar} p.\pageref{eq:DefProbaCaterpillar}.

Most of the examples given in this paper were treated this way. 
For each of them, an accompanying Jupyter notebook is provided\footnote{All
available from this address: \url{http://mmaazoun.perso.math.cnrs.fr/pcfs/}}.

\subsection{Outline of the paper}
\label{Subsec:outline_paper}
\begin{itemize}
\item We present in \cref{sec:framework} (\textsc{Our framework}) the combinatorial specifications of permutation classes (where permutations are represented by their \emph{standard trees} -- see \cref{defintro:StandardTree}), and the terminology essentially linear/essentially branching case.
\item In \cref{sec:OurResults} we give our main results: \cref{Th:linearCase} and \cref{Th:branchingCase}. We provide several applications to particular permutation classes.
\item In \cref{Sec:TreeToolbox} (\textsc{Tree Toolbox}, which is useful for both the essentially linear and the essentially branching case), we gather useful definitions and properties regarding the families of trees induced by our combinatorial decompositions. In particular we define in \cref{def:CriticalSpine} the \emph{critical subtree}  $\Crit_i(t)$ of a standard tree in $\TTT_i$. Critical subtrees play an important role in the analysis.
\item In \cref{Sec:ProofsLinear} (\textsc{The Essentially Linear Case}), we do the analysis which leads to the proof of \cref{Th:linearCase}. As the limiting object is in this case the $X$-permuton, we also state and prove in \cref{SsSec:MarginalesXPermutons} some of its properties. 
\item \cref{sec:branching} (\textsc{The essentially branching case}) is devoted to the proof of \cref{Th:branchingCase}. 
\item Our main theorems are stated under Hypothesis (SC), ensuring that $G^\star$ (the dependency graph of the underlying specification restricted to the critical families) is strongly connected. We explain in \cref{Sec:CouteauSuisse} (\textsc{Beyond the strongly connected case}) how to apply \cref{Th:linearCase,Th:branchingCase} in several situations where the graph $G^\star$ is not strongly connected.
\begin{itemize}
\item In \cref{Sec:SubTree} we give sufficient conditions under which there typically exists a \emph{giant} component in a standard tree in $\TTT_0$. It follows that we obtain the same limiting objects as in \cref{Sec:ProofsLinear,sec:branching}.
\item In \cref{ssec:SeveralSubstructures} we show that several macroscopic substructures can appear in a typical large tree of $\TTT_0$. 
  In that case, the limiting object is an assembling of Brownian separable permutons, or $X$-permutons, depending on the case. 
\end{itemize}
\item \cref{sec:complex_analysis} is a \emph{complex analysis toolbox}. 
We analyze, near their dominant singularity,
solutions $\mathbf Y = (Y_1,\ldots Y_c)$ of systems of equations of the form
$$
\mathbf Y(z) = \mathbf \Phi(z,\mathbf Y(z)), 
$$
where $\mathbf \Phi(z,\mathbf y) = (\Phi_1(z,\mathbf y),\ldots ,\Phi_c(z,\mathbf y))$ 
is a vector of multivariate power series of $(z,\mathbf y)$ with nonnegative integer coefficients. 
This is a standard problem in analytic combinatorics
(see, {\em e.g.}, \cite{Drmota97,Drmota,Violet,BanderierDrmota})
but we need variants or more precise/general versions of the statements
we could find in the literature.
These results could be useful independently of the present article.
\item In \cref{Sec:AppendixExamples} we work out several examples of specifications and their analysis.
  In particular we discuss the computational details. 
\end{itemize}

\section{Our framework} 
\label{sec:framework}

The starting point of our analysis of a permutation class is a \emph{(combinatorial) specification} for this class. 
We collect here the necessary definitions to set the framework of our study, 
and recall results from the literature that yield specifications of permutation classes. 
The results we obtain (presented in Section~\ref{sec:OurResults}) depend on the type of the specification we have, 
and we also present these different types of specifications in this section. 

\subsection{Permutations, patterns, and classes}
\label{ssec:basics_patterns}

For any positive integer $n$, the set of permutations of $[n]:= \{1,2,\ldots, n\}$ is denoted by $\Sn_n$. 
We write permutations of $\Sn_n$ in one-line notation as $\sigma = \sigma(1) \sigma(2) \dots \sigma(n)$. 
For a permutation $\sigma$ in $\Sn_n$, the \emph{size} $n$ of $\sigma$ is denoted by $|\sigma|$. 
We often view a permutation $\sigma$ of size $n$ as its \emph{diagram}: 
it is (up to rescaling) the set of points of coordinates $(i,\sigma(i))_{1\leq i\leq n}$ in the Cartesian plane. 

For $\sigma\in \Sn_n$, and $\SsEnsemble\subset [n]$ of cardinality $k$, let $\pat_\SsEnsemble(\sigma)$ be the permutation of $\Sn_k$ induced by $\{\sigma(i) : i\in \SsEnsemble\}$. 
For example for $\sigma=65831247$ and $\SsEnsemble=\{2,5,7\}$ we have
$$
\pat_{\{2,5,7\}}\left(6\mathbf{5}83\mathbf{1}2\mathbf{4}7\right)=312
$$
since the values in the subsequence $\sigma(2) \sigma(5) \sigma(7)=514$
are in the same relative order as in the permutation $312$. 
A permutation $\pi = \pat_\SsEnsemble(\sigma)$ is a \emph{pattern} involved (or contained) in $\sigma$, 
and the subsequence $(\sigma(i))_{i \in \SsEnsemble}$ is an \emph{occurrence} of $\pi$ in $\sigma$. 
When a pattern $\pi$ has no occurrence in $\sigma$, we say that $\sigma$ \emph{avoids} $\pi$. 
The pattern containment relation defines a partial order on $\mathfrak{S} = \cup_n \mathfrak{S}_n$: 
we write $\pi \preccurlyeq \sigma$ if $\pi$ is a pattern of $\sigma$.

\medskip

A \emph{permutation class}, $\CCC$, is a subset of $\mathfrak{S}$ which is downward closed under $\preccurlyeq$. 
Namely, for every $\sigma \in \CCC$, and every $\pi  \preccurlyeq \sigma$, it holds that $\pi \in \CCC$. 
It is known (see for example \cite[Paragraph 5.1.2]{BonaBook}) that permutation classes may equivalently be defined as subsets of $\mathfrak{S}$ characterized by the avoidance of a (finite or infinite) family of patterns. 
For every class $\CCC$, there is a unique such family, $B$, consisting of elements incomparable for $\preccurlyeq$. It is called the \emph{basis} of $\CCC$, 
and we write $\CCC = \Av(B)$.

\subsection{Substitution of permutations and encoding by trees}\label{sec:OperatorsPermutations} 

We now define formally the notion of substitution, already presented in the introduction
\begin{definition}%
Let $\theta=\theta(1)\cdots \theta(d)$ be a permutation of size $d$, and let $\pi^{(1)},\dots,\pi^{(d)}$ be $d$ other permutations. 
The \emph{substitution} of $\pi^{(1)},\dots,\pi^{(d)}$ in $\theta$ is the permutation of size $|\pi^{(1)}|+ \dots +|\pi^{(d)}|$ 
obtained by replacing each $\theta(i)$ by a sequence of integers isomorphic to $\pi^{(i)}$ while keeping the relative order induced by $\theta$ between these subsequences.\\
This permutation is denoted by $\theta[\pi^{(1)},\dots,\pi^{(d)}]$. 
\end{definition}
Examples of substitution are conveniently presented representing permutations by their diagrams
(see \cref{fig:sum_and_skew} below, or \cref{fig:Subst_Intro} in the introduction).
\begin{figure}[htbp]
    \begin{center}
      \includegraphics[width=7cm]{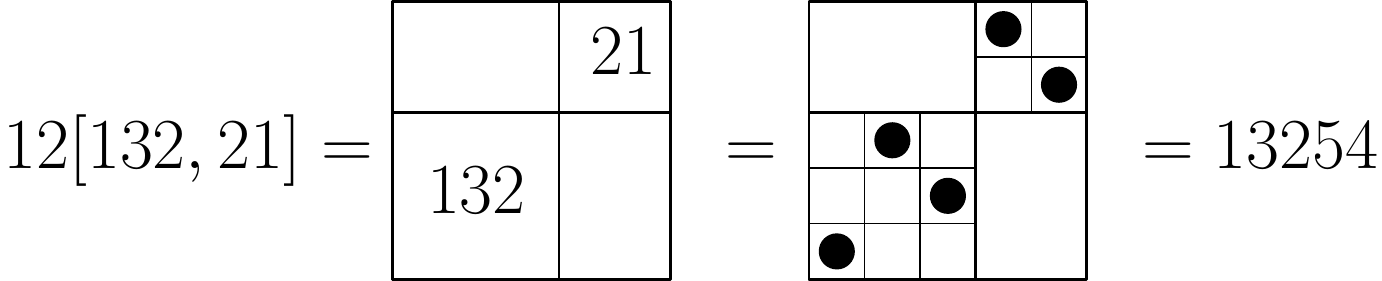} \qquad 
      \includegraphics[width=7cm]{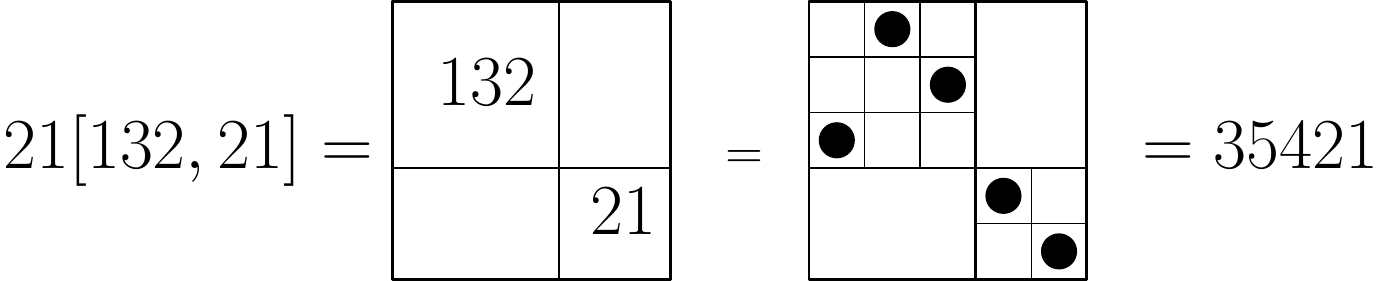}\\
    \end{center}
\caption{Substitution of permutations.\label{fig:sum_and_skew}}
\end{figure}

It will be interesting to consider nested substitutions,
starting from permutations of size $1$.
The corresponding succession of operations is then encoded by a tree,
called {\em substitution tree}.
\begin{definition}\label{def:SubstitutionTree}
A \emph{substitution tree} of size $n$ is a rooted plane tree with $n$ leaves,
where any internal node with $k \ge 2$ children is labeled by a permutation of size $k$.
Internal nodes with only one child are forbidden. 
The labels $12$ (resp. $21$) of internal nodes are often replaced by $\oplus$ (resp. $\ominus$).
\end{definition}

Given any tree $t$, we denote by $\Internal{t}$ the set of internal nodes of $t$ 
and by $\Leaves{t}$ the set of leaves of $t$. 
Also, given a tree $t$ and a node $v$ in $t$, 
we call \emph{fringe subtree} of $t$ rooted at $v$ the subtree of $t$ 
whose nodes are $v$ and all its descendants. 

\begin{definition}\label{Def:PermTree}
Let $t$ be a substitution tree. We define inductively the permutation $\perm(t)$ associated with $t$:
\begin{itemize}
	\item if $t$ is just a leaf, then $\perm(t)=1$;
	\item if the root of $t$ has $r\geq 2$ children with corresponding fringe subtrees $t_1,\ldots,t_r$ (from left to right), and is labeled with the permutation $\theta$, then $\perm(t)$ is the permutation obtained as the substitution of $\perm(t_1),\dots,\perm(t_r)$ in $\theta$:
	\[\perm(t) = \theta[\perm(t_1),\ldots,\perm(t_r)].\]
\end{itemize}
\end{definition}

\begin{figure}[htbp]
    \begin{center}
      \includegraphics[width=10cm]{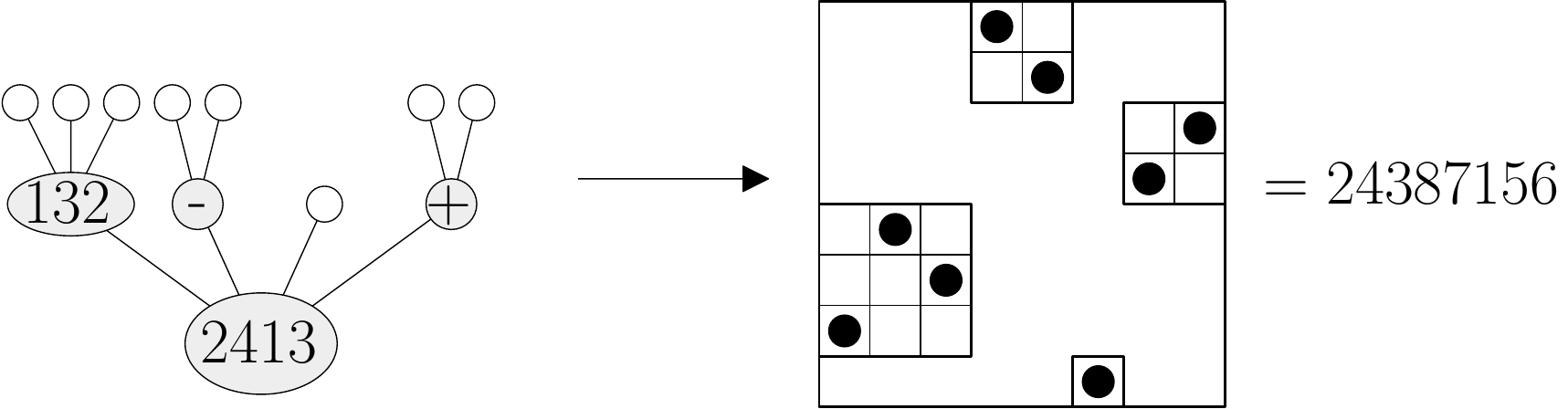}
    \end{center}
    \caption{A substitution tree encoding a permutation.}
    \label{fig:ExemplePermutationAssocieeAUnArbre}
\end{figure}

\cref{fig:ExemplePermutationAssocieeAUnArbre} illustrates this construction. 
When $\perm(t)=\sigma$, we say that $t$ is a tree that \emph{encodes} $\sigma$, or a tree \emph{associated with} $\sigma$. By construction, any tree associated with $\sigma$ has exactly $|\sigma|$ leaves.

In general, permutations may be encoded by several substitution trees.
In what follows, we recall how to exhibit a particular substitution tree associated with each permutation $\sigma$.
To this end, we need the notion of simple permutations.

\begin{definition}%
A {\em simple permutation} is a permutation $\sigma$ of size $n > 2$
that does not map any nontrivial interval 
(\emph{i.e.} a range in $[n]$ containing at least two and at most $n-1$ elements) onto an interval.
\end{definition}
For example, $451326$ is not simple as it maps $[3;5]$ onto $[1;3]$. 
The smallest simple permutations are $2413$ and $3142$ (there is no simple permutation of size $3$). 
We can now define the notion of standard trees.
\begin{definition}\label{defintro:StandardTree}
A \emph{standard tree}  is a substitution tree in which internal nodes satisfy the following constraints:
\begin{itemize}
\item Internal nodes are labeled by $\oplus$ (representing $12$), $\ominus$ (representing $21$), or by a simple permutation.
\item Every node labeled  by $\oplus,\ominus$ has degree\footnote{Throughout the paper, by \emph{degree} of a node in a tree, we mean the number of its children (which is sometimes called arity or out-degree in other works). 
Note that it is different from the graph-degree: for us, the edge to the parent (if it exists) is not counted in the degree.} two.
The left-child of a node labeled by  $\oplus$ (resp. $\ominus$) cannot be labeled by  $\oplus$ (resp. $\ominus$).
\item A node labeled by a simple permutation $\alpha$ has degree $|\alpha|$.
\end{itemize}
\end{definition}

The following proposition is an easy consequence of \cite[Proposition 2]{AA05}.
\begin{proposition}\label{prop:AA05_as_bijection}
  The mapping $\perm$ of \cref{Def:PermTree}
  defines a bijection from standard trees to permutations
  that maps the number of leaves of the tree to
  the size of the permutation.
\end{proposition}

From now on, we identify a permutation $\sigma$ and its associated standard tree.

\begin{remark}
  [regarding the terminology]
  In most papers in the literature, simple permutations may have size $2$ or more.
  With this definition, $12$ and $21$ are both simple permutations.
  In the context of substitution trees, they however play a different role than other simple permutations.
  This explains why we take another convention here.

  The standard trees that we consider here are a variant of the canonical trees 
  considered in \cite{Nous2}; in the latter,
  nodes labeled by $\ominus$ (resp. $\oplus$) can be of any degree (representing respectively permutations $12\dots k$ and $k\dots 21$ for any $k \geq 2$)
  but none of their children may have a label $\ominus$ (resp. $\oplus$).
  Going from one to the other is straightforward.
\end{remark}

\subsection{Combinatorial specifications for families of permutations}

The starting point of our study of a permutation class $\CCC$ is a \emph{combinatorial specification} for $\CCC$, 
or rather for the family of standard trees of permutations of $\CCC$. 
The specifications we will consider involve not only permutation classes, but also more general families of permutations (see \cref{dfn:classesGeneralized}), 
and we may as well consider specifications for these more general families. 
We identify any such family of permutations with the family of corresponding standard trees, $\TTT$. 
For any such $\TTT$, we denote by $\SSS_{\TTT}$ the set of simple permutations in $\TTT$. 
Throughout this article we will only consider families of permutations 
with a particular type of specification, called a \emph{tree-specification}, which we now define. 

\begin{definition}[Tree-specifications]\ \\
  \label{dfn:tree_specification}
Let $\TTT_0,\dots,\TTT_d$ be $d+1$ families of permutations.
A \emph{tree-specification} of $(\TTT_0,\dots,\TTT_d)$ is a system of combinatorial equations
\begin{equation}
\TTT_i\ =\ \eps_i\{\bullet \} \ \uplus \ \biguplus_{\pi\in\,\SSS_{\TTT_i} \uplus\set{\oplus,\ominus}} \ \biguplus_{(k_1,\dots,k_{|\pi|})\in K_\pi^i}
\pi[\TTT_{k_1},\dots,\TTT_{k_{|\pi|}}]
\qquad (0\leq i\leq d)
\tag{$\EEE_\TTT$}
\label{eq:Specif}
\end{equation}
where the symbol $\uplus$ denotes disjoint union,
$\bullet$ is the permutation of size $1$ and 
for every $i\leq d$, $\eps_i\in\{0,1\}$ (so that $\eps_i\{\bullet \}$ is either $\emptyset$ or $\{\bullet \}$) and $K_\pi^i$ is a subset of $\{0,\dots,d\}^{|\pi|}$. 
\end{definition}

Note that we extended the notation for substitution to sets of permutations in the obvious way: 
$\pi[\TTT_{k_1},\dots,\TTT_{k_{|\pi|}}]$ is the set of permutations $\pi[\theta^{(1)}, \dots, \theta^{(|\pi|)}]$
where for each $i$, $\theta^{(i)} \in \TTT_{k_i}$.

In order to avoid trivial cases, in this article we consider only tree-specifications such that every family $\TTT_i$ is nonempty,
at least one family $\TTT_i$ is infinite and at least one $\eps_i$ is nonzero.

\begin{definition}
Given a permutation class $\CCC$, a \emph{specification} for $\CCC$ is a tree-specification as above such that $\TTT_0$ is (the set of standard trees of) $\CCC$. 
\end{definition}

We present some cases where it is known that a specification for $\CCC$ exists.

\subsubsection*{The substitution-closed case.}

\begin{definition}%
A permutation class $\CCC$ is \emph{substitution-closed} if, for every $\theta, \pi^{(1)},\dots,\pi^{(d)}$ in $\CCC$, the substitution
$\theta[\pi^{(1)},\dots,\pi^{(d)}]$ also belongs to $\CCC$.
\end{definition}

A characterization of substitution-closed classes which is very convenient in some of our examples in the following, proved in~\cite[Proposition 1]{AA05}: 
a permutation class is substitution-closed if and only if
its basis contains only simple permutations. 

A specification for a substitution-closed class $\CCC$ (assuming that $\CCC$ contains $12$ and $21$) 
is easily obtained from \cite[Proposition 2]{AA05}, 
which we rephrased as \cref{prop:AA05_as_bijection} above. 
Indeed, in this case, $\CCC = \TTT$ is simply the set of standard trees such that all nodes carry labels
from $\SSS_\TTT \uplus \set{\oplus,\ominus}$. 
Then, denoting $\TTT^{\nonp}$ (resp. $\TTT^{\nonm}$) the subset of these standard trees 
whose root is {\em not} labeled by $\oplus$ (resp. $\ominus$), 
we have the tree-specification
\begin{equation}\label{eq:SpecifClosedClasses}
\begin{cases}
\quad\TTT &= \{\bullet\} \  \biguplus \ 
\oplus[\TTT^\nonp,\TTT]
\  \biguplus \ 
\ominus[\TTT^\nonm,\TTT]
\  \biguplus \ 
\bigg(\biguplus_{\pi \in \SSS_\TTT} 
\pi[\TTT, \dots, \TTT]
\bigg)\vspace{2mm}
\\
\ \TTT^{\nonp}  &= \{\bullet\} \  \biguplus \ 
\ominus[\TTT^\nonm,\TTT]
\  \biguplus \ 
\bigg(\biguplus_{\pi \in \SSS_\TTT} 
\pi[\TTT, \dots, \TTT]
\bigg)\vspace{2mm}
\\
\ \TTT^{\nonm}  &= \{\bullet\} \  \biguplus \ 
\oplus[\TTT^\nonp,\TTT]
\  \biguplus \ 
\bigg(\biguplus_{\pi \in \SSS_\TTT} 
\pi[\TTT, \dots, \TTT]
\bigg).
\end{cases}
\end{equation}

As already mentioned in the Introduction we proved \cite{Nous2} that under a mild sufficient condition the limiting permuton of a substitution-closed class is a biased Brownian separable permuton.

\subsubsection*{The general case.}

Assume now that $\CCC$ is a permutation class (still assumed to contain $12$ and $21$) which is not substitution-closed. 
Finding a specification for $\CCC=\TTT$ can be more complicated 
since $\TTT$ is only a subset of the standard trees with node labels
in $\SSS_\TTT \uplus \set{\oplus,\ominus}$.
Using the representation of permutations as standard trees, one can prove however that, when $\SSS_\TTT$ is finite, 
a tree-specification for $\TTT$ always exists -- see~\cite{BHV,BBPPR}. 
The main result of \cite{BBPPR} is that such a specification can be obtained algorithmically, given the basis $B$ of $\TTT$. Note that~\cite{AA05} ensures that $B$ is necessarily finite, since $\SSS_\TTT$ is finite.

In the resulting specification, the families $(\TTT_i)_{0\leq i\leq d}$ are sets of permutations 
defined by avoidance and containment of patterns and restrictions on the root label.
We introduce notation for such classes. 

\begin{definition}
\label{dfn:classesGeneralized}
For any set $\TTT$ of permutations (most often, a permutation class), 
for any sets of patterns $\{\sigma_1,\dots ,\sigma_k\}$ and $\{\tau_1,\dots ,\tau_\ell\}$, 
and, optionally, for any $\delta \in \{ \oplus, \ominus \}$, 
we define $\TTT^{\mathrm{not} \delta}_{\langle \sigma_1,\dots ,\sigma_k\rangle,( \tau_1,\dots ,\tau_\ell)}$ to be the subset of $\TTT$ such that
\begin{itemize}
\item the patterns  $ \sigma_1,\dots ,\sigma_k$ are excluded from every permutation,
\item the patterns $\tau_1,\dots ,\tau_\ell$  have to occur in every permutation,
\item the superscript $\mathrm{not} \delta$ (for $\delta = \oplus, \ominus$) is optional and indicates that permutations in this family are \emph{$\delta$-indecomposable} permutations,
\emph{i.e.} that the root of associated standard trees are not labeled with $\delta$.
\end{itemize}
\end{definition}

We will assume throughout the paper that we are given a tree-specification of $(\TTT_0,\dots, \TTT_d)$.
We will see later a few examples of specifications such that $\TTT_0$ is a permutation class. 

\subsection{System of equations, critical series, and dependency graph}

The specification \eqref{eq:Specif} of \cref{dfn:tree_specification} induces a system of $d+1$ equations for the generating functions $T_i$ of $\TTT_i$, of the form

\begin{equation}
\begin{cases}
T_0(z)&=\eps_0z+ F_0(T_0,T_1,\dots,T_d)\\
T_1(z)&=\eps_1z+F_1(T_0,T_1,\dots,T_d)\\
 & \dots \\
T_d(z)&=\eps_dz+F_d(T_0,T_1,\dots,T_d),
\end{cases}
\tag{$E_T$}
\label{eq:SystemeSeries}
\end{equation}
where $F_0,\dots,F_d$ are $d+1$ multivariate formal power series with nonnegative integer coefficients (whose variables are denoted $y_0, \dots, y_d$). The valuation of each $F_i$ with respect to the $(y_j)$'s all together is  greater than or equal to $2$.
Moreover, the solutions of this system can be computed recursively:  $(T_0,T_1,\dots,T_d)$ is the unique solution of \eqref{eq:SystemeSeries} in which all the $T_i$'s are power series with nonnegative integer coefficients and without constant term (by convention there is no permutation of size $0$).

Note that $F_i$ is a polynomial when the set of simple permutations of $\TTT_i$ is finite. 
\smallskip

For $0\leq i\leq d$, let $\rho_i \in[0,+\infty]$ be the radius of convergence of $T_i$. We set $\rho = \min_i\{\rho_i\}$.

\begin{definition}
The family $\TTT_i$ and its generating series $T_i$ are said \emph{critical} if $\rho_i=\rho$.
On the contrary, we say that $\TTT_i$ and $T_i$ are \emph{subcritical} if $\rho_i>\rho$.
\end{definition}

Denote by $I^\star \subseteq [ 0,d ]$ the set of indices of critical series. By abuse of notation we say that $i$ is critical if $i\in I^\star$. We can assume that $I^\star$ is of the form  $[ 0,c ]$.
In the case of a specification for a permutation class $\mathcal{C}$ obtained by the algorithm of \cite{BBPPR},
$\mathcal{C}$ is always critical. That is why we focus on critical families.

It is convenient to consider the dependency graph $G_{\eqref{eq:Specif}}$ of the specification \eqref{eq:Specif}. 
As we shall see with \cref{lem:MonotonieDesRho}, this graph will help us 
identify the critical series.
Informally, $G_{\eqref{eq:Specif}}$ contains an edge from $\TTT_j$ to $\TTT_i$
when $\TTT_i$ depends on $\TTT_j$.
\begin{definition}
  \label{def:DependencyGraph}
The \emph{dependency graph} $G_{\eqref{eq:Specif}}$ is  the directed graph with $d+1$ vertices labeled by $\TTT_0,\TTT_1,\dots,\TTT_d$, and whose edges are $\TTT_j\to \TTT_i$ for every $i,j$ such that $\TTT_j$ appears in the equation of $\eqref{eq:Specif}$ whose left-hand side is $\TTT_i$.
\end{definition}

Since we are interested in critical families, we also assume without loss of generality that 
for each subcritical family there is a directed path in $G_{\eqref{eq:Specif}}$ from that vertex to a critical family.
Indeed, we can simply remove the other subcritical families.

The dependency graph $G_{\eqref{eq:Specif}}$ of the specification can be used to identify critical families $\TTT_i$.

\begin{lemma}
\label{lem:MonotonieDesRho}
If there is an edge $\TTT_j\to \TTT_i$ in the dependency graph $G_{\eqref{eq:Specif}}$, then $\rho_i \leq \rho_j$.
Consequently, if $\TTT_j$ is critical and if there is an edge $\TTT_j\to \TTT_i$, then $\TTT_i$ is critical.
\end{lemma}
\begin{proof}
As the series $F_i$ involved in the system \eqref{eq:SystemeSeries} have nonnegative coefficients, 
the radius of convergence of the left hand side $T_i$ of an equation of \eqref{eq:SystemeSeries} 
is smaller than the radius of convergence of any $T_j$ 
appearing in the right hand side of the equation defining $T_i$ in \eqref{eq:SystemeSeries}.
\end{proof}

Not only criticality, but also aperiodicity (which will appear in the hypotheses of our main theorems), follows along the edges of the graph.
\begin{definition}
A series $A(z)=\sum_{n\geq 0}a_nz^n$ is said \emph{periodic} if there exist integers $r\in \mathbb{Z}_{\geq 0}$, $d\geq 2$ such that
$$
\{n,\ a_n\neq 0\} \subset r+d\mathbb{Z}_{\geq 0}.
$$ 
On the contrary, $A$ is \emph{aperiodic} if it is not periodic.
\end{definition}

\begin{lemma}
\label{lem:MonotonieAperiodicite}
If $\TTT_j$ is aperiodic and there is an edge $\TTT_j\to \TTT_i$ in the dependency graph of $G_{\eqref{eq:Specif}}$, then $\TTT_i$ is aperiodic.
\end{lemma}
\begin{proof}
Assume that there is an edge $\TTT_j\to \TTT_i$. As the series  $F_i$ in the system \eqref{eq:SystemeSeries} has nonnegative integer coefficients, this implies, up to a constant shift, a term-by-term domination of $T_j$ by $T_i$. Hence $T_i$ is aperiodic.
\end{proof}

In order to separate difficulties, we will often make the following strong assumption. 
Let $G^\star$ denote the subgraph  of $G_{\eqref{eq:Specif}}$ consisting of all critical families $\TTT_i$ .
\begin{hypothesis}[SC]
  \label{Hyp:StronglyConnected}
We assume that $G^\star$ is strongly connected.
\end{hypothesis}
In \cref{Sec:CouteauSuisse} we will see how to combine our results
in each strongly connected component in order to relax Hypothesis (SC).

\subsection{Essentially linear and essentially branching specifications}\label{Sec:DefLinearBranching}

In the following, we adopt some notational convention to guide the reading.
As above, curly letters (like $\mathcal{T}$) and capital letters (like $T$) denote respectively 
combinatorial families and their generating series. 
Moreover, vectors of generating series are denoted by bold letters (like $\mathbf{T}$) 
and matrices of such by thick letters (like $\mathbb{M}$). 
The superscript $\star$ indicates a restriction to critical families or critical series.

\begin{definition}
	The specification \eqref{eq:Specif} is \emph{essentially branching} if there exist $i,j,j'\in I^\star$ such that the equation defining $\TTT_i$ in \eqref{eq:Specif} involves a term of the form 
$\pi[\dots, \TTT_{j}, \dots, \TTT_{j'}, \dots]$.
	It is \emph{essentially linear} otherwise. 
	\label{Def:EssBranching}
\end{definition}
Equivalently, the specification is essentially branching  when there exist $i,j,j'\in I^\star$ such that  
$\frac {\partial F_i}{\partial y_j\partial y_{j'}} \neq 0$.

Denote by $\mathbf{T}^\star=(T_i)_{i\in I^\star}$ the vector of critical series.
We consider the restriction of the system \eqref{eq:SystemeSeries} to critical series and regard subcritical series as parameters:
\begin{equation}
	\mathbf T^\star(z) = \mathbf \Phi(z,\mathbf T^\star(z)), 
	\label{eq:generic_system}
\end{equation}
where $\mathbf \Phi(z,\mathbf y) = (\Phi_1(z,\mathbf y),\ldots ,\Phi_c(z,\mathbf y))$ is a vector of multivariate power series of $(z,\mathbf y)$ with nonnegative integer coefficients:
for all $i \in I^\star$, $\Phi_i\big(z, (y_j)_{j\in I^\star} \big) = \eps_i z + F_i\big((y_j)_{j\in I^\star}, (T_\ell(z))_{\ell \notin I^\star}\big)$.

\smallskip

In the essentially linear case, this system is linear and can be written as
\begin{equation}
\mathbf T^\star(z) = \mathbb M^\star(z) \, \mathbf T^\star(z) + \mathbf V^\star(z)
\label{eq:LinearSystem}
\end{equation}
where the entries of $\mathbb M^\star(z)$ and $\mathbf V^\star(z)$ involve only the variable $z$ and subcritical series. 

More precisely, for $i,j\in I^\star$, $\left(\mathbf V^\star(z)\right)_i = F_i(0, \ldots, 0, (T_\ell(z))_{\ell \notin I^\star})$, and
$\left(\mathbb M^\star(z)\right)_{i,j}$ is the coefficient of $T_j(z)$ in $F_i(T_0(z),\dots,T_d(z))$, so we can write 
\begin{equation}\label{eq:MDeriveesSecondes}
\mathbb M^\star(z)=\left(\frac{\partial F_i (y_0,\ldots ,y_d)}{\partial y_j}\bigg|_{(T_0(z),\dots,T_d(z))}\right)_{i,j\in I^\star}.
\end{equation}
Since the specification is essentially linear, in the substitution of $y_i$'s with $T_i$'s in \cref{eq:MDeriveesSecondes}, only subcritical series $T_i$'s are effectively substituted.
The analysis of such systems will be discussed in \cref{Sec:ProofsLinear}.

\smallskip

In the essentially branching case, the analysis of the restricted system relies on \cref{Thm:DLW},
a variant of the Drmota-Lalley-Wood theorem \cite[Thm. VII.6, p. 489]{Violet}. This analysis involves  the Jacobian matrix
\begin{equation}\label{eq:Def_Mstar}
\mathbb M^\star \big(z,(y_k)_{k\in I^\star}\big)  = \left(\frac{\partial F_i (y_0,\ldots ,y_d)}{\partial {y_j}} \bigg|_{\big((y_k)_{k\in I^\star} ,(T_\ell(z))_{\ell \notin I^\star}\big)} \right)_{i,j\in I^\star}.
\end{equation}
We observe that the definition of $\mathbb M^\star$ in \cref{eq:Def_Mstar} is consistent with \cref{eq:MDeriveesSecondes}. 
Indeed, in the linear case, $\mathbb M^\star$ does not depend on the $(y_k)'s$ for $k \in I^\star$, and therefore we keep only the first argument,~$z$.

\smallskip

In the essentially linear case, we will use the following assumption whose first item deals with the coefficients of $\mathbf V^\star$ and the second one with the coefficients of $\mathbb M^\star$.
\begin{hypothesis}[RC]
   \label{Hyp:RC}
  We assume that the following conditions are both satisfied
  \begin{enumerate}
  \item For all $i \in I^\star$, $F_i(0, \ldots, 0, (T_\ell(z))_{\ell \notin I^\star})$ has a radius of convergence strictly larger than $\rho$.
  \item For all $i,j \in I^\star$, $\frac{\partial F_i(y_0,\ldots ,y_d)}{\partial y_j} \Big|_{(T_0(z),\dots,T_d(z))}$ 
    has a radius of convergence strictly larger than $\rho$.
  \end{enumerate}
\end{hypothesis}

In the essentially branching case, we need the following assumption.
\begin{hypothesis}[AR]
   \label{Hyp:AR}
  We assume that for all $i \in I^\star$, $\Phi_i\big(z, (y_j)_{j\in I^\star} \big) = \eps_i z +  F_i\big((y_j)_{j\in I^\star}, (T_\ell(z))_{\ell \notin I^\star}\big)$ is analytic around $\big(\rho,(T_j(\rho))_{j \in I^\star}\big)$.
\end{hypothesis}

\begin{observation} \label{obs:polynomial}
When there is a finite number of simple permutations in the $(\TTT_i)$'s, then the $(F_i)$'s are polynomials and Hypotheses (RC) and (AR) are satisfied.
\end{observation}

\subsection{Examples of tree-specifications}

To illustrate the definitions seen so far, we present a few examples of tree-specifications obtained with the algorithm of \cite{BBPPR}. 
We will return to these examples at later stages of our analysis. 

\subsubsection{The case of substitution-closed classes}
Consider a substitution-closed class $\TTT$.
We introduce the generating series $S(u)=\sum_{\alpha \in\SSS_\TTT} u^{|\alpha|}$ of the set $\SSS_\TTT$ of simple permutations in $\TTT$. 
Recall that the tree-specification \eqref{eq:Specif} is given by \cref{eq:SpecifClosedClasses} p.\pageref{eq:SpecifClosedClasses}.
The associated system \eqref{eq:SystemeSeries} is then given by
\[
\begin{cases}
T&=z+T^{\nonp} T + T^{\nonm} T + S(T)\\
T^{\nonp}&=z+ T^{\nonm} T+S(T)\\
T^{\nonm}&=z+T^{\nonp} T +S(T).
\end{cases}
\]
The dependency graph (represented on the left of \cref{Fig:DependencyGraph}) is strongly connected.
Thanks to \cref{lem:MonotonieDesRho}, this ensures that the three series are critical.
It follows that the specification is essentially branching (although a very special case of such). 
Indeed, a product of two critical series appears in the equation for a critical series 
(\emph{e.g.} the product $T^{\nonp} T$ in the equation defining $T$). 
\begin{figure}[htbp]
\begin{center}
\begin{tabular}{c | c}
 \includegraphics[width=6cm]{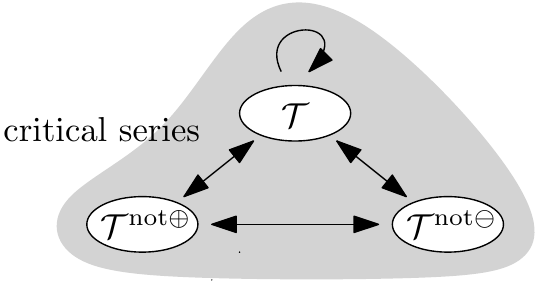} &
 \includegraphics[width=8cm]{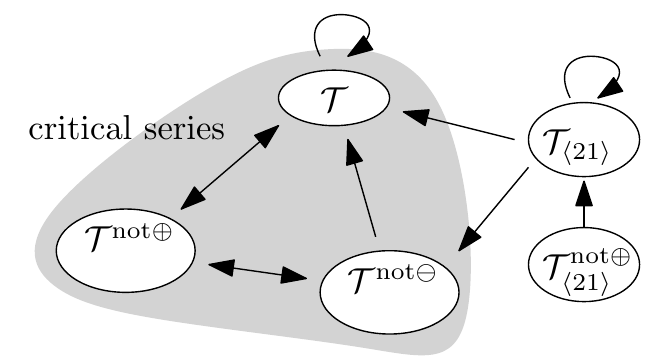}
\end{tabular}
\end{center}
\caption{Left:  The dependency graph in the case of a substitution-closed class. Right: The dependency graph for the specification \eqref{eq:SystemeAv132} of $\Av(132)$.}
\label{Fig:DependencyGraph}
\end{figure}

\subsubsection{An example of class having an essentially branching specification: $\Av(132)$} \label{sec:Av(132)_debut}
We consider  $\TTT= \Av(132)$, which is not substitution-closed, as $132$ is not simple. 
One can check that there is no simple permutation in $\TTT$. The algorithm of \cite{BBPPR} gives the following specification
\footnote{See the \href{http://mmaazoun.perso.math.cnrs.fr/pcfs/} {companion Jupyter notebook}  \texttt{examples/Av132.ipynb}}:

\begin{equation}
\begin{cases}
\quad\TTT &= \{\bullet\} \quad  \biguplus \quad
\oplus[\TTT^\nonp, \TTT_{\langle 21 \rangle}]
\quad \biguplus \quad
\ominus[\TTT^\nonm,\TTT] \\
\quad\TTT^\nonp &= \{\bullet\} \quad  \biguplus \quad
\ominus[\TTT^\nonm,\TTT] \\
\quad\TTT^{\nonm}  &= \{\bullet\} \quad \biguplus \quad
\oplus[\TTT^\nonp, \TTT_{\langle 21 \rangle}] \\
\quad\TTT_{\langle 21 \rangle}  &= \{\bullet\} \quad \biguplus \quad
\oplus[\TTT_{\langle 21 \rangle}^\nonp, \TTT_{\langle 21 \rangle}]\\
\quad\TTT_{\langle 21 \rangle}^\nonp  &= \{\bullet\}.
\end{cases}
\label{eq:SpecifAv132}
\end{equation}
Translating into series and then solving the system, we get
\begin{equation}
\begin{cases}
T&=z+T^{\nonp} T_{\langle 21\rangle}
+ T^{\nonm} T\\
T^{\nonp}&=z+ T^{\nonm} T\\
T^{\nonm}&=z+T^{\nonp} T_{\langle 21\rangle}\\
T_{\langle 21\rangle}&=z+T_{\langle 21 \rangle}^\nonp T_{\langle 21\rangle}\\
T_{\langle 21 \rangle}^\nonp  &= z.
\label{eq:SystemeAv132}
\end{cases}\qquad \qquad
\begin{cases}
T&=\frac{1-\sqrt{1-4z}}{2z}-1\\
T^{\nonp}&= \frac{1-\sqrt{1-4z}}{2}+z\\
T^{\nonm}&=(1-z)\frac{1-\sqrt{1-4z}}{2z}\\
T_{\langle 21\rangle}&=\frac{z}{1-z}\\
T_{\langle 21 \rangle}^\nonp  &= z.
\end{cases}
\end{equation}
In this case, 
the critical series are $T,T^{\nonp},T^{\nonm}$ with common radius of convergence $\rho=1/4$.
Since the product $T^{\nonm} T$ appears in the equation for $T$ in system \eqref{eq:SystemeAv132}, 
it follows that the specification \eqref{eq:SpecifAv132} is essentially branching.
Moreover, the restriction $G^\star$ of the dependency graph to critical series (see \cref{Fig:DependencyGraph}, right) is strongly connected.

\subsubsection{An example of class having an essentially linear specification: the $X$-class} \label{sec:ClasseX_debut}
We consider next the class $\TTT_0= \Av(2413,3142, 2143,3412)$, 
which is known as the $X$-class~\cite{elizalde:the-x-class-and:,waton:on-permutation-:}. 
This class is not substitution-closed and contains no simple permutation.
The algorithm of \cite{BBPPR} gives the following specification\footnote{See the \href{http://mmaazoun.perso.math.cnrs.fr/pcfs/} {companion Jupyter notebook}  \texttt{examples/X.ipynb}}:
\begin{equation}
\begin{cases}
\begin{array}{rl}
\mathcal T_{0}= &\{ \bullet \} \uplus \oplus[\mathcal T_{1},\mathcal T_{2}]\uplus \oplus[\mathcal T_{1},\mathcal T_{3}]\uplus \oplus[\mathcal T_{4},\mathcal T_{2}]\uplus \ominus[\mathcal T_{1},\mathcal T_{5}]\uplus \ominus[\mathcal T_{1},\mathcal T_{6}]\uplus \ominus[\mathcal T_{7},\mathcal T_{5}]\\
\mathcal T_{1}= &\{ \bullet \} \\
\mathcal T_{2}= &\{ \bullet \} \uplus \oplus[\mathcal T_{1},\mathcal T_{2}]\\
\mathcal T_{3}= &\oplus[\mathcal T_{1},\mathcal T_{3}]\uplus \oplus[\mathcal T_{4},\mathcal T_{2}]\uplus \ominus[\mathcal T_{1},\mathcal T_{5}]\uplus \ominus[\mathcal T_{1},\mathcal T_{6}]\uplus \ominus[\mathcal T_{7},\mathcal T_{5}]\\
\mathcal T_{4}= &\ominus[\mathcal T_{1},\mathcal T_{5}]\uplus \ominus[\mathcal T_{1},\mathcal T_{6}]\uplus \ominus[\mathcal T_{7},\mathcal T_{5}]\\
\mathcal T_{5}= &\{ \bullet \} \uplus \ominus[\mathcal T_{1},\mathcal T_{5}]\\
\mathcal T_{6}= &\oplus[\mathcal T_{1},\mathcal T_{2}]\uplus \oplus[\mathcal T_{1},\mathcal T_{3}]\uplus \oplus[\mathcal T_{4},\mathcal T_{2}]\uplus \ominus[\mathcal T_{1},\mathcal T_{6}]\uplus \ominus[\mathcal T_{7},\mathcal T_{5}]\\
\mathcal T_{7}= &\oplus[\mathcal T_{1},\mathcal T_{2}]\uplus \oplus[\mathcal T_{1},\mathcal T_{3}]\uplus \oplus[\mathcal T_{4},\mathcal T_{2}].
\end{array}
\end{cases}
\label{eq:SpecifClasseX}
\end{equation}
For the sake of readability, 
when examples become more complicated as above, 
we simply denote the families of trees occurring in the specification  by $(\mathcal T_i)_i$.  

The specification~\eqref{eq:SpecifClasseX} translates into a system on the series $(T_i)_{0 \leq i \leq 7}$, whose resolution gives 
\[
\begin{cases}
\begin{array}{rl}
 T_{0}= & \tfrac{-z(2z - 1)}{(2z^2 - 4z + 1)}\\
 T_{1}= &z \\
 T_{2}=T_{5}= &\tfrac{-z}{(z - 1)}\\
 T_{3}=T_{6}= &\tfrac{-z^2}{(z - 1)(2z^2 - 4z + 1)}\\
 T_{4}= T_{7}= &\tfrac{z^2(-z + 1)}{(2z^2 - 4z + 1)}\\
\end{array}
\end{cases}
\]

The factor $2z^2 - 4z + 1$ in the denominator determines the criticality here, and 
the critical series (of radius of convergence $\rho= 1-\sqrt{2}/2 \approx 0.2929 $) are $T_0, T_3, T_4, T_6$ and $ T_7$. 
It is easy to observe that, for any of these critical series $T_i$, 
in the analogue of system~\eqref{eq:SpecifClasseX} on series, the equation defining $T_i$ 
only contains terms involving at most one critical series (\emph{i.e.} no product of such). 
It follows that the specification~\eqref{eq:SpecifClasseX} is essentially linear, 
and that the associated dependency graph restricted to the critical $\mathcal T_i$ has two strongly connected components (see \cref{fig:DependendyGraphX}).

\begin{figure}[htbp]
\includegraphics[width=5cm]{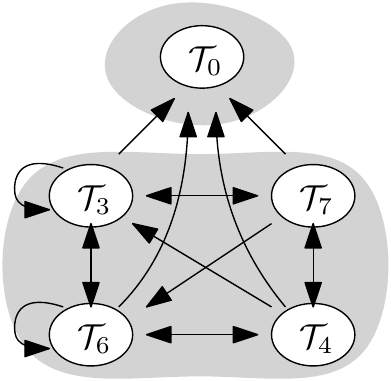}
\caption{The subgraph $G^\star$ restricted to critical families $\mathcal T_i$, for the specification~\eqref{eq:SpecifClasseX} of the class $\Av(2413,3142, 2143,3412)$. In this case, $G^\star$ has two strongly connected components $\{\TTT_0\}$ and $\{\TTT_3,\TTT_4,\TTT_6,\TTT_7\}$.}
\label{fig:DependendyGraphX}
\end{figure}

\begin{remark}
In the above examples, the dependency graph restricted to critical families, $G^\star$,  is very simple: 
  either it is strongly connected, or it has two strongly connected components, one of which consists of $\TTT_0$ alone. 
  To see an example with a much more complicated structure, we refer the reader to \cref{ssec:couteau_suisse_union},
  where $G^\star$ has nine strongly connected components.
\end{remark}

\section{Our results}\label{sec:OurResults}

In order to state our results, we first recall the formal definition of permutons, which are the convenient framework to describe scaling limits of permutations, 
as well as some properties of permutons.

\subsection{Permutons and limits of permutations}\label{ssec:permutons_intro}

Permutons were first considered by Presutti and Stromquist in \cite{PresuttiStromquist} under the name of \textit{normalized measures}. 
Permutations of all sizes are special cases of permutons, and weak convergence of measures allows to define convergent sequences of permutations.
Presutti and Stromquist realized that convergence in the space of permutons implies convergence of pattern densities, and that permutons allow to define natural models of random permutations.
The theory was developed independently by Hoppen, Kohayakawa, Moreira, Rath and Sampaio in \cite{Permutons}. Their main result is the equivalence between convergence to a permuton and convergence of all pattern densities.  The terminology  {\em permuton} was given afterwards by Glebov, Grzesik, Klimo\v{s}ov\'a and Kr{\'a}l \cite{FinitelyForcible}, by analogy with graphons.
The theory of (random) permutons will here allow us to state scaling limit results for sequences of (random) permutations.

Formally, a permuton is a probability measure on the unit square $[0,1]^2$ with both its marginals uniform. 
Permutons generalize permutation diagrams in the following sense:
to every permutation $\sigma\in \Sn_n$, we associate the permuton $\mu_\sigma$ with density 
\[
\mu_\sigma(dxdy) = n\One_{\sigma(\lceil xn \rceil) = \lceil yn \rceil} dxdy.
\]
Note that it amounts to replacing every point $(i,\sigma(i))$ in the diagram of $\sigma$ (normalized to the unit square) 
by a square of the form $[(i-1)/n,i/n]\times [(\sigma(i)-1)/n,\sigma(i)/n]$, which has mass $1/n$ uniformly distributed. 

The space $\mathcal M$ of permutons is equipped with the topology of weak convergence of measures, which makes it a compact metric space (for more details on weak convergence of measures, we refer to \cite{Billingsley}).
This allows to define convergent sequences of permutations: we say that $(\sigma_n)_n$ converges to a permuton $\mu$ when $(\mu_{\sigma_n}) \to \mu$ weakly. 
Similarly, one can define convergence in distribution of random permutations to a random permuton: we say that a random sequence of permutations $\bm \sigma_n$ converges in distribution to a random permuton $\bm \mu$ if $\mu_{\bm \sigma_n} \xrightarrow[n\to\infty]{(d)} \bm \mu$ in the space of permutons.

We now define the permutations induced by a (possibly random) permuton $\bm \mu$. Conditionally on $\bm \mu$, take a sequence of $k$ random points $(\XX,\YY) = ((\xx_1,\yy_1),\dots, (\xx_k,\yy_k))$ in $[0,1]^2$, 
independently with common distribution $\bm \mu$. 
Because $\bm \mu$ has uniform marginals and the $\xx_i$'s (resp. $\yy_i$'s) are independent, 
it holds that the $\xx_i$'s (resp. $\yy_i$'s) are almost surely pairwise distinct.
We denote by $(\xx_{(1)},\yy_{(1)}),\dots, (\xx_{(k)},\yy_{(k)})$ the $x$-ordered sample of $(\XX,\YY)$, \emph{i.e.} the unique reordering of the sequence $((\xx_1,\yy_1),\dots, (\xx_k,\yy_k))$ such that
$\xx_{(1)}<\cdots<\xx_{(k)}$.
Then the values $(\yy_{(1)},\dots,\yy_{(k)})$ are in the same
relative order as the values of a unique permutation, that we denote $\InducedPerm_k(\bm \mu)$.
Since the points are taken at random, $\InducedPerm_k(\bm \mu)$ is a random permutation of size $k$.

In \cite{Nous2} we proved the following criterion which is a stochastic generalization of the one given in \cite{Permutons}.

\begin{theorem}\label{thm:randompermutonthm}
	For any $n$, let $\bm\sigma_n$ be a random permutation of size $n$. 
	Moreover, for any fixed $k$, let ${\bm I}_{n,k}$ be a uniform random subset of $[n]$ with $k$ elements, independent of $\bm \sigma_n$.
	The following assertions are equivalent.
	\begin{enumerate}%
		\item [(a)] $(\mu_{\bm \sigma_n})_n$ converges in distribution for the weak topology to some random permuton $\bm \mu$. 
		\item [(b)]For every $k$, the sequence  $\big(\!\pat_{{\bm I}_{n,k}}(\bm\sigma_n)\big)_n$ of random permutations converges in distribution to some random permutation $\bm \rho_k$.
	\end{enumerate}
	If either condition is satisfied, we have 
 \begin{equation}
   \bm \rho_k \stackrel{(d)}= \InducedPerm_k(\bm \mu) ,\text{ for every } k\geq 1
   \label{eq:LimiteRhoPermuton}
\end{equation}
and the relations \eqref{eq:LimiteRhoPermuton} characterize the distribution of $\bm \mu$ as a random permuton.
\end{theorem}
Thanks to criterion (b), convergence in distribution of permutons may be reduced to combinatorial enumeration.

\subsection{Our results: The essentially linear case}

We introduce the necessary material to state our first main theorem (which will be proved in Section~\ref{Sec:ProofsLinear}). 

\begin{definition}\label{Def:XPermuton}
Let $\mathbf p = (p_+^{\gauche},p_+^{\droite},p_-^{\gauche},p_-^{\droite})\in[0,1]^4$ be a quadruple with sum $1$.
The $X$-permuton with parameter $\mathbf p$ is the following probability measure on the unit square
\[\mu^{X}_{\mathbf p} = \sum_{ \substack{ e\in \{ \gauche,\droite\},\\  \eps \in \{-,+\} }  } p_\eps^{ e }\ \nu(z_\eps^{ e },(a,b)), \]
where 
\begin{eqnarray*}&z_+^{\gauche}=(0,0),\quad z_-^{\gauche}=(0,1),\quad z_-^{\droite}=(1,0),\quad z_+^{\droite}=(1,1),\\
&a = p_+^{\gauche}+p_-^{\gauche}, \quad b=p_+^{\gauche}+p_-^{\droite},
\end{eqnarray*}
and $\nu(X,Y)$ denotes the normalized one-dimensional Lebesgue measure on the segment $(X,Y)$ in the plane (see \cref{fig:X-permuton}).
\end{definition}
\begin{figure}[htbp]
\begin{center} 
\includegraphics[width=6cm]{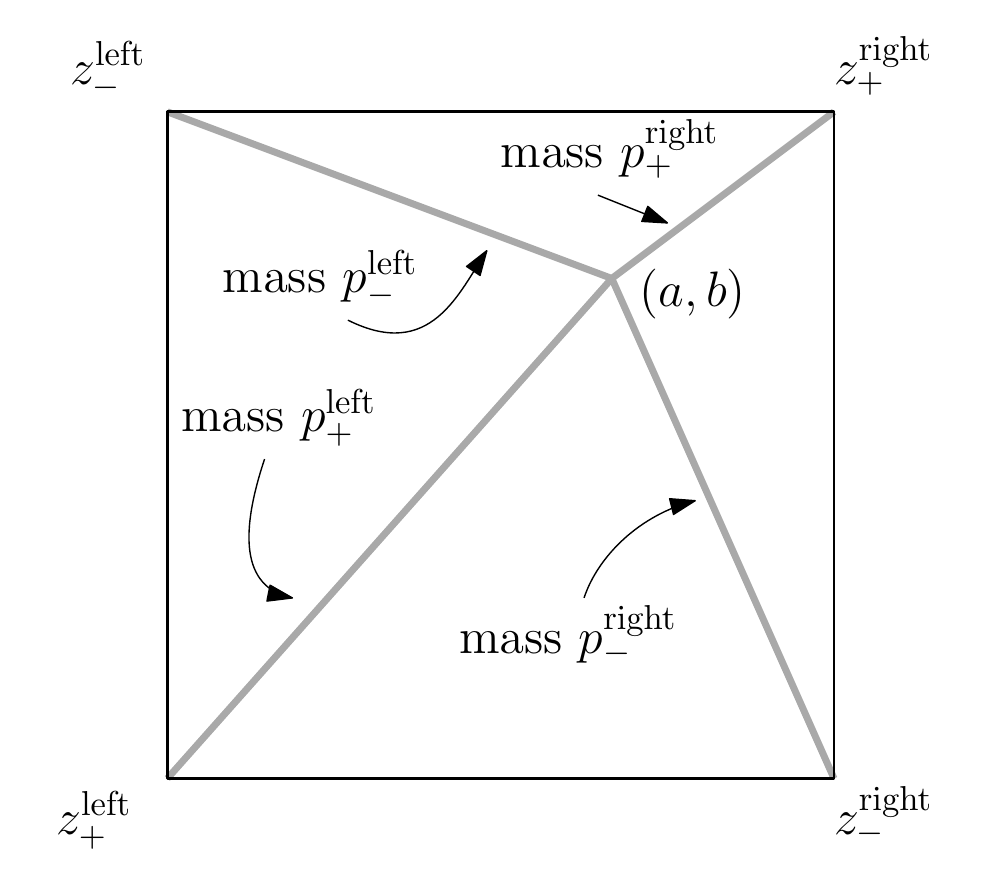}
\hspace{3mm}
\includegraphics[width=6cm]{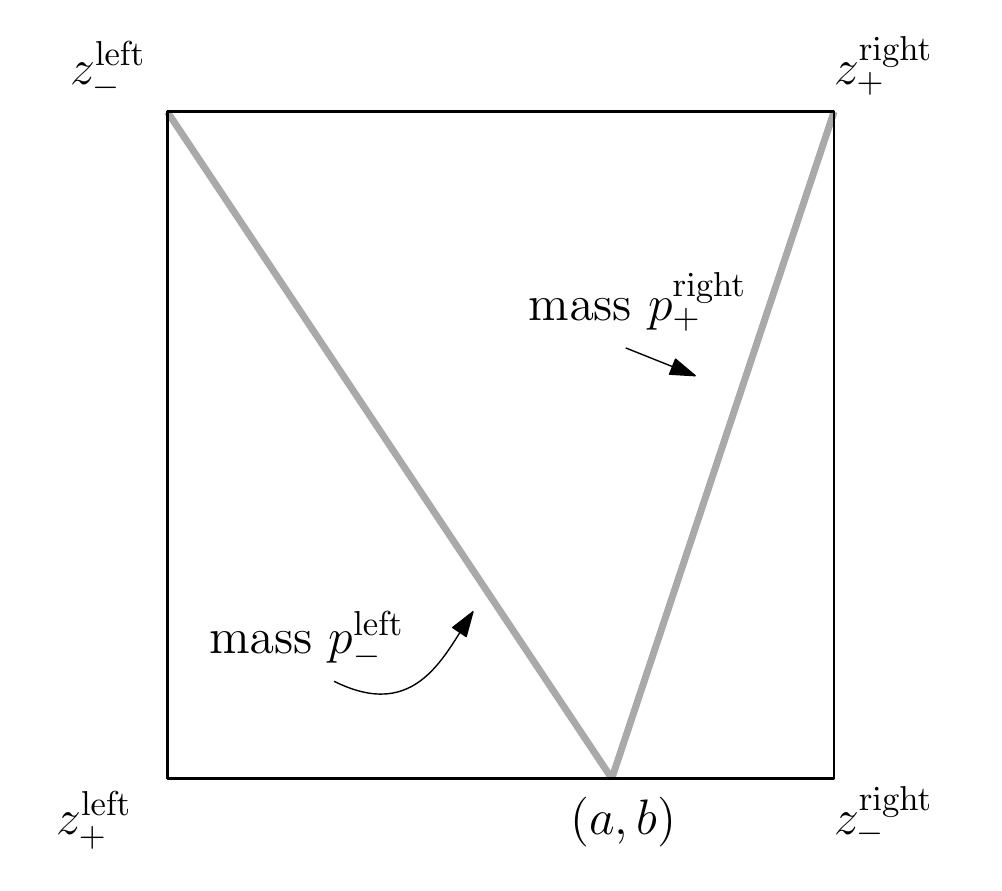}
\end{center}
\caption{The support of the $X$-permuton with parameter $\mathbf p = (p_+^{\gauche},p_+^{\droite},p_-^{\gauche},p_-^{\droite})$,
denoting $a=p_+^{\gauche}+p_-^{\gauche}$ and $b=p_+^{\gauche}+p_-^{\droite}$. Left: The generic case. Right: A degenerate case $b=0$.}
\label{fig:X-permuton}
\end{figure}

Let us verify that the above defined $\mu^{X}_{\mathbf p}$ is indeed a permuton, \emph{i.e.} that its marginals are uniform.
We first observe that $\mu^{X}_{\mathbf p} ([0,a] \times [0,1])=p_+^{\gauche}+p_-^{\gauche}=a$.
By proportionality, for each subinterval $[x_1,x_2]$ of $[0,a]$, we have $\mu^{X}_{\mathbf p} ([x_1,x_2] \times [0,1])=x_2-x_1$.
The same holds for subintervals of $[a,1]$, and hence for any subinterval of $[0,1]$.
This proves that the marginal distribution on the horizontal axis is uniform.
The marginal distribution on the vertical axis is treated similarly.

\begin{theorem}[Main Theorem: the essentially linear case]\label{Th:linearCase}
Consider a tree-specification \eqref{eq:Specif} for $\TTT_0,\dots,\TTT_d$ that verifies Hypothesis (SC) (p.\pageref{Hyp:StronglyConnected}). 
We assume that 
\begin{enumerate}
	\item the specification is essentially linear, 
	\item Hypothesis (RC) (p.\pageref{Hyp:RC}) holds,
	\item there is at least one subcritical series which is aperiodic.
\end{enumerate} 
Then all critical families converge to the same $X$-permuton. More precisely, there exists a parameter ${\bf p}=(p_+^{\gauche},p_+^{\droite},p_-^{\gauche},p_-^{\droite})$ such that for every $i\in I^\star$, letting $\bm \sigma_n$ be a uniform permutation of size $n$ in $\TTT_i$, we have
$$
\mu_{\bm \sigma_n} \stackrel{(d)}{\to} { \mu}^{X}_{\mathbf p}.
$$
Furthermore, ${\bf p}$ can be explicitly computed with \cref{eq:DefProbaCaterpillar} p.\pageref{eq:DefProbaCaterpillar}.
\end{theorem}

\begin{remark}
Recall that Hypothesis (RC) holds in particular if there are only finitely many simple permutations in the $\TTT_i$'s.\\
In item iii), the existence of some subcritical series is necessary for an essentially linear specification. 
Aperiodicity of at least one of them is a weak assumption, and it will be easily checked in all examples of the present paper. 
Indeed, most examples considered are tree-specifications for classes with finitely many simple permutations obtained by the algorithm of~\cite{BBPPR}. 
In such specifications all  $T_i$'s are of the form $T^{\mathrm{not} \delta}_{\langle \sigma_1,\dots ,\sigma_k\rangle,( \tau_1,\dots ,\tau_\ell)}$. 
And it was proved in \cite{DrmotaPierrot} that for such specifications, if $T_i$ is not a polynomial, then it is necessarily aperiodic.
\end{remark}

We now present several examples of classes where \cref{Th:linearCase} applies.
\subsubsection{A centered $X$-permuton: $\mathcal T = \Av(2413,3142, 2143,3412)$}\label{sec:ClasseX}  
We finish here the study of the so-called $X$-class, which we started in \cref{sec:ClasseX_debut}.

The specification of the $X$-class is given by \cref{eq:SpecifClasseX}, p.\pageref{eq:SpecifClasseX}.
We recall that the critical families are $\mathcal T_0$,  $\mathcal T_3$, $\mathcal T_4$, $\mathcal T_6$ and $\mathcal T_7$ and that the specification is essentially linear.
The corresponding dependency graph, already given in \cref{fig:DependendyGraphX}, has two strongly connected components,
one of which being $\TTT_0$ alone.
Removing the equation for $\TTT_0$, 
we obtain a specification for the other families satisfying Hypothesis (SC).
The Hypothesis (RC) holds trivially since we have a polynomial system (\cref{obs:polynomial})
and it is immediate to see that the subcritical series $T_2$ and $T_5$ are aperiodic.
We can therefore apply \cref{Th:linearCase}: there exists a parameter $\mathbf p$
such that a uniform permutation in
any of the class $\mathcal T_3$, $\mathcal T_4$, $\mathcal T_6$ and $\mathcal T_7$
tends towards $\bm \mu^X_{\mathbf p}$.

We now use a little trick to prove that the same holds for $\TTT_0$ as well.
We observe that $\TTT_0 = \TTT_2 \uplus \TTT_3$ and $\TTT_2$ is the set of increasing permutations. Hence when $n$ tends towards $+\infty$, a uniform permutation in $\TTT_0$ belongs to $\TTT_3$ with probability tending to one.
Consequently, a uniform random permutation in the $X$-class $\TTT_0$ 
also converges to the $X$-permuton of parameter $\mathbf p$.

Since the $X$-class has all symmetries of the square,
we necessarily have  $p_+^{\gauche}=p_+^{\droite}=p_-^{\gauche}=p_-^{\droite}=1/4$
(we do not need \cref{eq:DefProbaCaterpillar} to compute the parameter $\mathbf p$ in this case).

\subsubsection{A non-centered $X$-permuton: $\mathcal T = \Av(2413,3142, 2143,34512)$} \label{sec:ClasseXTilde} 
This is a variant of the previous example: 
again, this class is not substitution-closed and contains no simple permutation.
This case is handled as the previous one, except for the computation of the parameter $\mathbf p$,
since the symmetry argument does not apply. In \cref{sec:ClasseXTildeAnnexe}, 
we give a specification for $\mathcal T$ and use \cref{Th:linearCase} and \cref{eq:DefProbaCaterpillar} to show that the limit is the permuton $\bm \mu^X_{\mathbf p}$ where
\[ \mathbf p \approx (0.200258808255625,0.200258808255625,0.431332891374616,0.168149492114135)\]
is a quadruplet of algebraic numbers of degree 3. This is illustrated in \cref{Fig:exemple_Xtilde}

\begin{figure}[htbp] 
	\centering
	\includegraphics[width = 0.3\linewidth]{Simu_XTilde.png} \hspace{0.5cm} \includegraphics[width = 0.3\linewidth]{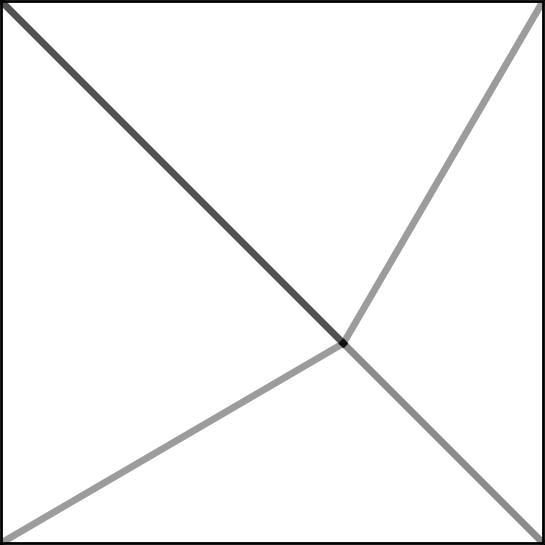}
	\caption{Left: A simulation of a uniform permutation of size 342 in $\Av(2413,3142, 2143,34512)$. Right: The limiting permuton, as predicted by \cref{Th:linearCase}.
		\label{Fig:exemple_Xtilde}}
\end{figure}

\subsubsection{A V shape: $\TTT= \Av(2413,1243,2341,41352,531642)$} \label{sec:ClasseV}

The example we consider next is the one chosen in \cite{BBPPR} 
to illustrate the computation of the specification. 
It is for us a benchmark to test the applicability of our results.

The only simple permutation in the class is $3142$, so that the algorithm of \cite{BBPPR} applies. 
In this case the combinatorial specification gives a system of $13$ equations, 
which we recall in \cref{sec:ClasseVAnnexe}. 
Also in this appendix, we use \cref{Th:linearCase}
to show that the limit is the permuton $\bm \mu^X_{\mathbf p}$
where $p_{\gauche}^+ =  p_{\droite}^-=0$, $p_{\droite}^+ = 1 - p_{\gauche}^-$,
and $p_{\gauche}^- \approx 0.818632668576995$ is the only real root of the polynomial
\[19168 z^{5} - 86256 z^{4} + 155880 z^{3} - 141412 z^{2} + 64394 z - 11773.\]
This is illustrated in \cref{Fig:exemple_V}.

\begin{figure}[htbp] 
	\centering
	\includegraphics[width = 0.3\linewidth]{simu_classeV_n248.pdf} \hspace{0.5cm} \includegraphics[width = 0.3\linewidth]{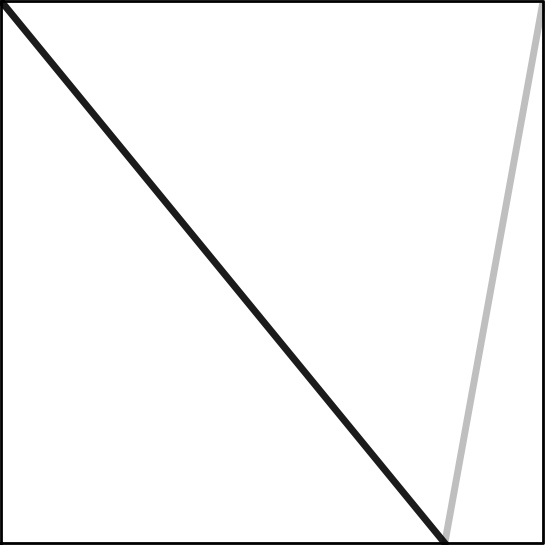}
	\caption{Left: A simulation of a uniform permutation of size 248 in $\Av(2413,1243,2341,41352,531642)$. Right: The limiting permuton, as predicted by \cref{Th:linearCase}.
		\label{Fig:exemple_V}}
\end{figure}

\subsubsection{A diagonal: $\TTT= \Av(231,312)$}\label{ex:layered} 

This is the class of so-called layered permutations. 
It contains no simple permutation and 
admits the following tree-specification:
\begin{equation*}
	\mathcal T_{0}= \{ \bullet \} \uplus \oplus[\mathcal T_{1},\mathcal T_{0}]\uplus \ominus[\mathcal T_{2},\mathcal T_{1}],\quad
	\mathcal T_{1}= \{ \bullet \} \uplus \ominus[\mathcal T_{2},\mathcal T_{1}],\quad
	\mathcal T_{2}= \{ \bullet \}.
\end{equation*}
The associated equations can be solved explicitly and 
$\mathcal T_0$ turns out to be the only critical family. 
So the specification is essentially linear, 
and \cref{Th:linearCase} applies.
We compute the parameters of the limit using \cref{eq:DefProbaCaterpillar}.
Looking at the specification, $D^{\gauche}_{-}=D^{\droite}_{+}=D^{\droite}_{-}=0$, so that the scaling limit for $\Av(231,312)$ is the $X$-permuton with parameters
$$
p^{\gauche}_{+}=1,\quad p^{\gauche}_{-}=p^{\droite}_{+}=p^{\droite}_{-}=0,
$$
\emph{i.e.} the permuton supported by the main diagonal $\{x=y\}$.

This convergence could also be proved easily in a more direct way,
since layered permutations are direct sums of decreasing permutations (\emph{i.e}. $\oplus[d_1,\dots,d_r]$, for decreasing permutations $d_1$, \ldots, $d_r$ of various sizes).
Nevertheless, we briefly commented on this example to illustrate that the diagonal permuton
can appear as a degenerate case of the $X$-permuton.

\subsubsection{An example with infinitely many simple permutations: pin-permutations}
\label{sec:PinPerm}
The class of pin-permutations has been introduced and used in the framework of decision problems
in the papers \cite{Pin1,Pin2}. 
This class contains an infinite number of simple permutations (and has an infinite basis),
so that the algorithm of \cite{BBPPR} does not apply to give a tree-specification.

However, the class was enumerated in \cite[Section 5]{PinPerm}
using a recursive description of their substitution tree.
This recursive description can be translated into a tree-specification.
Note that \cref{obs:polynomial} does not apply and hypothesis (RC) needs to be checked manually.
This is done in \cref{sec:PinPermAnnexe}, where we use \cref{Th:linearCase} to show that
the limiting shape of a uniform random pin-permutation is a centered $X$-permuton.

\subsection{Our results: The essentially branching case}
\
\begin{definition}
	\label{Def:BrownianPermuton}
	Let $p\in[0,1]$. The Brownian separable permuton of parameter $p$ is a random permuton $\bm \mu^p$ whose distribution is characterized by
	\[\InducedPerm_k(\bm \mu^p) = \perm(\bm b_k), \text{ for every } k\geq 1\]
	where $\bm b_k$ is a uniform random binary tree with $k$ leaves, whose internal nodes are independently decorated with i.i.d. signs of bias $p_+$ (namely, $\proba(+)=p_+$ and $\proba(-)=p_-=1-p_+$).
\end{definition}
The existence and uniqueness in distribution of this permuton is shown in \cite[Lemma B.1]{Nous2}. An intrinsic construction of this object is given in \cite{MickaelConstruction}.

\begin{theorem}[Main Theorem: the essentially branching case]\label{Th:branchingCase}
Consider a tree-specification \eqref{eq:Specif} for $\TTT_0,\dots,\TTT_d$ that verifies Hypothesis (SC) (p.\pageref{Hyp:StronglyConnected}). 
We assume that
	\begin{enumerate}
		\item the specification is essentially branching,
		\item Hypothesis (AR) (p.\pageref{Hyp:AR}) holds,
		\item at least one series (either critical or subcritical) is aperiodic.
	\end{enumerate} 
Then all critical families converge to the same Brownian separable permuton. More precisely, there exists $p_+ \in [0,1]$ such that for every $i\in I^\star$, letting $\bm \sigma_n$ be a uniform permutation of size $n$ in $\TTT_i$,
$$
\mu_{\bm \sigma_n} \stackrel{(d)}{\to} \bm \mu^{p_+}.
$$
Furthermore, the bias parameter $p_+$ can be explicitly computed with \cref{eq:DefParametrePermutonBrownien} p.\pageref{eq:DefParametrePermutonBrownien}.
\end{theorem}

\begin{remark}
Recall that Hypothesis (AR) holds in particular if there are only finitely many simple permutations in the $\TTT_i$'s.\\
Item iii) is again a weak assumption. It is automatically satisfied in the case of classes with finitely many simple permutations.
Indeed, at least one series is not a polynomial (otherwise the class itself is finite) and again by \cite{DrmotaPierrot} it has to be aperiodic.
\end{remark}

We show two examples of classes having an essentially branching decomposition, 
whose limits are Brownian separable permutons of explicit parameters. 
The first example is build on purpose to display a limiting behavior of this kind 
for a class which is not substitution-closed. 
The second example is the famous class $Av(132)$. 
Its limiting permuton, which is supported by the antidiagonal,
is a degenerate  Brownian separable permuton. 

\subsubsection{A non-degenerate branching case}
\label{sec:Exemple_branching}\ \\
We consider the class $\TTT_0=\Av(2413,31452,41253,41352,531642)$.
The only simple permutation in the class is $3142$,
so that we apply the algorithm of \cite{BBPPR}.
In \cref{sec:Exemple_branching_annexe}, we give the specification of this class and apply \cref{Th:branchingCase}, 
to get that the limit is the biased Brownian separable permuton of parameter $p_+$, where $p_+ \approx 0.4748692376...$ is the only real root of the polynomial
\[z^{9} - 3 z^{8} + \frac{232819}{62348} z^{7} - \frac{78093}{31174} z^{6} + \frac{243697}{249392} z^{5} - \frac{54293}{249392} z^{4} + \frac{24529}{997568} z^{3} - \frac{125}{62348} z^{2} + \frac{45}{62348} z - \frac{2}{15587}.\]

\subsubsection{A degenerate branching case: $Av(132)$}\label{sec:Av(132)_suite}
We continue the study of this Catalan class, which we started in \cref{sec:Av(132)_debut}. 
Recall that this class has an essentially branching specification, with a single strongly connected component among the critical series. 
Moreover, it involves the subcritical series $T_{\langle 21\rangle}=\frac{z}{1-z}$ which is aperiodic.
Finally, since there is no simple permutation in $Av(132)$, Hypothesis (AR) holds and we can apply \cref{Th:branchingCase}:
there exists some parameter $p_+$ such that the limiting permuton of $Av(132)$ is the Brownian separable permuton of parameter $p_+$.
Moreover, we can read directly from the specification that for all $i,j,j'$, we have $E_{i,j,j'}^+ = 0$ where $E^\eps_{i,j,j'}$ are defined in \cref{def:Eij}. It follows from
 \cref{eq:DefParametrePermutonBrownien} p.\pageref{eq:DefParametrePermutonBrownien} that $p_+ = 0$ and $p_- = 1$: the limiting permuton is the antidiagonal.

We point out that for this particular class $Av(132)$, much more is known regarding the limiting shape \cite{MinerPak,HoffmanBrownian1} and the limiting distributions of pattern occurrences \cite{Janson132}.
We chose to present here this class to show a degenerate example which converges to the main diagonal.

\begin{remark}
	In \cref{ex:layered} we saw another permutation class whose limiting permuton is supported by a diagonal.
	The example $Av(132)$ is however very different: the limit is a degenerate Brownian separable permuton
	while the limit of the layered permutations of \cref{ex:layered} is a degenerate $X$-permuton.
\end{remark}

\subsection{Outline of the proof}
\label{Subsec:outline_proof}

As mentioned in \cref{ssec:ToolsIntro}, we make use of analytic combinatorics tools to establish our results.
To this end, we first note that our hypothesis implies the following
behavior of critical series near the dominant singularity:
\begin{itemize}
  \item in the essentially linear case, all critical series have simple poles;
  \item in the essentially branching case, they have square-root singularities.
\end{itemize}
For details, we refer to \cref{lem:CheckerAppendice,lem:CheckerAppendiceBranching} respectively.

We will use the following characterization of convergence of random permutations to a random permuton: 
it is equivalent to
the convergence of the random patterns of the considered random permutations to the random permutations induced by the permuton (\cref{thm:randompermutonthm}).
Since we view permutations as trees, and we wish to study patterns in permutations, 
we are lead to consider trees with marked leaves (see \cref{Sec:InducedTrees}).
Using a decomposition, we obtain a combinatorial equation describing the family of trees
with $k$ leaves inducing a given tree (\cref{prop:EnumerationCaterpillar,prop:EnumerationBranchingCase}).
Then we perform a careful analysis of the corresponding generating series 
to determine their behavior near the singularity (\cref{eq:Asympt_Tto_Linear,eq:Asympt_Tto_Branching}).

This allows us to compute the limiting distribution of the random subtree induced by $k$ uniform random leaves
in a uniform random tree in any one of the critical families (\cref{prop:probaCaterpillars,prop:ProbaBinaire}).
In the essentially linear case, this limiting distribution is supported by trees called {\em caterpillar} (see \cref{def:caterpillar}).
Since the substitution tree of a random permutation induced by the $X$-permuton 
is a caterpillar with the same distribution (\cref{prop:MarginalsXPermuton}), this concludes the proof of \cref{Th:linearCase}.
On the contrary, in the essentially branching case, the  limiting distribution is supported by signed binary trees.
Since the substitution tree of a random permutation induced by the Brownian separable permuton
is a signed binary tree with the same distribution (\cref{Def:BrownianPermuton}),
this concludes the proof of \cref{Th:branchingCase}.

\section{Tree toolbox}\label{Sec:TreeToolbox}

\subsection{Induced trees}\label{Sec:InducedTrees}

Since permutations are encoded by trees and since we are interested in patterns in permutations,
we consider an analogue of patterns in trees: this leads to the notion of {\em induced trees}.

\begin{definition}[First common ancestor]\label{dfn:common_ancestor}
	Let $t$ be a tree, and $u$ and $v$ be two nodes (internal nodes or leaves) of $t$. 
	The \emph{first common ancestor} of $u$ and $v$ is the node furthest away from the root $\varnothing$ that appears 
	on both paths from $\varnothing$ to $u$ and from $\varnothing$ to $v$ in $t$. 
\end{definition}

\begin{definition}[Induced tree]\label{dfn:induced_subtree}
	Let $t$ be a substitution tree, and let $\SsEnsemble$ be a subset of the leaves of $t$.
	The tree $t_\SsEnsemble$ induced by $\SsEnsemble$ is the substitution tree of size $|\SsEnsemble|$ defined as follows.
	The tree structure of $t_\SsEnsemble$ is given by:
	\begin{itemize}
		\item the nodes of $t_\SsEnsemble$ are in correspondence with
		the union of $\SsEnsemble$ and of the set of first common ancestors of two (or more) nodes in $\SsEnsemble$; 
		\item the ancestor-descendant relation in $t_\SsEnsemble$ is inherited from the one in $t$; 
		\item the order between the children of an internal node of $t_\SsEnsemble$ is inherited from $t$. 
	\end{itemize}
	The label of an internal node $v$ of $t_\SsEnsemble$ is defined as follows:
	\begin{itemize}
		\item if $v$ is labeled by a permutation $\theta$ in $t$, the label of $v$ in $t_\SsEnsemble$ is given by the pattern of $\theta$ 
		induced by the children of $v$ having a descendant that belongs to $t_\SsEnsemble$ (or equivalently, to $\SsEnsemble$).
	\end{itemize}
\end{definition}
In the specific case of a subtree induced by \emph{two} leaves, $\ell_1$ and $\ell_2$, 
the induced subtree may be \begin{tikzpicture}[baseline=0pt, every node/.style={inner sep=0},level distance=4mm, sibling distance=5mm] \node {$\oplus$} [grow=up]
child {node {$\bullet$} } 
child {node {$\bullet$} };\end{tikzpicture} 
or \begin{tikzpicture}[baseline=0pt, every node/.style={inner sep=0},level distance=4mm, sibling distance=5mm] \node {$\ominus$} [grow=up]
child {node {$\bullet$} } 
child {node {$\bullet$} };\end{tikzpicture}.
In the first (resp. second) case, we say that $\ell_1$ and $\ell_2$ \emph{induce the sign} $\oplus$ (resp. $\ominus$). 

A detailed example of the induced tree construction is given in \cref{fig:ExampleCanonicalTree}.
\begin{figure}[htbp]
	\begin{center}
		\includegraphics[width=12cm]{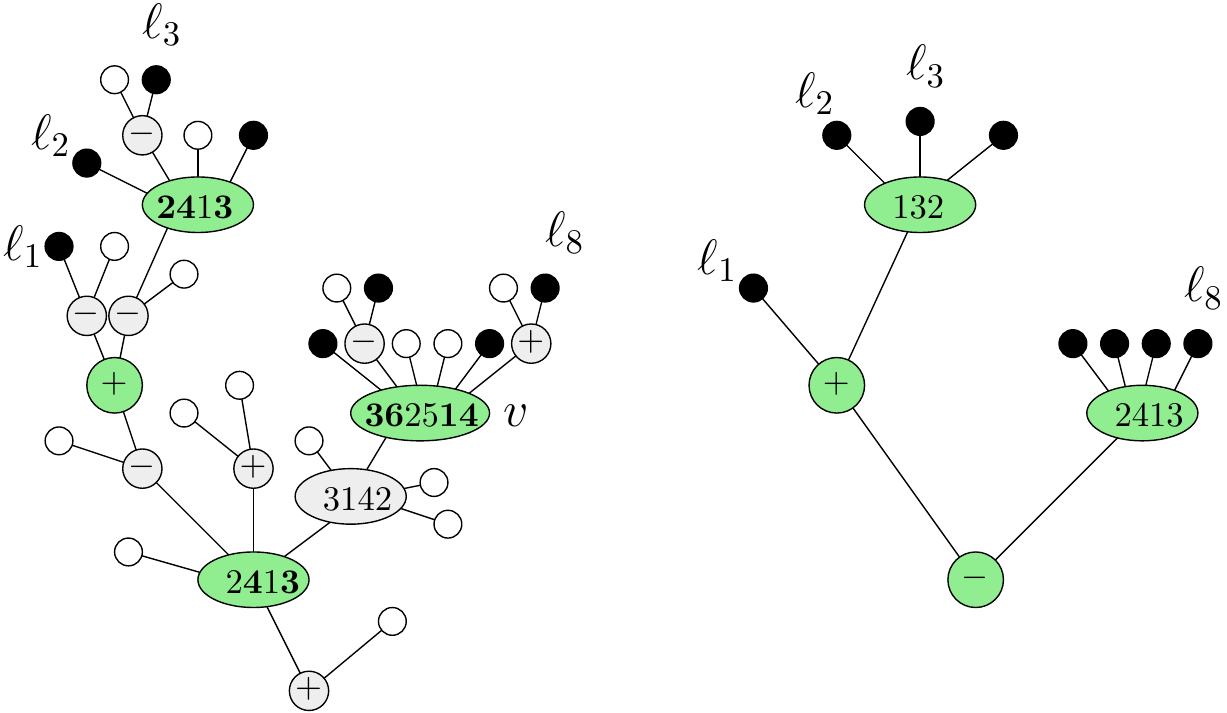}
	\end{center}
	\caption{On the left: A substitution tree $t$ of size $n=24$ (which happens to be a standard tree), where leaves are indicated both by $\circ$ and $\bullet$. 
		Among these $24$ leaves, $|\SsEnsemble|=8$ leaves are marked and indicated by $\bullet$. 
		In green are shown the internal nodes of $t$ which are first common ancestors of these $8$ marked leaves.
		On the right: The substitution tree induced by the $8$ marked leaves.
		Observe that the node $v$ labeled by $362514$ in $t$ is labeled by $2413$ in $t_\SsEnsemble$. 
		This is because only the first, second, fifth and sixth children of $v$ have descendants that belong to $\SsEnsemble$, and $\pat_{\{1,2,5,6\}}(362514)=2413$.
		The induced tree is not standard since $132$ is not simple.}
	\label{fig:ExampleCanonicalTree}
\end{figure}

\begin{remark}
The definition of induced trees can be extended in the case when $\SsEnsemble$ is a subset of nodes (not necessarily leaves),
but in this case $t_\SsEnsemble$ is not necessarily a substitution tree and its number of leaves may be less than $|\SsEnsemble|$.
\end{remark}

For a substitution tree with $n$ leaves, it is convenient to identify the leaves of $t$ from left to right with $[n]=\{1\dots n\}$.

\begin{observation}
	By definition, for any substitution tree $t$ with $n$ leaves and subset $\SsEnsemble$ of $[n]$, $t_\SsEnsemble$ is a substitution tree.
	However, if $t$ is a standard tree, $t_\SsEnsemble$ is a substitution tree which is not necessarily standard
	(see for example \cref{fig:ExampleCanonicalTree}).
\end{observation}

Moreover, we have the following important feature (illustrated by \cref{fig:DiagrammeCommutatif}).

\begin{lemma}\label{lem:DiagrammeCommutatif} 
	Let $t$ be a substitution tree with a subset $\SsEnsemble$ of marked leaves.
	We have 
	$$
	\pat_\SsEnsemble(\perm(t)) = \perm(t_\SsEnsemble).
	$$
\end{lemma}

As in our previous work~\cite{Nous2}, this lemma is essential, 
since it allows to replace the counting of the number of occurrences of a given pattern in some family of permutations 
by that of induced trees equal to a given tree $\patterntree$ in the corresponding family of standard trees.

\begin{figure}[htbp]
	\begin{center}
		\includegraphics[width=10cm]{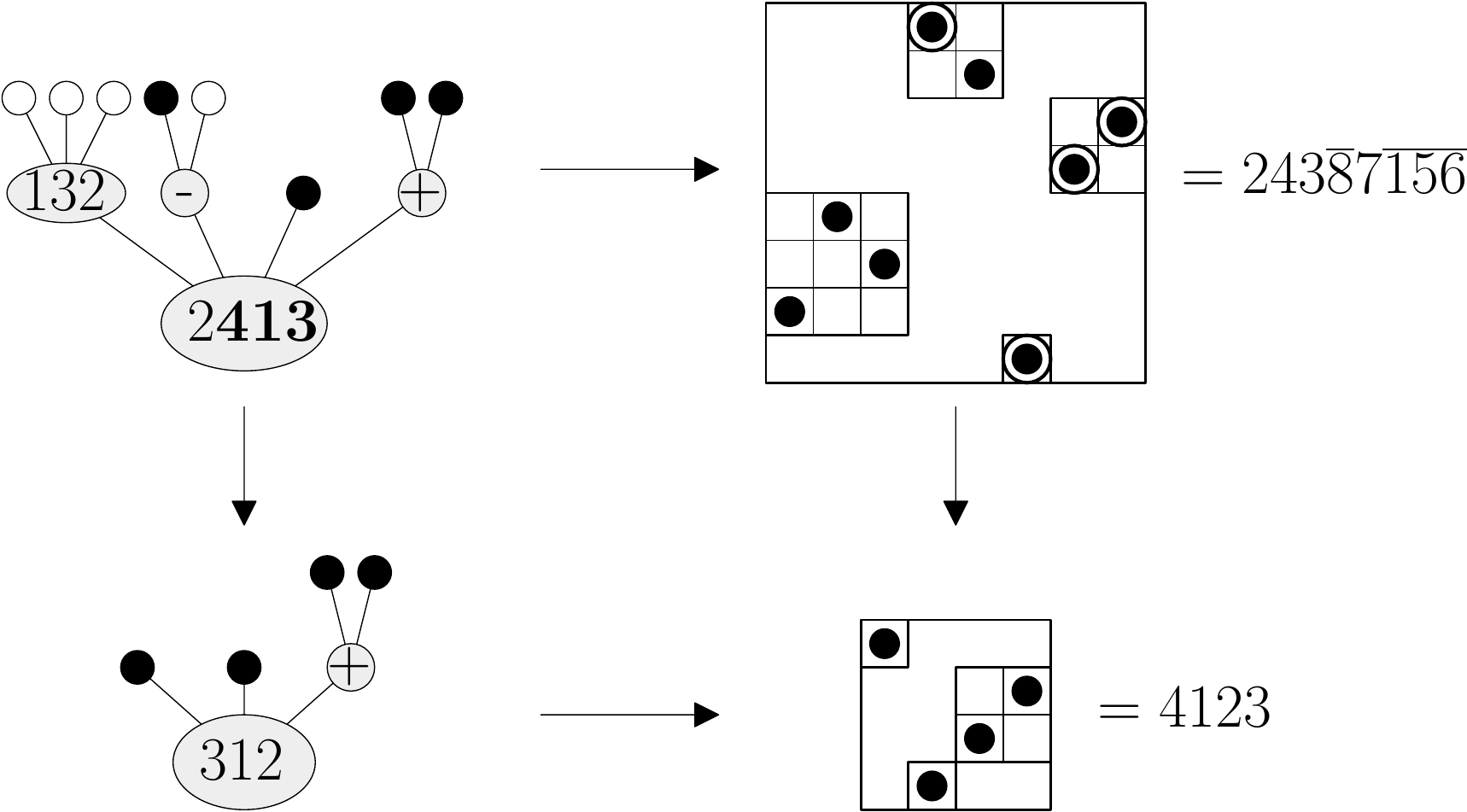}
	\end{center}
	\caption{Illustration of \cref{lem:DiagrammeCommutatif}. 
		Top: A substitution tree $t$ with marked leaves (in this example $\SsEnsemble=\{4,6,7,8\}$), and the permutation $\perm(t)$ it encodes, with the corresponding $|\SsEnsemble|$ marked elements (at positions in $\SsEnsemble$). 
		Bottom: The induced tree $t_\SsEnsemble$ and the induced pattern $\pat_\SsEnsemble(\perm(t)) = \perm(t_\SsEnsemble)$.}
	\label{fig:DiagrammeCommutatif}
\end{figure}

\subsection{Type of a node}

A tree-specification like \eqref{eq:Specif} allows to build the elements of the families $\TTT_i$ recursively in a canonical way.
In this recursive construction of a tree $t$ of $\TTT_i$,
every fringe subtree is taken in one of the $\TTT_j$. 
We will say that the subtree, or equivalently its root, is of type $j$. 
More formally, the type of a node in a tree $t$ in $\TTT_i$ can be recursively defined as follows.

\begin{definition}[Type of a node]
Consider a specification of the form of \eqref{eq:Specif} (see p.\pageref{eq:Specif}). 
Let $t$ be a tree in some $\TTT_i$, and let $v$ be a node in $t$.
The \emph{type} of $v$ in $t$ in $\TTT_i$ is defined as follows.
\begin{itemize}
\item If $v$ is the root of $t$, then the type of $v$ in $t$ in $\TTT_i$ is $i$.
\item Otherwise, there is a unique $\pi \in\SSS_{\TTT_i} \uplus \{\oplus, \ominus\}$ 
  and a unique $|\pi|$-tuple $(k_1,\dots,k_{|\pi|})\in K^i_\pi$ such that $t$ can be decomposed as:
\begin{center}
 \includegraphics[width=3cm]{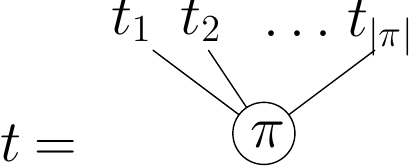},
\end{center}
where each $t_j \in\TTT_{k_j}$. 
Let $\ell\leq |\pi|$ be such that $v\in t_\ell$, 
then the type of $v$ in $t$ in $\TTT_i$
is the type of $v$ in $t_\ell$ in $\TTT_{k_\ell}$.
\end{itemize}
\end{definition}

\begin{remark}
It may happen that $\TTT_i \cap \TTT_j\neq \emptyset$. 
For example, in the specification \eqref{eq:SpecifClosedClasses} p.\pageref{eq:SpecifClosedClasses} for substitution-closed classes,
all trees whose root is labeled by a simple permutation
belong to all three classes. 
In such a case, caution is needed: the type of a node $v$ in a tree $t\in \TTT_i \cap \TTT_j$ 
is defined differently depending on whether $t$ is seen as a tree of $\TTT_i$ or of $\TTT_j$.
\end{remark}

\begin{example}
  Consider a substitution-closed class $\TTT$ with its tree-specification given by \eqref{eq:SpecifClosedClasses}.
  The three families of trees $\TTT$, $\TTT^{\nonp}$ and $\TTT^{\nonm}$ appear in this specification. 
  Let $t$ be a tree in any of $\TTT$, $\TTT^{\nonp}$ or $\TTT^{\nonm}$. 
  The type of the node of $t$ is either $\emptyset$, $\nonp$, or $\nonm$.
  Moreover, it is easy to see that the type of a non-root node $v$ in $t$ 
  is $\nonp$ (resp. $\nonm$) if the node is the left child of a node labeled with $\oplus$ (resp. $\ominus$), 
  and is $\emptyset$ otherwise.
  Only the type of the root of $t$ depends on which family $t$ is (considered to be) an element of.
  The type of the root of $t$ is by definition $\emptyset$ (resp. $\nonp$, $\nonm$) 
  when $t$ is (considered as) a tree of $\TTT$ (resp. $\TTT^{\nonp}$, $\TTT^{\nonm}$).
\end{example}

\subsection{Critical part of a tree}\label{Seq:Crit}

Consider again a tree-specification as in \eqref{eq:Specif}. 
Recall that a family $\TTT_i$ is critical (resp. subcritical) if $\rho_i=\min_j\{\rho_j\}$ (resp. $\rho_i>\min_j\{\rho_j\}$). 
For the asymptotic analysis,
it will be important to identify in a tree the set of nodes
of critical types. This is the purpose of the next definition.
\begin{definition}\label{def:CriticalSpine}
Consider a specification of the form \eqref{eq:Specif} and let $t$ be a standard tree in $\TTT_i$, for some $i$.
We denote by $\Crit_i(t)$  the set
 of nodes $v$ in $t$ such that the type of $v$ in $t$ in $\TTT_i$ is critical.
\end{definition}

Note that from \cref{lem:MonotonieDesRho}, $\Crit_i(t)$ is empty if $i \notin I^\star$. 
Again from \cref{lem:MonotonieDesRho}, if $i \in I^\star$, $\Crit_i(t)$ is a connected subset of $t$ (hence a tree) which contains the root. 
This allows to refer to the set $\Crit_i(t)$ as the \emph{critical subtree} of $t$ and to define, for every node $v$ of $t$, its \emph{first critical ancestor}: 
it is the first node met on the unique path from $v$ to the root of $t$ whose type is critical.

Furthermore, in the essentially linear case, for any tree $t$ in $\TTT_i$ with $i \in I^\star$, 
the critical subtree of $t$ is actually a chain from the root to a node $v$ of $t$. 
We alternatively call $\Crit_i(t)$ the \emph{critical spine} of $t$ in this case, 
and the node $v$ is referred to as the \emph{head} of $t$.

\subsection{Blossoming trees}

In both the essentially linear and the essentially branching cases, we derive asymptotics from an exact combinatorial result (\cref{prop:EnumerationCaterpillar} or \cref{prop:EnumerationBranchingCase}) that gives an expression for the generating function of trees of type $\TTT_i$ with $k$ marked leaves which induce a given subtree $\patterntree$. 
This expression results from a decomposition of the families $\TTT_i$ into some families of \emph{blossoming trees}, that we now define.

\begin{definition}\label{def:Tfleche}
	For $0\leq i,j\leq d$, we define $\TTT_{\to i}^j$ as the 
	family of trees $t$ with one marked leaf $\ell$, called the {\em blossom} and represented by $\ast$, 
	such that the tree obtained by replacing $\ast$
	by a tree of $\TTT_j$ belongs to $\TTT_i$, with the additional condition that the type in $\TTT_i$ of the node that used to be the blossom is $j$.
\end{definition}
Observe that in general, a tree in  $\TTT_{\to i}^j$ does not belong to  $\TTT_{i}$.

In the following proposition, we show that families $\TTT_{\to i}^j$'s inherit a combinatorial specification
from the one of the $\TTT_{i}$'s.

\begin{proposition}[Specification of the $\TTT_{\to i}^j$'s]\label{Prop:SpecifArbres}
	Assume that the equation for $\TTT_i$ in the specification \eqref{eq:Specif} is
	$$
	\TTT_i= \eps_i\{\bullet \} \ \uplus \ \biguplus_{\pi\in\SSS_{\TTT_i}\uplus\set{\oplus,\ominus}} \biguplus_{(k_1,\dots,k_{|\pi|})\in K_\pi^i} \pi[\TTT_{k_1}, \TTT_{k_2}, \ldots, \TTT_{k_{|\pi|}}]        \qquad (0\leq i\leq d),
	$$
	where $\bullet$ is the trivial tree made of just one leaf.
	Then we have:
\begin{equation}
  \TTT_{\to i}^j= \mathbf 1_{i=j}\{\ast \} \ \uplus \ \biguplus_{\pi\in\SSS_{\TTT_i} \uplus\set{\oplus,\ominus}} \biguplus_{(k_1,\dots,k_{|\pi|})\in K_\pi^i} \biguplus_{\ell=1}^{|\pi|}  \pi[\TTT_{k_1}, \ldots, \TTT_{\to k_\ell}^j, \ldots, \TTT_{k_{|\pi|}}] 
  \quad (0\leq i,j\leq d),
		\label{Eq:SpecifArbres}
\end{equation}
		where $\ast$ is the trivial tree reduced to the blossom.
\end{proposition}

\begin{proof}
  Trivially, the class $\TTT_{\to i}^j$ contains the tree reduced to a blossom
  if and only if $i=j$.
  This explains the term $\mathbf 1_{i=j}\{\ast\}$.

  Let $t \in \TTT_{\to i}^j$. 
We now restrict to the case where the blossom of $t$ is not at the root. Let $t_j \in \TTT_{j}$. Denote by $tt_j$ the tree obtained by replacing the blossom of $t$ with $t_j$. 
By definition of the class $\TTT_{\to i}^j$,
the tree $tt_j$ is in $\TTT_i$. As a result, $tt_j$ belongs to the union
	$$
	\TTT_i= \eps_i\{\bullet \} \ \uplus \ \biguplus_{\pi\in\SSS_{\TTT_i}\uplus\set{\oplus,\ominus}} \biguplus_{(k_1,\dots,k_{|\pi|})\in K_\pi^i} \pi[\TTT_{k_1}, \TTT_{k_2}, \ldots, \TTT_{k_{|\pi|}}]. $$
We cannot have $tt_j = \bullet $, because then necessarily the blossom of $t$ is its root.
Hence $tt_j$ belongs to a term of the form $\pi[\TTT_{k_1}, \ldots, \TTT_{k_{|\pi|}}]$ for $\pi\in\SSS_{\TTT_i} \uplus\set{\oplus,\ominus}$ and $(k_1,\dots,k_{|\pi|})\in K_\pi^i$. 
Then the blossom (and the copy of $t_j$) must be contained in one of the fringe subtrees rooted at a child of the root of $tt_j$, say the $\ell$-th one, with $1\leq \ell \leq |\pi|$. 
Hence $t$, which is recovered by removing the copy of $t_j$ in $tt_j$ and replacing it by a blossom, belongs to $\pi[\TTT_{k_1}, \ldots, \TTT^j_{\to k_\ell},\ldots , \TTT_{k_{|\pi|}}]$.

This proves the direct inclusion in the statement of the proposition. For the reverse inclusion, consider a tree $t$ belonging to the right hand side of \cref{Eq:SpecifArbres}, and replace the blossom by a tree $t_j$ of $\TTT_j$. This immediately yields a tree in $\TTT_i$. Hence $t\in \TTT_{\to i}^j$.
\end{proof}

For $0\leq i\leq d$, let $T_{\to i}^j$ be the generating function of the family $\TTT_{\to i}^j$, where trees are counted by the number of leaves (we take the convention that the blossom is not counted).

\cref{Prop:SpecifArbres} has the following consequence (recall that series $F_i$'s are defined by \eqref{eq:SystemeSeries} p.\pageref{eq:SystemeSeries}).

\begin{corollary}
	Let $\TT_\to(z)$ be the matrix of generating functions $\TT_\to=(T^j_{\to i})_{0\leq i,j\leq d}$. It holds that
	\begin{equation}
	\TT_\to(z) = \mathbb K(T_0(z),\ldots T_d(z))\cdot \TT_\to(z) + \Id, 
	\label{eq:TPasFleche}
	\end{equation}
	where $\mathbb K$ is the $(d+1)\times (d+1)$ matrix defined by 
$$
K_{i,j}(y_0,\ldots y_d) = \frac{\partial F_i(y_0,\ldots y_d)}{\partial y_j}.
$$	
	If we restrict to critical families and define $\mathbb T^\star_\to = (T_{\to i}^j)_{i,j\in I^\star}$, then we have
	\begin{equation}
	\mathbb T^\star_\to(z) = \mathbb M^\star(z,\mathbf T^\star(z)) \mathbb T^\star_\to(z) + \Id,
	\label{eq:Tfleche}
	\end{equation}
        where $\mathbb M^\star$ was defined in \cref{eq:LinearSystem} (p.\pageref{eq:LinearSystem}) for the essentially linear case, 
        and in \cref{eq:Def_Mstar} (p.\pageref{eq:Def_Mstar}) for the essentially branching case. 
\end{corollary}

\section{The essentially linear case}\label{Sec:ProofsLinear}

This section is to devoted to the proof of \cref{Th:linearCase} through the asymptotic analysis of 
the systems \eqref{eq:SystemeSeries} and \eqref{eq:TPasFleche} in the essentially linear case. 
In this case, an important consequence of the specification is that standard trees can be decomposed along a critical spine (\cref{def:CriticalSpine}).

To help with the reading of this section, we summarize here the different generating series which we will use throughout \cref{Sec:ProofsLinear}: 

\begin{center}
\begin{tabular}{| l l l l |}
\hline
Series & Counts for... & Defined in... & Counted by... \\
\hline
$T_{\to i}^j$ & Blossoming trees & \cref{def:Tfleche} & Number of leaves (without the blossom)\\
$\DDD^{\gauche,+}_{j,i}$ & Marked blossoming trees & \cref{Defi:D_gauche_plus} & Number of unmarked leaves\\
$T_{i,\patterntree}$ & Trees inducing $\patterntree$ & \cref{Defi:T_i_t0} & Number of unmarked leaves\\
\hline
\end{tabular}
\end{center}

\subsection{Caterpillar and associated permutations}

Because of the existence of a critical spine, some particular trees will play a significant role in the analysis: these are the \emph{caterpillars}. 

We say that a tree is \emph{binary} when every internal node has \emph{exactly} $2$ children. 

\begin{definition}\label{def:caterpillar}
	A {\em caterpillar} of size $k$ is a binary plane tree with
	\begin{itemize}
		\item $k$ internal nodes 
		labeled by either $\oplus$ or $\ominus$;
		\item a special leaf, called the {\em head}; 
	\end{itemize}
	such that 
	all internal nodes are on the path from the root to the head.

Leaves different from the head are called {\em regular}.
\end{definition}
A caterpillar is drawn in \cref{Fig:Caterpillar}.
Since a caterpillar is binary, there is exactly one regular leaf 
branching on each internal node
and, therefore, the number of regular leaves in a caterpillar of size $k$ is $k$.

We take the following convention:
\begin{itemize}
	\item internal nodes are ordered from the first node $v_1$ to the $k$-th node $v_k$ according to their distance to the root (namely, $v_r$ is at distance $r-1$ from the root);
	\item leaves are ordered by the breadth-first traversal: for $1\leq r\leq k$, the $r$-th leaf $\ell_r$ is a child of the $r$-th internal node $v_r$.
\end{itemize}
To a caterpillar $\patterntree$ of size $k\geq 1$ we associate its code word $(e_1,\eps_1)\ldots(e_k,\eps_k)$, defined as follows: for each $1\leq r\leq k$
\begin{itemize}
	\item $e_r\in\{\gauche,\droite\}$  
	indicates whether $\ell_r$ is a left or a right child of $v_r$, 
	is the sign of the internal node $v_r$ of $\patterntree$.
\end{itemize}
Remark that a caterpillar is completely determined by its code word.
\begin{remark}\label{rem:caterpillar}
	In the literature, caterpillars are usually trees seen as unrooted graphs
	whose internal  nodes form a path. 
	Our caterpillars are, on the contrary, always rooted and binary, that is, every internal node has \emph{exactly} 2 children.
\end{remark}

With a caterpillar $\patterntree$ (of size $k$), 
we associate a substitution tree $\Reduc(\patterntree)$
as follows: erase the head of $\patterntree$,
merge its parent (the internal node $v_k$) and its sibling (the leaf $\ell_k$) 
into a new leaf, also denoted by $\ell_k$. 
Of course, this substitution tree encodes the permutation $\perm(\Reduc(\patterntree))$.
\cref{Fig:Caterpillar} shows an example of caterpillar,
with its associated substitution tree and permutation.
\begin{figure}[thbp]
	\begin{center}
		\includegraphics[width=9cm]{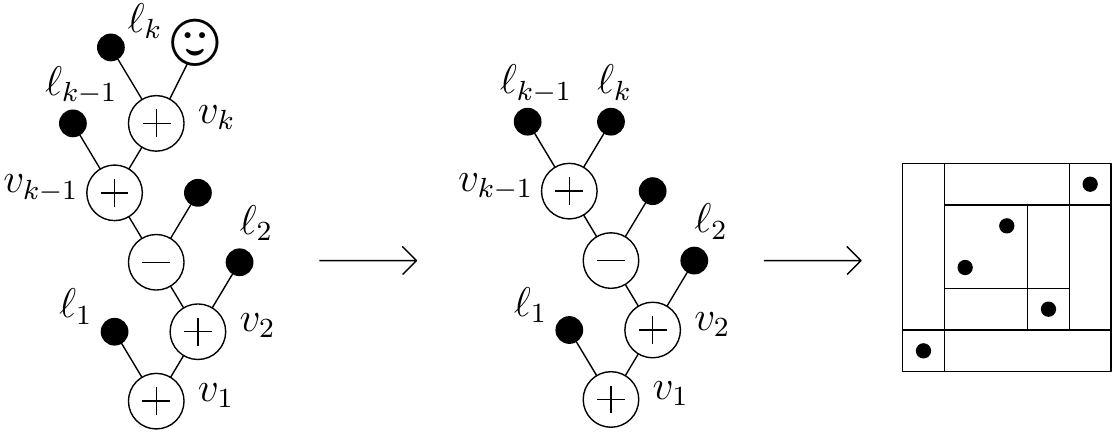}
		\caption{
Left: A caterpillar $\patterntree$ with $k=5$ regular leaves and one head. Its code word is $(\text{left},+)(\text{right},+)(\text{right},-)(\text{left},+)(\text{left},+)$.
Middle: The associated substitution tree $\Reduc(\patterntree)$. 
Right: The permutation $\perm(\Reduc(\patterntree))$.
			\label{Fig:Caterpillar}}
	\end{center}
\end{figure}

\subsection{Extracting a caterpillar}
In this section, we consider standard trees in a critical family $\TTT_i$, with $k$ marked leaves.
Recall from \cref{Seq:Crit} that in the essentially linear case,
the set of critical nodes in each tree $t$ in $\TTT_i$ forms the critical spine of $t$, whose node furthest away from the root is called the head of $t$.

\begin{definition}\label{Defi:T_i_t0}
	Fix a caterpillar $\patterntree$ of size $k$.
	For $i\in I^\star$, the family $\TTT_{i,\patterntree}$ is the set of pairs $(t,\SsEnsemble)$ where $t$ is a tree in $\TTT_i$ and $\SsEnsemble$ is a subset of $k$ leaves in $t$ 
	(called \emph{marked} leaves, and taken without any order on them) such that 
	\begin{itemize}
		\item the $k$ marked leaves together with the head of $t$ 
		induce the subtree $\patterntree$;
		\item moreover, in this induction, the head of $t$
		should correspond to the head of $\patterntree$.
	\end{itemize}
	We denote by $T_{i,\patterntree}$ the corresponding counting series (where the size is the number of unmarked leaves).
\end{definition}
\begin{remark}
  The reader might be surprised that we consider the subtree induced by the head and $k$ random leaves,
  while we announced in \cref{Subsec:outline_proof} that we would be interested in that induced by only the $k$ random leaves.
  Clearly, the former contains more information than the latter.
  Moreover, this refinement will prove useful,
  because it makes easier the decomposition of $\TTT_{i,\patterntree}$ used in the proof of \cref{prop:EnumerationCaterpillar}.
\end{remark}

Our next step towards the enumeration of $\TTT_{i,\patterntree}$ (\cref{prop:EnumerationCaterpillar}) is to decompose $\TTT_{i,\patterntree}$ in terms of smaller classes.
For this, we need to define yet another family of marked trees.
\begin{definition}\label{Defi:D_gauche_plus}
Let $\DDD^{\gauche,+}_{i,j}$ be the combinatorial class of trees $t$ in $\TTT_{\to i}^{j}$ with one additional marked leaf such that
\begin{itemize}
	\item the blossom is a child of the root of $t$;
	\item the additional marked leaf is to the left of the blossom;
	\item the blossom and the marked leaf induce the sign $\oplus$ (see definition in \cref{Sec:InducedTrees}).
\end{itemize}
A schematic view of a tree in $\DDD^{\gauche,+}_{i,j}$ is provided in \cref{{Fig:D_+_left}}.
We define in an analogous way the combinatorial classes $\DDD^{\droite,+}_{i,j}$,
$\DDD^{\gauche,-}_{i,j}$ and $\DDD^{\droite,-}_{i,j}$.\\
We denote by  $D^{\gauche,+}_{i,j}$, $D^{\droite,+}_{i,j}$ $D^{\gauche,-}_{i,j}$ and $D^{\droite,-}_{i,j}$ the associated generating functions. 
In these series, the power of $z$ is the number of leaves which are {\em neither blossom nor marked leaves}.
\end{definition}

\begin{figure}[thbp]
	\begin{center}
		\includegraphics[width=4cm]{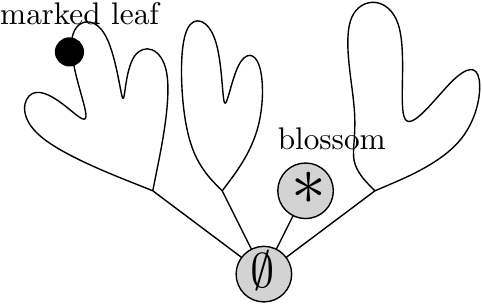}
		\caption{Left: A sketch of a tree in some family $\DDD^{\gauche,+}_{i,j}$ (assuming that the marked leaf and the blossom induce $\oplus$).
			\label{Fig:D_+_left}}
	\end{center}
\end{figure}

\begin{proposition}
	 For all $i,j \in I^\star$, we have
	\[\sum_{e,\eps} D^{\eps, e }_{i,j}(z) =   \frac{\partial }{\partial z} \mathbb M^\star_{i,j}(z) . \]
	If Hypothesis (RC) holds, this implies in particular that all $D^{\eps, e }_{i,j}(z)$ converges at $z=\rho$.
	\label{prop:SumOfD}
\end{proposition}
\begin{proof}
We have that 
	$\DDD^{\droite,+}_{i,j}\uplus \DDD^{\droite,-}_{i,j}\uplus \DDD^{\gauche,+}_{i,j}\uplus \DDD^{\gauche,-}_{i,j}$
 is the combinatorial class of trees in  $\TTT_{\to i}^{j}$  with one marked leaf such that the blossom is a child of the root.
From \eqref{eq:SystemeSeries}, 
 it is counted by

\[\sum_{e,\eps} D^{\eps, e }_{i,j} (z)= \frac{\partial}{\partial z}\left(\frac{\partial F_i(y_0,\ldots ,y_d)}{\partial y_j} \Big|_{(T_0(z),\dots,T_d(z))}\right).\]
Indeed the operator $\tfrac{\partial }{\partial y_j}$ indicates the replacement of one child of type $j$ of the root by a blossom; and the operator  $\tfrac{\partial }{\partial z}$ amounts to marking a leaf.
The equality $\sum_{e,\eps} D^{\eps, e }_{i,j} (z)=   \frac{\partial }{\partial z} \mathbb M^\star_{i,j}(z)$ then follows by definition of $\mathbb{M}^\star$ (see p.\pageref{eq:MDeriveesSecondes}) 
and Hypothesis (RC) ensures the convergence at $z=\rho$.
\end{proof}
\begin{proposition}[Enumeration of trees with marked leaves inducing a given caterpillar]
	\label{prop:EnumerationCaterpillar}
	Let $\patterntree$ be a caterpillar with $k$ regular leaves with code word $(e_1,\eps_1)\dots (e_k,\eps_k)$. 
    Then the vector $\mathbf T^\star_\patterntree=(T_{i,\patterntree})_{i \in I^\star}$ is given by
	\begin{equation}
      \mathbf{T}^\star_{\patterntree}= 	\mathbb{T}^\star_{\to }\, \mathbb{D}^{e_1,\eps_1}\, \mathbb{T}^\star_{\to }\, \mathbb{D}^{e_2,\eps_2} \dots \mathbb{T}^\star_{\to }\,  \mathbb{D}^{e_k,\eps_k}\, \mathbf{T}^\star,
	\label{eq:MasterEquation}
	\end{equation}
where $ \mathbb{D}^{e,\eps}$ denotes the matrix $ \left(D^{e,\eps}_{i,j}\right)_{i,j\in I^\star}$.\\
(\emph{Recall that in \eqref{eq:MasterEquation}, the trees of $\TTT_{i,\patterntree}$ are counted by the number of unmarked leaves.})
\end{proposition}

\begin{proof}
\emph{(The main notation of the proof is summarized in \cref{Fig:SqueletteCompliqueCaterpillar}.)}\\
\begin{figure}[thbp]
	\begin{center}
		\includegraphics[width=13cm]{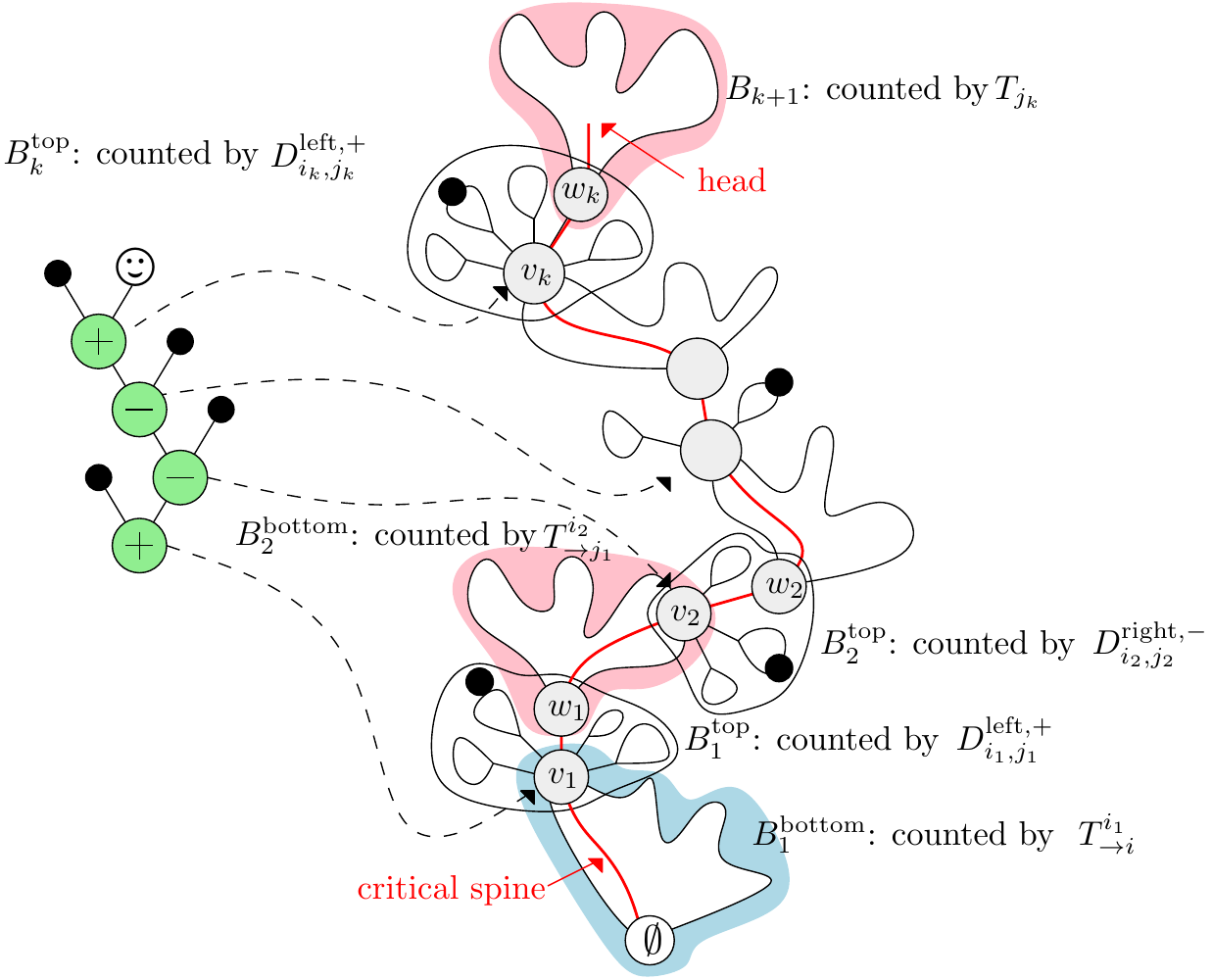}
		\caption{Left: A caterpillar tree $\patterntree$ with $k=4$ regular leaves and one head.
			Right: A schematic view of a tree with $k$ marked leaves in $\TTT_{i,\patterntree}$.
			\label{Fig:SqueletteCompliqueCaterpillar}}
	\end{center}
\end{figure}
We start by fixing some general convention to decompose trees.
Given a node $v$ in a tree, we can split the tree into two parts:
the top part (\emph{i.e.} the fringe subtree rooted at $v$) and the bottom part
(its complement in terms of edges). The node $v$ belongs to both parts.
To be able to reverse the operation without keeping extra information (such as the label of $v$),
we replace $v$ in the bottom part by a blossom.

Let $t\in \TTT_{i,t_0}$.
By definition, the $k$ marked leaves together with the head of $t$
induce the caterpillar $t_0$. 
Denote by $V$ the set of all first common ancestors of these $k+1$ nodes in $t$.
Because the head of $t$ corresponds to the head of $t_0$, all nodes of $V$ belong to the critical spine.
Therefore they can be totally ordered $v_1,\dots,v_k, v_{k+1}$ from the root to the head of $t$ (which is then $v_{k+1}$).
For $1\leq \ell \leq k$ we denote by $w_\ell$ the only child
of $v_\ell$ on the critical spine;
we also denote by $i_\ell$ (resp. $j_\ell$) the type of $v_\ell$ (resp. $w_\ell$) in $t$ in $\TTT_i$.
Then $i_\ell$ and $j_\ell$ belong to $I^\star$.
It is also convenient to set $w_0$ to be the root of $t$ and $j_0=i$ its type.

We now decompose $t$ successively with respect to the nodes $v_1,w_1, \dots, v_k, w_k$.
This results in $2k+1$ pieces that we denote respectively 
$B_1^{\mathrm{bottom}},B_1^{\mathrm{top}},\dots,B_k^{\mathrm{bottom}},B_k^{\mathrm{top}},B_{k+1}$ as follows:
$B_i^{\mathrm{bottom}}$ and $B_i^{\mathrm{top}}$ are the pieces rooted in $w_{i-1}$ and $v_i$, respectively,
while $B_{k+1}$ is the piece rooted at $w_k$.

By construction,
\begin{itemize}
\item $B_{k+1}$ is any tree in $\TTT_{j_k}$;
\item for $1\leq \ell\leq k$ the piece $B_\ell^{\mathrm{bottom}}$ is any tree in the family $\TTT_{\to j_{\ell-1}}^{i_\ell}$;
\item for $1\leq \ell\leq k$ the piece $B_\ell^{\mathrm{top}}$ is any tree in the family $\DDD^{e_\ell,\eps_\ell}_{i_\ell,j_\ell}$.
\end{itemize}
Only the last item needs a justification.
Recall that  $e_\ell\in\{\gauche,\droite\}$  is the position of the $\ell$-th leaf of $t_0$ with respect to its parent.
Since $t \in \TTT_{i,t_0}$, this forces the relative position of the marked leaf of $B_\ell^{\mathrm{top}}$ with respect to its blossom. Similarly, the pattern of the permutation labeling $v_\ell$ induced by the leaf and the blossom must match the sign $\eps_\ell\in\{+,-\}$.

Finally, this correspondence between  $t$ 
and $(B_1^{\mathrm{bottom}},B_1^{\mathrm{top}},\dots,B_k^{\mathrm{bottom}},B_k^{\mathrm{top}},B_{k+1})$ is one-to-one thanks to the unambiguous splitting/gluing procedure at blossoms. Moreover, this correspondence preserves the size (\emph{i.e.} the number of unmarked leaves), as the blossoms are not counted in families 
$\DDD^{e_\ell,\eps_\ell}_{i_\ell,j_\ell}$.

This decomposition translates on generating series as follows: for any $i \in I^\star$,
\begin{equation}\label{eq:MasterEquationPasMatricielle}
T_{i,t_0}=
\sum_{\substack{i_1,\dots,i_k\in I^\star \\ j_1,\dots,j_k\in I^\star}}
T_{\to i}^{i_1}\ D^{e_1,\eps_1}_{i_1,j_1}\ T_{\to j_1}^{i_2}\ D^{e_2,\eps_2}_{i_2,j_2}\ \dots\ T_{\to j_{k-1}}^{i_{k}}\  D^{e_{k},\eps_{k}}_{i_{k},j_{k}}\ T_{j_k}.
\end{equation}
Written in matrix notation this is exactly \eqref{eq:MasterEquation}.
\end{proof}

\subsection{Asymptotics of the main series} \label{Sec:AsymptMainSeriesLinear}

Our goal here is to describe the singular behavior of the series in $\mathbf{T}^\star_{\patterntree}$.
Hence (from \cref{prop:EnumerationCaterpillar}), we need information on  the singular behavior of the series that are the entries of $\mathbf T^\star(z)$ and $\mathbb T_\to^\star(z)$.

The following lemma is a consequence of a general result on linear systems proved in the appendix (\cref{Prop:AsymptotiqueLinear}). Recall that $\rho$ is the common radius of convergence of the critical series.

\begin{lemma}\label{lem:CheckerAppendice}
In the essentially linear case, the system we start from is 
\[\mathbf T^\star(z) = \mathbb M^\star (z) \mathbf T^\star(z) + \mathbf V^\star(z) \text{ with } \mathbb M^\star(z)=\left(\frac{\partial F_i(y_0,\ldots ,y_d)}{\partial y_j}\Big|_{(T_0(z),\dots,T_d(z))}\right)_{i,j\in I^\star}.
\]
Under Hypotheses (SC) and (RC) (p.\pageref{Hyp:RC}), assuming moreover that at least one subcritical series is aperiodic, we have the following results.

  All series in $\mathbb T_\to^\star(z)  = (\Id - \mathbb M^\star)^{-1}$ and $\mathbf T^\star$ are analytic on a $\Delta$-domain at $\rho$.

Moreover, the matrix $\mathbb M^\star(\rho)$ has Perron eigenvalue $1$.
Denoting $\mathbf u$ and $\mathbf v$ the corresponding left and right positive eigenvectors normalized so that $\transpose{\mathbf u} \mathbf v=1$ ($\transpose{\mathbf u}$ stands for the transpose of the vector $\mathbf u$), we also have the following asymptotics near $\rho$:
	\begin{align}
		\mathbb T_\to^\star(z)  = (\Id - \mathbb M^\star(z))^{-1} &\sim \left(\frac {1} {\transpose{\mathbf u}(\mathbb M^\star)'(\rho)\mathbf v}\right) \frac {1} {\rho - z}\ \mathbf v\transpose{\mathbf u}.
		\label{eq:Asymp_Tfleche_linear}
		\\
		\mathbf T^\star(z)  &\sim 
		\left(\frac {\transpose{\mathbf u} \mathbf V^\star(\rho)} {\transpose{\mathbf u}(\mathbb M^\star)'(\rho) \mathbf v }\right) 
		\frac {1} {\rho - z}\ \mathbf v.
		\label{eq:Asymp_T_linear}
	\end{align}
In the above equations $\sim$ stands for coefficient-wise asymptotic equivalence.
Observe that the factors preceding $\tfrac {1} {\rho - z}$ are real numbers.
\end{lemma}

\begin{proof}
We check that this system satisfies all hypotheses of \cref{Prop:AsymptotiqueLinear}. 

\begin{itemize}
\item By assumption the system is strongly connected\footnote{This notion on systems is defined in the Appendix only. It is however equivalent to 
the graph $G^\star$ being strongly connected, which is ensured by Hypothesis (SC).} and linear.

\item As the valuation of each $F_i$ is at least 2, $\mathbb M^\star(z)$ involves series of valuation at least $1$ in the $T_i(z)$'s.
Since $T_i(0)=0$ for every $i$, we also have $\mathbb M^\star(0)=\mathbf 0$. 
\item Since $\mathbb M^\star(0)=\mathbf 0$, the matrix $\mathrm{Id}-\mathbb M^\star (z)$ is invertible in the ring of formal series. By
\cref{eq:LinearSystem} we have $\mathbf V^\star (z) = (\mathrm{Id}-\mathbb M^\star (z))\mathbf{T}^\star(z) \neq \mathbf 0$ because $\mathbf{T}^\star(z)$ is not identically zero.
\item Hypothesis (RC) ensures that the radius of convergence of all entries of $\mathbb M^\star$  and $\mathbf V^\star$ is strictly larger than $\rho$.
\item By assumption, there is at least one subcritical series $T_{i_0}$ which is aperiodic. 
Moreover there is a path $T_{i_0}\to T_{i_1}\to \dots \to T_{i_\ell} $ in $G_{\eqref{eq:Specif}}$ from $T_{i_0}$ to the critical strongly connected component (see \cref{Sec:DefLinearBranching}). 
We choose this path such that $T_{i_{\ell-1}}$ is subcritical and $T_{i_{\ell}}$ is critical,
therefore the series $T_{i_{\ell-1}}$ is aperiodic thanks to \cref{lem:MonotonieAperiodicite}. 
And as $T_{i_{\ell-1}}$ appears in at least one coefficient of $\mathbb M^\star$ (at line $i_\ell$)
this ensures that the g.c.d. of the periods of the series in $\mathbb M^\star$ is $1$.
\item Moreover by \cref{eq:Tfleche} (p.\pageref{eq:Tfleche}),  $\mathbb T_\to^\star(z) = (\Id - \mathbb M^\star(z))^{-1}$.
\end{itemize}
\cref{Prop:AsymptotiqueLinear} gives us the desired result.
\end{proof}

\subsection{Probabilities of caterpillars}

For all $e\in \{\gauche,\droite\}$, $\eps\in\{+,-\}$, we set
\begin{equation}\label{eq:DefProbaCaterpillar}
p_\eps^{ e } = \frac{\transpose{\mathbf u} \mathbb{D}^{\eps, e }(\rho) \mathbf v}{ \transpose{\mathbf u} (\mathbb M^\star)'(\rho) \mathbf v},
\end{equation}
where the matrix $\mathbb{D}^{\eps, e }$ is defined according to \cref{Defi:D_gauche_plus}, $\mathbb M^\star$, $\mathbf u$ and $\mathbf v$ are given in \cref{lem:CheckerAppendice}.

Then from \cref{prop:SumOfD},
\begin{equation} \label{eq:SommeDesP}
p_+^{\gauche}+p_+^{\droite}+p_-^{\gauche}+p_-^{\droite} = 1. 
\end{equation}
Hence we can see $\mathbf p = (p_+^{\gauche},p_+^{\droite},p_-^{\gauche},p_-^{\droite})$ 
as a probability distribution on $\{\gauche,\droite\}\times\{+,-\}$. 
We will prove at the end of Section~\ref{Sec:ProofsLinear} that the limiting object of the class $\TTT_i$ (with $i \in I^\star$) 
is the $X$-permuton of parameter $\mathbf p$. 
An important step is the following proposition. 

\begin{proposition}[Occurrences of a given caterpillar] 
Fix $i\in I^\star$ and $k\geq 2$. Consider a uniform random tree with $n$ leaves in $\TTT_i$, in which $k$ leaves are marked, also chosen uniformly at random.
We denote by $\tkinhead$ the tree induced by these $k$ marked leaves and the head of the critical spine.

In the essentially linear case, under Hypotheses (SC) and (RC),
assuming moreover that at least one subcritical series is aperiodic, we have:
\begin{enumerate}
\item The probability that  $\tkinhead$ is a caterpillar tends to $1$ when $n$ tends to infinity.
\item Let $\patterntree$ be a caterpillar with $k$ regular leaves and  with code word $(e_1, \eps_1)\ldots (e_k,\eps_k)$. 
\begin{equation}\label{eq:ProbaCaterpillar}
\proba (\tkinhead = \patterntree) \stackrel{n\to +\infty}{\to} p_{\eps_1}^{e_1}p_{\eps_2}^{e_2}...p_{\eps_k}^{e_k},
\end{equation}
where $p_{\eps}^{e}$'s are defined by \cref{eq:DefProbaCaterpillar}. In particular, the limit does not depend on $i \in I^\star$.
\end{enumerate}
\label{prop:probaCaterpillars}
\end{proposition}

\begin{proof}
Since the right-hand side of \cref{eq:ProbaCaterpillar}, summed among all code words of caterpillars of size $k$, add up to $1$ (see \cref{eq:SommeDesP}), the first item is an immediate consequence of the second item.

Let us prove \cref{eq:ProbaCaterpillar}.
	Fix a caterpillar $\patterntree$ with $k$ regular leaves and code word $(e_1, \eps_1)\ldots (e_k,\eps_k)$.
    We claim that
    \begin{equation}
      \proba (\tkinhead = \patterntree) = \frac {[z^{n-k}] T_{i,t_0}}{[z^{n-k}] \frac 1 {k!} T_i^{(k)}}.
      \label{eq:prob_t0_linear}
 \end{equation}
	Indeed, the numerator is the number of trees in $\TTT_i$ with $n$ leaves,
    among which $k$ unordered leaves are marked and induce, together with the head of the spine,
     the caterpillar $t_0$
     (recall that the exponent of $z$ in $T_{i,t_0}$ is the number of unmarked leaves,
     here $n-k$).
     Similarly,
     the denominator is the total number of trees in $\mathcal T_i$ with $n$ leaves including $k$ unordered  marked leaves.
The quotient is therefore the probability that $k$ unordered marked leaves
in a uniform random tree with $n$ leaves in $\TTT_i$ induce $t_0$, as claimed.

	We want to apply the transfer theorem (\cref{thm:transfert}) to the series $T_{i,\patterntree}$ and $\frac {T_i^{(k)}}{k!}$.\smallskip

    We first justify that $T_{i,t_0}$ and $T_i^{(k)}$ have radius of convergence $\rho$
    and are $\Delta$-analytic at $\rho$.
    For $T_i^{(k)}$, this follows from the first claim of \cref{lem:CheckerAppendice}.
   For $T_{i,t_0}$, we need to use this same lemma,
   together with \cref{eq:MasterEquationPasMatricielle} and the analyticity of $D^{e, \eps }_{i,j}$ 
   at $\rho$ (\cref{prop:SumOfD}).

	We now establish the asymptotics of these series near $\rho$.
	Recall \cref{eq:MasterEquation}:
    \[
		\mathbf T^\star_{\patterntree}
		=\TT_\to^\star \mathbb{D}^{e_1,\eps_1} \TT_\to^\star \mathbb{D}^{e_2,\eps_2} \dots \TT_\to^\star \mathbb{D}^{e_k,\eps_k} \mathbf T^\star \label{eq:ProduitDegueu}.
        \]

	We can plug in the value of the series $D^{ e ,\eps}_{i,j}$'s, since they converge at $\rho$ from \cref{prop:SumOfD},
	and the asymptotics  near $\rho$ of $\TT_\to^\star$ and $\mathbf T^\star$ (see \cref{eq:Asymp_Tfleche_linear,eq:Asymp_T_linear}). We get
	\begin{align}
      \mathbf T^\star_{\patterntree} &\stackrel{z\to\rho}{\sim}
		\frac {1} {(\rho-z)^{k+1}} 
		\left(\frac { 1}{\transpose{\mathbf u}  (\mathbb M^\star)'(\rho) \mathbf v}\right) \mathbf v\transpose{\mathbf u}\ 
		\mathbb{D}^{e_1,\eps_1}(\rho) \left( \frac {1}{\transpose{\mathbf u} (\mathbb M^\star)'(\rho) \mathbf v}\right) \mathbf v\transpose{\mathbf u}\ \mathbb{D}^{e_2,\eps_2}(\rho) \nonumber \\ 	
  &\ \hspace{9cm} \dots \mathbb{D}^{e_k,\eps_k}(\rho)\ \mathbf v \left(\frac {\transpose{\mathbf u}\mathbf V^\star(\rho)}{\transpose{\mathbf u} (\mathbb M^\star)'(\rho) \mathbf v}\right) \nonumber\\
		&= \frac {1} {(\rho-z)^{k+1}} 
		\left(\frac { 1}{\transpose{\mathbf u} (\mathbb M^\star)'(\rho) \mathbf v}\right) \mathbf v
		 \left(\prod_{\ell=1}^{k}\frac{ \transpose{\mathbf u}\mathbb{D}^{e_\ell,\eps_\ell}(\rho)\mathbf v}{\transpose{\mathbf u} (\mathbb M^\star)'(\rho) \mathbf v}\right)
		\transpose{\mathbf u} \mathbf V^\star(\rho)\nonumber\\
		&= \frac {1} {(\rho-z)^{k+1}} \frac {\transpose{\mathbf u} \mathbf V^\star(\rho)}{\transpose{\mathbf u} (\mathbb M^\star)'(\rho) \mathbf v}
		\left(\prod_{\ell=1}^{k} p^{e_\ell}_{\eps_\ell}\right) \ \mathbf v.
        \label{eq:Asympt_Tto_Linear}
	\end{align}

	We turn to $\frac {T_i^{(k)}}{k!}$. From \cref{eq:Asymp_T_linear},
 	applying singular differentiation \cite[Thm. VI.8 p. 419]{Violet} to each entry of $\mathbf{T}^\star$ we obtain
	\begin{align*}
	\frac 1 {k!} (\mathbf{T}^\star)^{(k)}(z)&\stackrel{z\to\rho}{\sim} \frac {1} {(\rho-z)^{k+1}} \left(\frac {\transpose{\mathbf u} \mathbf V^\star(\rho)}{\transpose{\mathbf u} (\mathbb M^\star)'(\rho) \mathbf v}\right) \ \mathbf{v}.
	\end{align*}
	\smallskip

Applying the transfer theorem (\cref{thm:transfert}) to $T_{i,t_0}$ and $\tfrac 1 {k!} T_i^{(k)}$ yields
\[\frac {[z^{n-k}] T_{i,t_0}}{[z^{n-k}] \frac 1 {k!} T_i^{(k)}} \xrightarrow[n\to\infty]{} \prod_{\ell=1}^{k} p^{e_\ell}_{\eps_\ell},\]
concluding the proof.
\end{proof}

\subsection{Permutations induced by the $X$-permuton}\label{SsSec:MarginalesXPermutons}

The $X$-permuton $\mu^{X}_{\mathbf p}$ was defined in \cref{Def:XPermuton}.
In this section we describe the permutations induced by the $X$-permuton,
\emph{i.e.}, for each $k \ge 1$, the random permutation formed by $k$
independent points in $[0,1]^2$ with common distribution $\mu^{X}_{\mathbf p}$. 

For a set $\{(x_i,y_i), 1\leq i \leq k\}$ of $k$ points in the unit square 
(assumed to have pairwise distinct $x$- (resp. $y$-)coordinates), we denote 
by $\perm(\{(x_i,y_i), 1\leq i \leq k\})$ the permutation whose diagram is the (suitably normalized) set of these points. 

We start by a lemma, illustrated in \cref{fig:XPermuton_CaterpillarInduit}.
\begin{lemma}
  \label{lem:PermPointsX}
	Let $(e_1,\eps_1)\ldots (e_k,\eps_k)$ be the code word of a caterpillar $\patterntree$.
    Fix $(a,b)\in (0,1)^2$, $0<u_1<\ldots<u_k<1$ and set
	\begin{equation}
	(x_i,y_i) = (1-u_i)z^{e_i}_{\eps_i} + u_i (a,b),
	\quad 1\leq i \leq k
	\label{eq:X_permuton}
	\end{equation}
	Then $\perm(\{(x_i,y_i), 1\leq i \leq k\}) = \perm(\Reduc(\patterntree))$.
\end{lemma}

\begin{figure}[htbp]
	\begin{center}
		\includegraphics[width=15cm]{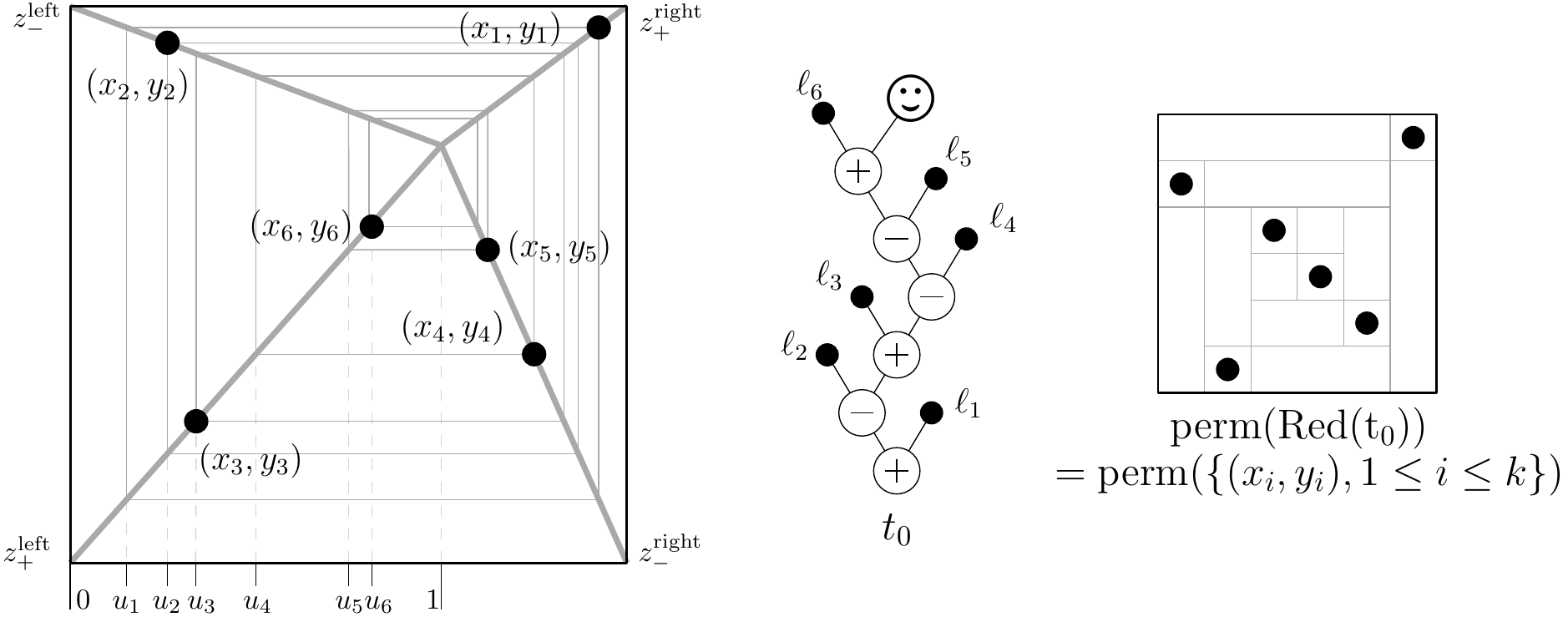}
	\end{center}
	\caption{An example illustrating \cref{lem:PermPointsX}, with a caterpillar of code word 
	 $((\droite,+), (\gauche,-), (\gauche,+), (\droite,-) ,(\droite,-), (\gauche,+))$.
	 \label{fig:XPermuton_CaterpillarInduit}}
\end{figure}
\begin{proof}
	Let $\tau = \perm(\{(x_i,y_i), 1\leq i \leq k\})$. Let $\alpha$ be the permutation such that $x_{\alpha(1)}<\ldots <x_{\alpha(k)}$. Then by definition 
	\[\forall 1\leq i<j \leq k,\quad \tau(i)>\tau(j) \iff y_{\alpha(i)}>y_{\alpha(j)}.\]
	By case analysis, from \cref{eq:X_permuton}, we can prove that 
	\begin{equation}
	\forall \ 1\leq i<j \leq k, \quad (e_i = \gauche)\iff x_i<x_j \iff \alpha^{-1}(i)<\alpha^{-1}(j).
	\label{eq:X_permuton_tech1}
	\end{equation}
	Similarly, and again by case analysis from \cref{eq:X_permuton}, we can prove that for 
    $1\leq i<j\leq k$, we have  $$(\eps_i = -) \iff (x_j-x_i)(y_j-y_i)<0.$$ Hence for $1\leq i<j\leq k$, $\eps_{\min(\alpha(i), \alpha(j))} = - $ if and only if $(x_{\alpha(j)}-x_{\alpha(i)})(y_{\alpha(j)}-y_{\alpha(i)})<0$, which reduces to  $y_{\alpha(j)}<y_{\alpha(i)}$. All in all, we have shown
	\begin{equation}
	\forall \ 1\leq i<j\leq k,\quad \eps_{\min(\alpha(i), \alpha(j))} = - \iff \tau(i)> \tau(j).
	\label{eq:X_permuton_tech2}
	\end{equation}
	
	Now let $\pi = \perm(\Reduc(\patterntree))$ and denote $\ell_{\gamma(1)},\ldots, \ell_{\gamma(k)}$ the reordering of the leaves of $t_0$ according to the depth-first search. 
    By definition of $t_0$, for $1\leq i<j \leq k$, the following equivalence holds:
    $(\gamma^{-1}(i)<\gamma^{-1}(j)) \iff (e_i = \gauche)$. Together with \cref{eq:X_permuton_tech1}, this shows $\gamma = \alpha$.
	
	Finally, looking at the way the permutation $\pi$ is constructed, we see that for $1\leq i < j \leq k$, $\pi(j)<\pi(i)$ if and only if there is a sign $\ominus$ on the first common ancestor $v_{\min(\gamma(i), \gamma(j))}$ of $\ell_{\gamma(i)}$ and $\ell_{\gamma(j)}$, if and only if $\eps_{\min(\gamma(i), \gamma(j))} = -$. Since $\gamma = \alpha$, together with \cref{eq:X_permuton_tech2}, this shows $\pi = \tau$ , \emph{i.e.} the lemma.
\end{proof}
 
Recall from \cref{ssec:permutons_intro} some notation regarding permutons.
For a fixed permuton ${\mu}$ and a fixed integer $k$, we denote by $(\XX,\YY)$ a $k$-tuple of i.i.d. points distributed according to ${\mu}$. This $k$-tuple, seen as a \emph{set} of points in the unit square, induces a permutation $\perm(\{(\bm x_i,\bm y_i), 1\leq i \leq k\})$ that we denote  $\InducedPerm_k(\mu)$.

\begin{proposition}
\label{prop:MarginalsXPermuton}
For every $k\geq 1$, 
we have
$$
\InducedPerm_k(\mu^{X}_{\mathbf p}) \stackrel{\mathrm{(d)}}{=}\perm(\Reduc(\bm t_0)),
$$
where $\bm t_0$ is a random caterpillar whose code word is a $k$-uple of i.i.d. random variables of distribution $\mathbf p$.
\end{proposition}
The fact that $\InducedPerm_k(\mu^{X}_{\mathbf p})$ is a permutation encoded by the reduced tree of a caterpillar is illustrated in \cref{fig:XPermuton_CaterpillarInduit}.

\begin{proof}
	Because of the construction of ${\mu}^X_{\mathbf p}$, an i.i.d sequence $((\bm x_1,\bm y_1),\ldots, (\bm x_k,\bm y_k))$ drawn according to ${\mu}^X_{\mathbf p}$ can be represented as 
	\[(\bm x_i,\bm y_i) = (1-\bm u_i)z^{\bm e_i}_{\bm \eps_i} + \bm u_i (a,b), \quad 1\leq i \leq k,\]
	where $\bm u_1, \ldots,\bm u_k$ are uniform in $[0,1]$,  
    $(\bm e_1, \bm \eps_1),\ldots ,(\bm e_k, \bm \eps_k)$ are random variables 
    according to the measure $\mathbf p$, all of these being independent from each other.
    By definition  $\InducedPerm_k(\mu^X_{\mathbf p})$ is distributed like the permutation $\perm(\{(\bm x_i,\bm y_i), 1\leq i \leq k\})$.
    
    Consider the permutation $\bm \sigma$ such that $\bm u_{\bm \sigma(1)}<\ldots < \bm u_{\bm \sigma(k)}$.
    Clearly,
    \[\perm(\{(\bm x_i,\bm y_i), 1\leq i \leq k\})=\perm(\{(\bm x_{\bm \sigma(i)},\bm y_{\bm \sigma(i)}), 1\leq i \leq k\}),\]
    and from \cref{lem:PermPointsX}, this is the permutation associated to the caterpillar whose code word is
    $(\bm e_{\bm \sigma(1)}, \bm \eps_{\bm \sigma(1)})\ldots (\bm e_{\bm \sigma(k)}, \bm \eps_{\bm \sigma(k)})$.
But the sequence $((\bm e_{\bm \sigma(i)},\bm \eps_{\bm \sigma(i)}))_{1\leq i \leq k}$ is 
    an i.i.d. sample of the measure $\mathbf p$.
    Indeed, it is a shuffling of an i.i.d. sequence by the independent random permutation $\bm \sigma$.
    This concludes the proof.
\end{proof}

\subsection{Back to permutations and conclusion of the proof of \cref{Th:linearCase}}
  We can now conclude the proof of the main theorem for the essentially linear case. 
\begin{proof}[Conclusion of the proof of \cref{Th:linearCase}]
Consider a tree specification \eqref{eq:Specif} satisfying the hypotheses of \cref{Th:linearCase}. 
Let $i\in I^\star$ be the index of a critical family and let $k\geq 1$.
Finally, we let
$\tkin$ the random subtree induced by $k$ uniform random leaves
in a uniform random tree with $n$ leaves in $\TTT_i$.
Comparing with the notation of \cref{prop:probaCaterpillars}, 
we have $\tkin=\Reduc(\tkinhead)$.

Moreover, we denote by $\bm \sigma_n$ a uniform permutation of size $n$ in $\TTT_i$
and $\bm I_{n,k}$ an independent uniform subset of $[1,n]$ of size $k$.
Thanks to \cref{lem:DiagrammeCommutatif}, we have
\[\pat_{ {\bm I}_{n,k}}(\bm\sigma_n)= \perm(\tkin).\]

According to \cref{prop:probaCaterpillars}, as $n\to\infty$,
$\tkinhead$ 
 converges in distribution to the caterpillar $\bm t_0$,
whose code word is given by a $k$-tuple of i.i.d. random variables of distribution $\mathbf p$.
Therefore we have the following convergence in distribution
$$
\pat_{{\bm I}_{n,k}}(\bm\sigma_n)= \perm(\tkin) = \perm(\Reduc(\tkinhead))
\stackrel{n\to +\infty}{\longrightarrow} \perm(\Reduc(\bm t_0)).
$$
\cref{Th:linearCase} then follows from \cref{thm:randompermutonthm} (characterization of convergence of random permutons) and \cref{prop:MarginalsXPermuton} (giving the distribution of $\InducedPerm_k(\mu^{X}_{\mathbf p})$).
\end{proof}

\section{The essentially branching case}\label{sec:branching}

\subsection{Tree decomposition}
Following the same strategy as in Section~\ref{Sec:ProofsLinear}, 
we are first aiming at an analogue of \cref{prop:EnumerationCaterpillar}, 
which gives the generating function of trees of type $\TTT_i$ 
with $k$ marked leaves inducing a given subtree. 
This will be obtained in \cref{prop:EnumerationBranchingCase} below.
To state it, we need to consider \emph{doubly blossoming trees} 
in addition to the (simply) blossoming trees already defined for the essentially linear case. 

\begin{definition}\label{def:doublyBlossomingTree}
	For $0\leq i,j,j'\leq d$, we define $\HHH_{\to i}^{j,j'}$ as the 
	family of trees $t$ with an ordered pair of marked leaves $(\ell_1, \ell_2)$, called the {\em first} and {\em second blossoms}, 
	which are required to be children of the root of $t$, 
	such that the tree obtained by replacing $\ell_1$ by a tree of $\TTT_j$ and $\ell_2$ 
	by a tree of $\TTT_{j'}$ belongs to $\TTT_i$, 
	with the additional condition that the type in $\TTT_i$ of the node that used to be $\ell_1$ (resp. $\ell_2$) is $j$ (resp. $j'$).
\end{definition}
Similarly to the case of (simply) blossoming trees, 
in general, a tree in  $\HHH_{\to i}^{j,j'}$ does not belong to  $\TTT_{i}$.

\begin{definition}\label{def:Eij}
	Let $i,j,j' \in I^\star$ and $\eps\in\{+,-\}$. The class $\mathcal E_{ijj'}^{\eps}$ is the class of doubly blossoming trees in the class $\mathcal H_{\to i}^{j,j'}$ with the following additional conditions:
	\begin{itemize}
		\item the first blossom is to the left of the second blossom;
		\item the pattern induced by the two blossoms on the permutation labeling the root is $12$ if $\eps=+$ and $21$ if $\eps=-$.
	\end{itemize}
We denote by $E_{ijj'}^{\eps}(z)$ the corresponding generating series.
\end{definition}

\begin{proposition}
	\label{prop:SumOfE}
For every $i,j,j'\in I^\star$, we have that
	\begin{equation*}
	E_{ijj'}^{+}(z) + E_{ijj'}^{-}(z) + E_{ij'j}^{+}(z) + E_{ij'j}^{-}(z)\ =\ H_{\to i}^{j,j'}(z)\ =\
        \frac {\partial^2 F_i(y_0, \dots, y_d)}{\partial y_j \partial y_{j'}} \Big|_{(T_0(z), \dots, T_d(z))}.
	\end{equation*}
\end{proposition}

In addition, in the essentially branching case (see \cref{Def:EssBranching}) it holds that at least one of the $H_{\to i}^{j,j'}$ for $i,j,j'\in I^\star$ is a nonzero series. 

\begin{proof}
The first equality is obtained by partitioning $\mathcal H_{\to i}^{j,j'}$ into four parts,
depending on the position (left or right) of the first blossom w.r.t. the second blossom, and on the pattern ($12$ or $21$) induced by the blossoms.

The second equality comes from the definition of $\mathcal H_{\to i}^{j,j'}$ and the specification \eqref{eq:Specif} (the arguments are similar to the proof of \cref{prop:SumOfD}).
\end{proof}

Let us fix a signed binary tree $\patterntree$ with $k$ leaves. 
Recall that for us, binary indicates that every internal node has degree \emph{exactly} 2. 
Recall that $\Internal{\patterntree}$ (resp. $\Leaves{\patterntree}$) 
denotes the set of internal nodes (resp. leaves) of $\patterntree$. 
For $v\in \Internal{\patterntree}$ we set
\begin{itemize}
	\item $\eps(v)$ the sign labeling the node $v$ in $\patterntree$;
	\item $\mathfrak l(v) \in \Internal{\patterntree} \uplus \Leaves{\patterntree}$ the left child of $v$;
	\item $\mathfrak r(v) \in \Internal{\patterntree} \uplus \Leaves{\patterntree}$ its right child.
\end{itemize}
We also use the convention that $\varnothing\in \Internal{\patterntree}$ denotes the root of $\patterntree$.

\begin{definition}
	For $i\in I^\star$, let $\TTT_{i,\patterntree}$ be the class of trees in $\TTT_i$ with $k$ unordered marked leaves, such that
	\begin{itemize}
		\item they induce the subtree $\patterntree$,
		\item for every marked leaf $\ell$, the first critical ancestor of $\ell$ is strictly closer to $\ell$ than any first common ancestor of $\ell$ and another marked leaf.
	\end{itemize}
	\label{Def:TTT_patterntreeBranching}
\end{definition}

Note that if a marked tree $(t,(\ell_1,\dots,\ell_k))\in \TTT_{i,\patterntree}$, then $\pat_{\ell_1,\dots,\ell_k}(\perm(t)) = \perm(\patterntree)$.

\begin{proposition}
\label{prop:EnumerationBranchingCase}
	We have, for every $i_0\in I^\star$,
	\begin{multline}
      T_{i_0,\patterntree} = \sum_{
	i,j,j' \in {I^\star}^{\Internal{\patterntree}}
	}  \left[ T_{\to i_0}^{i(\varnothing)}
	\prod_{\substack{v\in \Internal{\patterntree}\\\mathfrak l(v)\in \Internal{\patterntree}}} T_{\to j(v)}^{i(\mathfrak l(v))}
	\prod_{\substack{v\in \Internal{\patterntree}\\\mathfrak r(v)\in \Internal{\patterntree}}} T_{\to j'(v)}^{i(\mathfrak r(v))} \right.\\
    \left. \cdot
	\prod_{\substack{v\in \Internal{\patterntree}\\\mathfrak l(v)\in \Leaves{\patterntree}}} T_{j(v)}'
	\prod_{\substack{v\in \Internal{\patterntree}\\\mathfrak r(v)\in \Leaves{\patterntree}}} T_{j'(v)}'
	\prod_{v\in \Internal{\patterntree}}E^{\eps(v)}_{i(v)j(v)j'(v)}\right].
	\label{eq:MasterEquationBranching}
	\end{multline}
	The above sum runs over all triples $(i,j,j')$ of functions from $\Internal{\patterntree}$ to $ I^\star$.
\end{proposition}

\begin{figure}[thbp]
	\begin{center}
		\includegraphics[width=14cm]{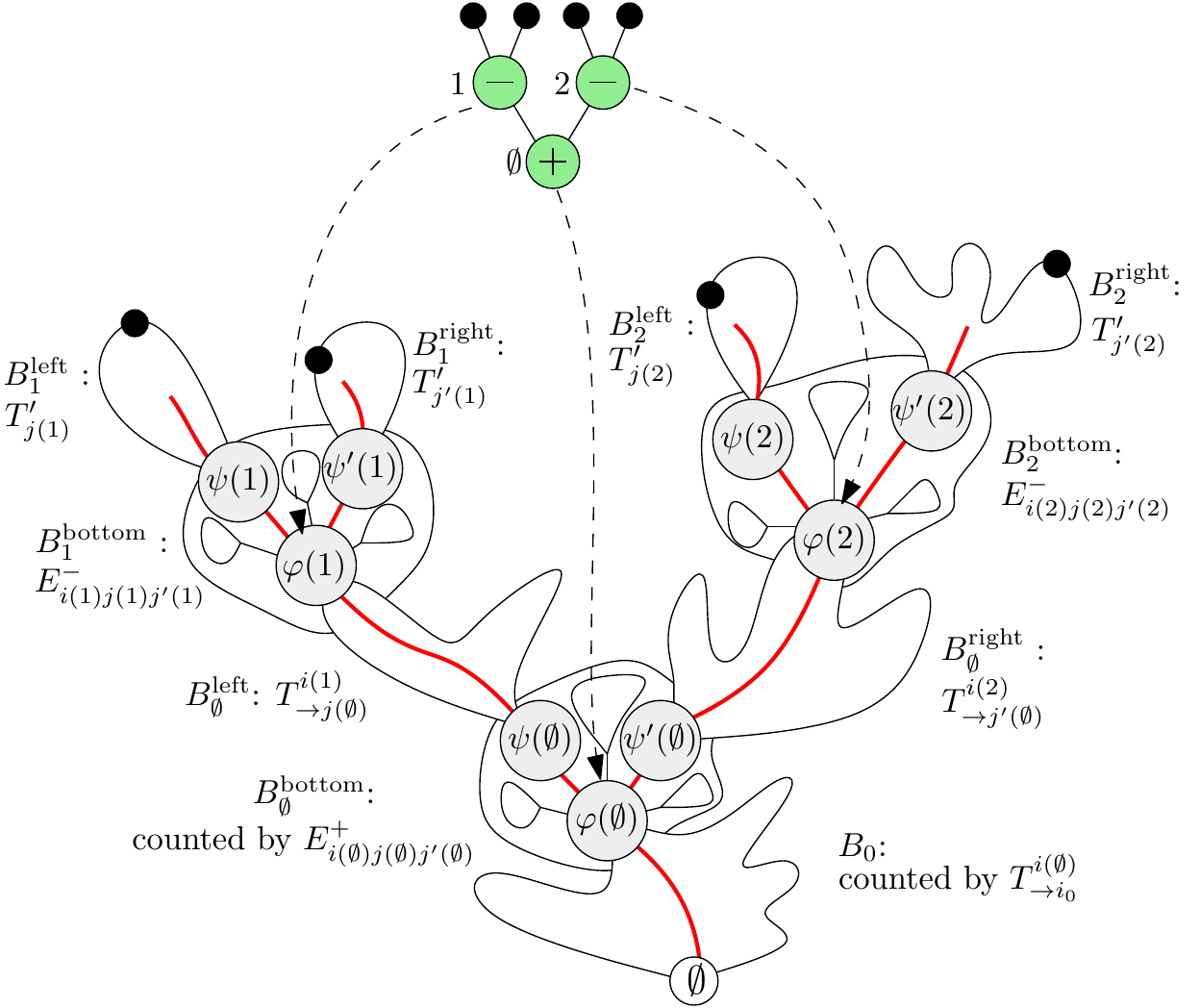}
		\caption{Top: A signed binary tree $\patterntree$ with $k=4$ leaves.
			Bottom: A schematic view of a tree with $k$ marked leaves in $\TTT_{i,\patterntree}$.
			As in \cref{Fig:SqueletteCompliqueCaterpillar}, the red paths consist only of nodes of critical type.
			\label{Fig:SqueletteCompliqueBinaire}}
	\end{center}
\end{figure}

\begin{proof}\emph{(The main notation of the proof is summarized in \cref{Fig:SqueletteCompliqueBinaire}.)}\\
	Consider a marked tree $t\in \TTT_{i_0,\patterntree}$.
	Then every interval node $v$ (resp. leaf $\ell$) in $t_0$ is in correspondence 
	with some interval node (resp. some marked leaf) of $t$, 
    which we denote by $\varphi(v)$ (resp. $\varphi(\ell)$).
	By definition of induced subtree, when $v$ is an internal node,
	$\varphi(v)$ is the first common ancestor of 
	$\varphi(\mathfrak l(v))$ and $\varphi(\mathfrak r(v))$.
	Denote by $\psi(v)$ the child of $\varphi(v)$ which is an ancestor of (and possibly equal to) $\varphi(\mathfrak l(v))$. 
	Similarly, $\psi'(v)$ is the child of $\varphi(v)$ which is an ancestor of $\varphi(\mathfrak r(v))$. 
	By definition of induced subtree, the pattern of $\varphi(v)$ induced by the elements corresponding to $\psi(v)$ and $\psi'(v)$ is $\eps(v)$ 
	(here and in what follows we identify $12$ and $+$ on one hand 
    and $21$ and $-$ on the other hand).
	
	Now for every $v\in \Internal{\patterntree}$, let $i(v)$ (respectively $j(v)$, $j'(v)$)
	be the type of $\varphi(v)$ (respectively $\psi(v)$, $\psi(v')$) in $t$.
	Those types are necessarily critical because of the second assumption in \cref{Def:TTT_patterntreeBranching}.
	
	We now decompose $t$ successively, cutting at all nodes $\varphi(v)$, $\psi(v)$ and $\psi'(v)$ for $v\in \Internal{\patterntree}$. 
	This is similar to the construction in the proof of  \cref{prop:EnumerationCaterpillar}, and we use the same notational conventions. 
	Since there are $3(k-1)$ cuts, we end up with $3(k-1)+1$ pieces.
	
	\begin{enumerate}
		\item We denote by $B_0$ the piece containing the root of $t$. 
          Concretely, $B_0$ is obtained from $t$ (which has type $i_0$) by replacing
          the fringe subtree rooted at $\varphi(\varnothing)$ (which has type $i(\varnothing)$ in $t$)
          by a blossom. 
		 Hence $B_0 \in \TTT_{\to i_0}^{i(\varnothing)}$. 
		\item For all $v\in \Internal{\patterntree}$, we denote by $B_{v}^{\gauche}$ the piece rooted at $\psi(v)$. There are then two possible cases. 
		\begin{enumerate}
			\item Either $\mathfrak l(v) \in \Internal{\patterntree}$, in which case $B_{v}^{\gauche}$ is 
			the fringe subtree of $t$ rooted at $\psi(v)$ in which the fringe subtree rooted at $\varphi(\mathfrak l(v))$ has been replaced by a blossom. 
			Hence $B_{v}^{\gauche} \in \TTT_{\to j(v)}^{i(\mathfrak l(v))}$.
			\item Or $\mathfrak l(v) \in \Leaves{\patterntree}$, in which case $B_{v}^{\gauche}$ 
			is simply the fringe subtree of $t$ rooted at $\psi(v)$;
            this tree contains one marked leaf, namely $\varphi(\mathfrak l(v))$.
			Hence $B_{v}^{\gauche}$ belongs to the family of marked trees of type $j(v)$;
            this family is counted by the series  $T_{j(v)}'$.
		\end{enumerate}
		\item Similarly for $v\in \Internal{\patterntree}$, we denote by $B_{v}^{\droite}$ 
          the piece rooted at $\psi'(v)$. Then, 
		\begin{enumerate}
			\item either $\mathfrak r(v) \in \Internal{\patterntree}$ and $B_{v}^{\droite} \in \TTT_{\to j'(v)}^{i(\mathfrak r(v))}$,
			\item or $\mathfrak r(v) \in \Leaves{\patterntree}$ and $B_{v}^{\droite} \in \TTT_{j'(v)}'$.
		\end{enumerate}
		\item For all $v\in \Internal{\patterntree}$, 
          we denote by $B_{v}^{\mathrm{bottom}}$ the piece rooted at $\varphi(v)$. 
		This piece is exactly
		the fringe subtree of $t$ rooted at $\varphi(v)$ in which the fringe subtrees rooted at $\psi(v)$ and $\psi'(v)$ have been replaced by blossoms. 
		Hence it has type $i(v)$ at the root, and contains two blossoms that are children of the root, the left one being of type $j(v)$ and the right one of type $j'(v)$. These two blossoms induce the permutation $\eps(v)$ on the root. So $B_{v}^{\mathrm{bottom}} \in \mathcal E^{\eps(v)}_{i(v)j(v)j'(v)}$.
	\end{enumerate}
	Summing up, we have associated to each tree $t\in \TTT_{i_0,\patterntree}$ the data consisting of 
	$(i,j,j')$ (where $i$, $j$ and $j'$ are three functions from $\Internal{\patterntree}$ to $ I^\star$)
	and the tuple of pieces 
	\begin{equation*}
	\Big(B_0,\;
	(B_{v}^{\gauche})_{\substack{v\in \Internal{\patterntree}\\
			\mathfrak l(v)\in \Internal{\patterntree}}},\;
	(B_{v}^{\droite})_{\substack{v\in \Internal{\patterntree}\\
			\mathfrak r(v)\in \Internal{\patterntree}}},\;
	(B_{v}^{\gauche})_{\substack{v\in \Internal{\patterntree}\\
			\mathfrak l(v)\in \Leaves{\patterntree}}},\;
	(B_{v}^{\droite})_{\substack{v\in \Internal{\patterntree}\\
			\mathfrak r(v)\in \Leaves{\patterntree}}},\;
	(B_{v}^{\mathrm{bottom}})_{v\in \Internal{\patterntree}} 
	\Big), 
	\end{equation*}  
	The map associating to $t$ its tuple of pieces is size-preserving, because each unmarked leaf in $t$ becomes an unmarked leaf in one of the pieces,
	and no other unmarked leaf is created (recall that blossoms and marked leaves do not contribute to the size).
	
	We denote by $\Omega$ the disjoint union, over triples $(i,j,j')$ of functions from $\Internal{\patterntree}$ to $ I^\star$, 
	of $\Omega_{i,j,j'}$. 
	In the above described procedure which ``cuts'' $t$ into pieces, no information is lost. 
	Namely, any $t\in \TTT_{i_0,\patterntree}$ can be recovered unambiguously from its associated tuple of trees by the simple inverse ``gluing'' procedure.
Moreover, performing this gluing procedure
from an arbitrary element of $\Omega$ yields a tree in $\mathcal T_{i_0}$
with $k$ marked leaves which induce $t_0$.
This tree belongs to $\TTT_{i_0,t_0}$:
indeed, the second condition in \cref{Def:TTT_patterntreeBranching} is satisfied,
because the pieces in  
	\[(B_{v}^{\gauche})_{\substack{v\in \Internal{\patterntree}\\
    \mathfrak l(v)\in \Leaves{\patterntree}}} \text{ and }
	(B_{v}^{\droite})_{\substack{v\in \Internal{\patterntree}\\
			\mathfrak r(v)\in \Leaves{\patterntree}}}\]
	have a critical type at the root.

	This shows that $\TTT_{i_0,\patterntree} \to \Omega$ is a size-preserving bijection, of which \cref{eq:MasterEquationBranching} is the translation in terms of series.
\end{proof}

\subsection{Asymptotics of the main series} 
We want the asymptotic behavior of the series that are the entries of $\mathbf{T}^\star(z)$,$(\mathbf T^\star)'(z)$ and $\mathbb T_\to^\star(z)$.
Recall from \cref{eq:generic_system} that the entries of $\mathbf T^\star$ are solutions of the system $\mathbf T^\star(z) = \mathbf \Phi(z,\mathbf T^\star(z))$. 
Recall also from \cref{eq:Tfleche} the identity  $\mathbb T_\to^\star(z) = (\Id - \mathbb M^\star(z,\mathbf T^\star(z)))^{-1}$.

The following lemma is a consequence of a general result on nonlinear systems proved in the appendix (\cref{Thm:DLW}). Recall that $\rho$ is the common radius of convergence of the critical series.

\begin{lemma}\label{lem:CheckerAppendiceBranching}
 Assume that the specification \eqref{eq:Specif} is essentially branching and satisfies hypotheses (SC) and (AR). Assume also that one of the series $T_i$, critical or subcritical, is aperiodic.

Then all entries of $(\Id - \mathbb M^\star(z,\mathbf T^\star(z)))^{-1}$ and $\mathbf T^\star(z)$ are analytic on a $\Delta$-domain at $\rho$.
Moreover, the matrix $\mathbb M^\star(\rho,  \mathbf T^\star(\rho))$ is irreducible and has Perron eigenvalue $1$, 
and denoting $\mathbf u$ and $\mathbf v$ the corresponding left and right positive eigenvectors  normalized so that $\transpose{\mathbf u} \mathbf v=1$,
we have the following asymptotics\footnote{In the above equations $\sim$ stands for coefficient-wise asymptotic equivalence.} near $\rho$:
\begin{align}
\mathbf T^\star(z) &= \mathbf T^\star(\rho) - \frac{ \beta \mathbf v}{ \zeta} \sqrt{\rho - z} + o(\sqrt{\rho - z})
\label{eq:AsympTstarBranching}\\
(\mathbf T^\star)'(z) &\sim \frac{ \beta \mathbf v}{2 \zeta \sqrt{\rho - z}}
\label{eq:AsympTstar'Branching}\\ 
\mathbb T_\to^\star(z)=(\Id - \mathbb M^\star(z,\mathbf T^\star(z)))^{-1} &\sim \frac{ \mathbf v\transpose{\ \mathbf u}}{2 {\beta \zeta} \sqrt{\rho - z}}
\label{eq:AsympTstarflecheBranching}
\end{align}
where
\begin{equation*}
Z = \frac 1 2 \sum_{i,j,j'\in I^\star} u_iv_jv_{j'} H_{\to i}^{j,j'}(\rho), \quad 
\zeta = \sqrt{Z}, 
\end{equation*}
and $\beta>0$ is some computable constant.
\end{lemma}

\begin{proof}[Proof of \cref{lem:CheckerAppendiceBranching}]
In order to apply \cref{Thm:DLW}, we check that hypotheses (i)-(iii) and (v) of this theorem hold.
\begin{itemize}
\item Assumption (i) is granted by the form of the specification \eqref{eq:Specif}. If $\mathbf \Phi(z,\mathbf 0)$ were zero, then the specification would be empty, negating for instance the aperiodicity assumption.
\item Assumption (ii) is the essentially branching assumption.
\item Assumption (iii) is Hypothesis (SC).
\item Assumption (v) is Hypothesis (AR).
\end{itemize}
We also have to check aperiodicity of all critical series. By assumption at least one series $T_i$ (critical or subcritical) is aperiodic.
By \cref{lem:MonotonieAperiodicite} this ensures that every critical series is aperiodic since Hypothesis (SC) holds.

We conclude the proof applying \cref{Thm:DLW}, using also \cref{prop:SumOfE} to obtain the expression of $Z$ involving the $H_{\to i}^{j,j'}$.
\end{proof}

\begin{corollary}
  \label{corol:SumOfE}
 Under hypothesis (AR), 
 each of the series $E_{ijj'}^{\eps}$ and $H_{\to i}^{j,j'}$ 
 have radius of convergence $\rho$,
 are convergent at $\rho$
 and are $\Delta$-analytic at $\rho$. 
\end{corollary}
\begin{proof}
The second equality in \cref{prop:SumOfE}
shows that
$H_{\to i}^{j,j'}(z)$ is of the form $Q_{i}^{j,j'}(z,\mathbf T^{\star}(z))$,
where
\[Q_{i}^{j,j'}(z,\mathbf y^\star) = \frac {\partial^2 F_i(y_0, \dots, y_c,T_{c+1}(z),\ldots,T_d(z))}{\partial y_j \partial y_{j'}}. \]
From the previous lemma, $T^\star(\rho)$ has radius of convergence $\rho$,
is convergent at $\rho$ and
 $\Delta$-analytic at $\rho$.
From hypothesis (HR), the above function $Q_{i}^{j,j'}$ is analytic around $(\rho,\mathbf T^\star(\rho))$.
This proves the corollary for $H_{\to i}^{j,j'}$ (for all $i,j,j'$ in $I^\star$).

Using their combinatorial definition, we see that the series $E^\eps_{ijj'}$
are also of the form $R^\eps_{ijj'}(z,\mathbf T^{\star}(z))$,
where $R^\eps_{ijj'}$ is coefficient-wise dominated by $Q_{i}^{j,j'}$.
In particular $R^\eps_{ijj'}$ is analytic around $(\rho,\mathbf T^\star(\rho))$
and the same argument as above prove the corollary for $E^\eps_{ijj'}$.
\end{proof}

\subsection{Probabilities of tree patterns} \label{section:TreePatternsbranching}
We now set 
\begin{equation}\label{eq:DefParametrePermutonBrownien}
\begin{cases}
p_+ &= \frac 1 Z \sum_{i,j,j'\in I^\star} E^+_{ijj'}(\rho)u_iv_jv_{j'}\\
p_- &= \frac 1 Z \sum_{i,j,j'\in I^\star} E^-_{ijj'}(\rho)u_iv_jv_{j'},
\end{cases}
\end{equation}
where $E^\eps_{ijj'}$ are defined in \cref{def:Eij}
and $u_i, v_j$ and $Z$ are defined in \cref{lem:CheckerAppendiceBranching}.

Thanks to \cref{prop:SumOfE}, $p_+ + p_- = 1$.
\begin{proposition}\label{prop:ProbaBinaire}
We assume that we are in the essentially branching case, that Hypotheses (SC) and (AR) are satisfied,
and that at least one series (either critical or subcritical) is aperiodic.

Let $t_0$ be a signed binary tree with $k$ leaves.
For $i\in I^\star$, we consider a
uniform tree with $n$ leaves in $\TTT_i$, with $k$ uniform 
marked leaves, and denote $\tkin$ the tree induced by these $k$ marked leaves.
We have, for all $i \in I^\star$,
\[
\proba(\tkin =t_0) \stackrel{n\to +\infty}{\longrightarrow} \frac 1 {\Cat_{k-1}}\prod_{v\in \Internal{\patterntree}} p_{\eps(v)}.
\]
\end{proposition}
In the above expression the limiting probabilities do not depend on $i$
and add up to $1$ (summing over all signed binary trees $t_0$ with $k$ leaves).
We deduce that $k$ marked leaves in a large uniform tree in $\TTT_i$ 
induce a binary tree with high probability\footnote{Throughout the paper,
we say that an event holds {\em with high probability} if its probability tends to $1$.}
when $n$ goes to infinity,
and that this signed binary tree is asymptotically distributed like a uniform binary tree with i.i.d. signs of bias $p_+$ (independently of the critical type $i$ that we consider).

\begin{proof}
  We fix $i$ throughout the proof.
  Similarly to the linear case (see \eqref{eq:prob_t0_linear}), we have,
  for any signed binary tree $t_0$,
  \begin{equation}
    \proba (\tkin = \patterntree) \ge 
          \frac {[z^{n-k}] T_{i,t_0}}{[z^{n-k}] \frac 1 {k!} T_i^{(k)}}.
          \label{eq:Proba_t0_Branching}
     \end{equation}
    Note that we only have an inequality.
    Indeed, because of the second item in \cref{Def:TTT_patterntreeBranching},
    the numerator only counts a subset of trees in $\mathcal T_i$
    with marked leaves inducing $t_0$.

    We want to apply the transfer theorem to the series $T_{i,\patterntree}$ and $T_i^{(k)}$.
	
    We first check that those series are analytic on a $\Delta$-domain at $\rho$.
    It is the case of $T_i$ (and all its derivatives) by \cref{lem:CheckerAppendiceBranching}.
    In addition, for all critical types $i$, $j$ and $j'$, the series $T_{\to j}^{i}$ and $E^\eps_{ijj'}$ 
    also are analytic on a $\Delta$-domain at $\rho$ (by \cref{lem:CheckerAppendiceBranching} and \cref{corol:SumOfE} respectively). 
    Hence by multiplication the same holds for $T_{i,t_0}$.
    
    We now look for asymptotic equivalents $T_{i,\patterntree}$ and $T_i^{(k)}$
    in a $\Delta$-neighborhood of $\rho$.
    For the former, we take \cref{eq:MasterEquationBranching} (p.\pageref{eq:MasterEquationBranching}),
	and plug in the values at $\rho$ of the convergent series 
    ($E^{\eps}_{ijj'}$ is convergent thanks to \cref{corol:SumOfE}) 
	and the asymptotics near $\rho$ of the divergent series given by \cref{eq:AsympTstar'Branching,eq:AsympTstarflecheBranching}, yielding
	\begin{align*}
	T_{i,\patterntree}(z) &\sim \frac {1}{(\sqrt{\rho-z})^{2k-1}}
	\sum_{i,j,j' \in {I^{\star \Internal{\patterntree}}}} \left[
	\frac{v_iu_{i(\varnothing)}}{2\beta\zeta} \prod_{\substack{v\in \Internal{\patterntree}\\l(v)\in \Internal{\patterntree}}}
    \frac{ v_{j(v)}u_{i(l(v))}}{2 {\beta \zeta}} \right. \\
	& \qquad \prod_{\substack{v\in \Internal{\patterntree}\\r(v)\in \Internal{\patterntree}}} \frac{ v_{j'(v)}u_{i(r(v))}}{2 {\beta \zeta}} 
	\left.\prod_{\substack{v\in \Internal{\patterntree}\\l(v)\in \Leaves{\patterntree}}} \frac{ \beta v_{j(v)}}{2 \zeta}
	\prod_{\substack{v\in \Internal{\patterntree}\\r(v)\in \Leaves{\patterntree}}} \frac{ \beta v_{j'(v)}}{2 \zeta}
    \prod_{v\in \Internal{\patterntree}}E^{\eps(v)}_{i(v)j(v)j'(v)}(\rho) \right]. 
  \end{align*}
    This can be simplified as
	\begin{align}
    \label{eq:Asympt_Tto_Branching}
	T_{i,\patterntree}(z)&=\frac {v_i\, {\beta}^{k-(k-1)}}{(\sqrt{\rho-z})^{2k-1}(2 {\zeta})^{2k-1}}
	\sum_{i,j,j' \in {I^{\star\Internal{\patterntree}}}}
	\prod_{v\in \Internal{\patterntree}}E^{\eps(v)}_{i(v)j(v)j'(v)}(\rho)\,u_{i(v)}v_{j(v)}v_{j'(v)}\\
	&=(\rho - z)^{1/2 - k}\ \frac {v_i {\beta}}{2^{2k-1} {\zeta Z^{k-1}}}
	\prod_{v\in \Internal{\patterntree}} \sum_{\substack{(i,j,j')\in I^{\star 3}}}
	E^{\eps(v)}_{ijj'}(\rho)\,u_{i}v_{j}v_{j'}\nonumber\\
	&=(\rho - z)^{1/2 - k}\ \frac {v_i {\beta}}{2^{2k-1} {\zeta}}
	\prod_{v\in \Internal{\patterntree}} p_{\eps(v)}.\nonumber
	\end{align}
    For $T_i^{(k)}$, we simply use singular differentiation of \cref{eq:AsympTstar'Branching}:
	\begin{align*}
	\frac {T_i^{(k)}}{k!} \sim (\rho - z)^{1/2 - k}\frac {v_i {\beta}}{2{\zeta}} \frac { \frac 12 \frac 32 \dots \frac {2k-3}{2}}{k!}
	&= (\rho - z)^{1/2 - k}\frac {v_i {\beta}}{{\zeta}} \frac {(2k-2)!}{2^k(k-1)!2^{k-1}k!}\\
	&= (\rho - z)^{1/2 - k}\frac {v_i {\beta}}{2^{2k-1}{\zeta}} {\Cat_{k-1}}.
	\end{align*}
    Applying the transfer theorem and using \cref{eq:Proba_t0_Branching} yields 
    \begin{equation}
      \liminf_{n \to \infty} \proba (\tkin = \patterntree)
    \ge \frac 1 {\Cat_{k-1}}\prod_{v\in \Internal{\patterntree}} p_{\eps(v)}. 
    \label{eq:Tech}
 \end{equation}
    Consider the sum over all signed binary tree $t_0$.
    The right-hand side sums to $1$ (recall that $p_++p_-=1$).
    On the other hand, for each fixed $n$, 
    the sum of $\proba (\tkin = \patterntree)$ over $t_0$ is at most $1$.
    This forces the infimum limit in \eqref{eq:Tech} to be an actual limit and the inequality
    to be an equality, proving the proposition.
\end{proof}

\subsection{Back to permutations and conclusion of the proof of \cref{Th:branchingCase}}
  We can now conclude the proof of the main theorem for the essentially branching case.  
\begin{proof}[Conclusion of the proof of \cref{Th:branchingCase}]
Consider a tree specification \eqref{eq:Specif} satisfying the hypotheses of \cref{Th:branchingCase}. 
Let $i\in I^\star$ be the index of a critical family and let $k\geq 1$.
We use the notation of \cref{prop:ProbaBinaire}, \emph{i.e.}
$\tkin$ the random subtree induced by $k$ uniform random leaves
in a uniform random tree with $n$ leaves in $\TTT_i$.

Moreover, we denote by $\bm \sigma_n$ a uniform permutation of size $n$ in $\TTT_i$
and $\bm I_{n,k}$ an independent uniform subset of $[1,n]$ of size $k$.
Thanks to \cref{lem:DiagrammeCommutatif}, we have
\[\pat_{ {\bm I}_{n,k}}(\bm\sigma_n)= \perm(\tkin ).\]

According to \cref{prop:ProbaBinaire},
$\tkin$ is binary with high probability as $n\to\infty$.
More precisely, $\tkin$ converges in distribution to $\bm b_k$,
where $\bm b_k$ is a uniform binary tree of size $k$ whose internal nodes carry i.i.d. signs with bias $p_+$.

Therefore we have the following convergence in distribution:
$$
\pat_{{\bm I}_{n,k}}(\bm\sigma_n)= \perm(\tkin) \stackrel{n\to +\infty}{\longrightarrow} \perm(\bm b_k).
$$

\cref{Th:branchingCase} then follows, thanks to \cref{thm:randompermutonthm} (characterization of convergence of random permutons) and \cref{Def:BrownianPermuton} (definition of the Brownian separable permuton).
\end{proof}

\section{Beyond the strongly connected case}\label{Sec:CouteauSuisse}

The goal of this section is to provide some tools
to describe the typical behavior of permutations in some families $\mathcal T_0$
having a tree-specification which does not satisfy Hypothesis (SC).
We do not provide general theorems, because of the many possible situations
that can occur. Instead, we present a method with some generic lemmas,
and illustrate it on examples.

Recall that $G^\star$ denotes the dependency graph of the tree-specification
restricted to the critical families.
We first find its strongly connected components with no edge pointing towards them.
Such a component has a vertex set $\{\mathcal T_i\}_{i \in J}$, for some $J \subset I^\star$.
Restricting the tree-specification to $\{\mathcal T_i\}_{i \in J} \uplus \{\mathcal T_i\}_{i \notin I^\star}$,
we obtain a new tree-specification satisfying Hypothesis (SC).
Then \cref{Th:linearCase} or \cref{Th:branchingCase}
gives us the limiting permuton of uniform permutations in any of the families $(\mathcal T_i)_{i \in J}$.

We now discuss the case of a strongly connected component $C=\{\mathcal T_i\}_{i \in J}$ of $G^\star$ that has some incoming edges,
originating from the strongly connected components $C_1,\dots,C_h$ of $G^\star$.
Consider a family $\mathcal T$ in $C$ and a tree in $\mathcal T$.
This tree consists of a root and fringe subtrees whose type are either subcritical 
or in one of the $C_j$'s or in $C$.
Recursively, we may assume that we know the limiting permuton of trees with types
in $C_1$, \ldots, $C_h$.
To deduce from there a limiting result for trees in $\mathcal T$,
we need to know if one of the fringe subtrees is giant 
or whether there are typically several macroscopic ones.

\subsection{Sufficient conditions for having a giant subtree}\label{Sec:SubTree}
  Let $\mathcal T_0, \mathcal T_1, \ldots, \mathcal T_r$ be combinatorial classes whose generating series have
 the same radius of convergence $\rho$ and are analytic on a $\Delta$-domain.
  We assume that $\mathcal T_0$ is related to $\mathcal T_1,\ldots,\mathcal T_r$ 
  through an equation $\mathcal T_0= \mathcal F(\mathcal Z,\mathcal T_1,\dots\mathcal ,T_r)$. 
  Here, $\mathcal Z$ is the class with a single combinatorial structure, of size $1$, classically called \emph{atom}; 
  in this paper, we rather refer to the atoms which constitute a combinatorial structure as its \emph{elements}. 
  In combinatorial terms, a structure in $\mathcal T_0$ is an $\mathcal F$-structure of size $s$
  and a list of $s$ substructures that are either atoms or belong to one of $\mathcal T_i$.
  This translates on generating series as $T_0= F(z,T_1,\dots,T_r)$.

  We now present two results which ensure, under appropriate assumptions, that $k$ uniformly marked
  elements in a large random uniform structure in $\mathcal T_0$  belong with high probability to the same $\mathcal T_i$
  substructure; in this case we speak of a \emph{giant substructure}.
  \begin{itemize}
    \item In our first lemma, the singularities of the $T_i$'s are simple poles 
      and $F$ is linear in the $T_i$'s (with coefficients depending on $z$).
    \item In our second lemma, the $T_i$'s have square-root singularities
      and $F$ is analytic on a neighborhood of $(\rho,T_1(\rho),\dots,T_r(\rho))$.
  \end{itemize}

\bigskip

Let us set up notation for the first lemma. We assume that the singularities of the generating series $T_1,\dots,T_r$
  are simple poles, namely, that for some reals $\delta_i$, 
  \begin{equation}
    T_i(z) = \frac{\delta_i}{\rho - z} + \O(1)\,\ ,\ 1 \leq i \leq r.
    \label{eq:asymp_Ti_polar}
  \end{equation}
  Assume in addition that 
  \begin{equation}
     F(z,T_1,\dots,T_r) = \sum_{i=1}^r G_i(z)\,T_i + G(z), 
     \label{eq:F_Affine}
   \end{equation}
  where $G(z)$ and the $G_i(z)$'s are convergent in $\rho$
  (they may be subcritical, or critical and convergent in $\rho$,
  {\em e.g.} with a square-root singularity in $\rho$).

From a combinatorial point of view, 
this identity of generating series means the following. 
There exist combinatorial classes $\mathcal G$ and $\mathcal G_i$ (for $1 \leq i \leq r$), 
whose generating functions are $G$ and the $G_i$'s, respectively, 
and such that 
a $\mathcal T_0$-structure is either a pair of structures in $\mathcal{G}_i\times \mathcal T_i$, for some $i$, or a $\mathcal{G}$-structure.

\begin{lemma}[Giant component: the simple pole case]
   Let $\mathcal T_0, \mathcal T_1, \ldots, \mathcal T_r$ be combinatorial classes whose generating series have the same radius of convergence $\rho$ and are  analytic on a $\Delta$-domain. Assume that $T_0= F(z,T_1,\dots,T_r)$
  and \cref{eq:asymp_Ti_polar,eq:F_Affine} hold.

Let $\mathbf{t}_n$ be a uniform random structure of size $n$ in $\mathcal T_0$, 
with a set of $k$ marked elements, chosen uniformly at random.
For $j \in \{1,\dots,r\}$, we call 
$E^{(n)}_{j}$ the event that $\mathbf{t}_n$ is a pair of substructures in $\mathcal{G}_{j}\times \mathcal T_{j}$
{\em and} that all $k$ marked elements belong to the $\mathcal{T}_{j}$-substructure. Then, we have
\begin{equation}
   \proba (E^{(n)}_{j}) \stackrel{n\to +\infty}{\longrightarrow} \frac {\delta_{j} G_{j}(\rho) } 
	   {\sum_{i=1}^r \delta_i G_i(\rho)}.
       \label{eq:lim_probEj}
  \end{equation}
  \label{lem:giantComponent_Linear}
\end{lemma}
Note that the right-hand side of \cref{eq:lim_probEj} above sums to $1$.
Informally, the lemma says that, with high probability,
the structure $\mathbf{t}_n$ has a giant substructure of some type $\mathcal{T}_{j}$.
This type (\emph{i.e.}~the value of $j$) is however random and \cref{eq:lim_probEj}
gives the limiting probabilities. 
When the $\mathcal{T}_i$  are families of permutations
and assuming that we know the limiting permutons of the $\mathcal{T}_{j},j>0$, 
we can conclude that the limiting permuton of $\mathcal{T}_{0}$ is
taken at random among those of the $\mathcal{T}_{j}$ 
with probabilities given by \cref{eq:lim_probEj}.

\begin{proof}
       We fix $j \in \{1,\dots,r\}$.
  The generating series of structures in $\mathcal T_0$ with a set of $k$ marked elements
  is given by  $T_0^{(k)}/k!$.
  On the other hand, the generating series of structures in $\mathcal{G}_{j}\times \mathcal T_{j}$
  with a set of $k$ marked elements, {\em all in the $\mathcal T_{j}$-substructure},
  is $G_{j}(z) T_{j}^{(k)}(z)/k!$.
  Therefore
  \begin{equation}
    \label{eq:PEj}
    \proba (E^{(n)}_{j}) = 
 \frac{[z^n]G_{j}(z) T_{j}^{(k)}(z)}{[z^n]T_0^{(k)}(z) }. 
 \end{equation}
We now evaluate the limit of the above quantity when $n$ tends to infinity using singularity analysis.
  From the assumptions \eqref{eq:asymp_Ti_polar} 
  and \eqref{eq:F_Affine}, we get that,
  for $z$ in a $\Delta$-neighborhood of $\rho$,
  \[T_0(z) = \frac{1}{\rho-z} \left( \sum_{i=1}^r \delta_i G_i(\rho) \right)
       + \O(1). \]
  By singular differentiation, in a $\Delta$-neighborhood of $\rho$,
  \[T_0^{(k)}(z) = \frac{k!}{(\rho- z)^{k+1}} \left( \sum_{i=1}^r \delta_i                   
      G_i(\rho) \right) + \O\bigg(\frac1{(\rho-z)^{k}}\bigg). \]
  Similarly, 
  \[T_j^{(k)}(z) = \frac{k!\, \delta_j}{(\rho- z)^{k+1}} 
  + \O\bigg(\frac1{(\rho-z)^{k}}\bigg). \]
    By the transfer theorem (\cref{thm:transfert}), we obtain
\begin{align*}
[z^n]\left(T_0^{(k)}(z)\right) &\sim \frac{n^k}{\rho^{n+k+1}}\sum_{i=1}^r \delta_i G_i(\rho);\\
[z^n]\left(G_j(z) T_j^{(k)}(z)\right) &\sim \frac{n^k}{\rho^{n+k+1}}\delta_j G_j(\rho).
\end{align*}
Plugging these estimates back into \eqref{eq:PEj}, we have
\[
\proba (E^{(n)}_{j}) = 
\frac{[z^n]G_{j}(z) T_{j}^{(k)}(z)}{[z^n]T_0^{(k)}(z) } \stackrel{n\to +\infty}{\longrightarrow} 
\frac {\delta_{j} G_{j}(\rho) }{\sum_{i=1}^r \delta_i G_i(\rho)}.\qedhere
\]
\end{proof}

We now give a similar statement when all $T_i$ have square-root singularities.

\begin{lemma}[Giant component: the square-root case]
   Let $\mathcal T_0, \mathcal T_1, \ldots, \mathcal T_r$ be combinatorial classes
  whose generating series have the same radius of convergence $\rho$ and are analytic on a $\Delta$-domain.
  We assume that $T_0= F(z,T_1,\dots,T_r)$ for some function $F$ which is analytic 
  on a neighborhood of $\{|z| \le \rho, |y_i| \le T_i(\rho)\}$ and that there exist $\beta_i$'s such that
\begin{equation}
  T_i(z) = T_i(\rho) - \beta_i \sqrt{\rho - z} + \O(\rho - z)\,\ ,\ 1 \leq i \leq r.
  \label{eq:asymp_Ti}
\end{equation}
Let $\mathbf{t}_n$ be a uniform random structure of size $n$ in $\mathcal T_0$, 
with a set of $k$ marked elements, chosen uniformly at random. 
Let $E^{(n)}_{j}$ be the event that all $k$ marked elements belong to the same $\mathcal{T}_{j}$-substructure.
Then
\begin{equation}
   \proba (E^{(n)}_{j}) \stackrel{n\to +\infty}{\longrightarrow} \beta_{j}
   \frac{\partial F (y_0, \dots, y_d)}{\partial y_{j}} \Big|_{(\rho, T_1(\rho),\dots,T_r(\rho))}
\times\left(\sum_{i=1}^r \beta_i \frac{\partial F (y_0, \dots, y_d)}{\partial y_{j}} \Big|_{(\rho, T_1(\rho),\dots,T_r(\rho))}\right)^{-1}. 
\label{eq:lim_probEj2}
\end{equation}
  \label{lem:giantComponent_Branching}
\end{lemma}

Contrary to the simple pole case, we do not assume that $F$ is linear.
Consequently, a structure in $\mathcal T_0$ might be composed of an $\mathcal F$-structure with several $\mathcal T_i$-substructures. 
Since the limiting probabilities in \cref{eq:lim_probEj2} sum to one,
the above lemma states that, with high probability, the structure has a giant substructure of some type $\mathcal T_j$.
\cref{eq:lim_probEj2} gives us the limiting distribution of this random type $\mathcal T_j$. 
As for \cref{lem:giantComponent_Linear},
when the $\mathcal T_j$ are families of permutations,
this lemma can be used to infer the limiting permuton
of $\mathcal{T}_{0}$ from those of the $\mathcal{T}_{j}$.

\begin{proof}
  We fix $\{1,\dots,r\}$. Similarly to the proof of \cref{lem:giantComponent_Linear}, we can express
  $\proba (E^{(n)}_{j})$ as a quotient of coefficients of generating series: in this case,
\[\proba (E^{(n)}_{j}) = \frac1{[z^n]T_0^{(k)}(z) } \, \cdot \,  [z^n]  
\left(T_j^{(k)}(z)\,\frac{ \partial F (y_0, \dots, y_d)}{\partial y_{j}} \Big|_{(z, T_1(z),\dots,T_r(z))}  \right).
  \]
 From assumption \eqref{eq:asymp_Ti} and the analyticity of $F$, we get that,
  for $z$ in a $\Delta$-neighborhood of $\rho$,
  \[T_0(z) = T_0(\rho) - \sqrt{ \rho- z} \left( \sum_{i=1}^r \beta_i 
      \frac{\partial F (y_0, \dots, y_d)}{\partial y_{j}} \Big|_{(\rho, T_1(\rho),\dots,T_r(\rho))} \right)
       + \O(\rho-z). \]
  By singular differentiation, we have, on a $\Delta$-neighborhood of $\rho$,
  \[T_0^{(k)}(z) = (\rho- z)^{1/2-k} \, C_k\, \left( \sum_{i=1}^r \beta_i                   
      \frac{\partial F (y_0, \dots, y_d)}{\partial y_{j}} \Big|_{(\rho, T_1(\rho),\dots,T_r(\rho))} \right)
      +\O\big((\rho-z)^{1-k}\big), \]
  where $C_1 = 1/2$ and $C_k=1 \cdot 3 \dots (2k-3)/2^{k}$ for $k \geq 2$.
  Similarly, 
  \[T_j^{(k)}(z) = (\rho- z)^{1/2-k} \, C_k \beta_j
  +\O\big(( \rho-z)^{1-k}\big). \]
Since $F$ is analytic in $\big( \rho, T_1(\rho),\dots,T_r(\rho) \big)$, 
the series $\frac{ \partial F (y_0, \dots, y_d)}{\partial y_{j}} \Big|_{(z, T_1(z),\dots,T_r(z))} $ converge in $\rho$
and we have
\begin{multline*}
T_j^{(k)}(z)\,\frac{ \partial F (y_0, \dots, y_d)}{\partial y_{j}} \Big|_{(z, T_1(z),\dots,T_r(z))}  
= (\rho- z)^{1/2-k} \, C_k \beta_j \frac{ \partial F (y_0, \dots, y_d)}{\partial y_{j}} \Big|_{(\rho, T_1(\rho),\dots,T_r(\rho))} \\
+ \O\big(  (\rho-z)^{1-k}\big). 
\end{multline*}
We conclude using the transfer theorem, as in the proof of \cref{lem:giantComponent_Linear}.
\end{proof}

\cref{lem:giantComponent_Linear,lem:giantComponent_Branching}
can also be applied in the particular situation where
one $\mathcal  T_i$ is equal to $\mathcal T_0$.
In such cases, the lemma yields the existence of a giant substructure
that is of type $\mathcal T_0$ with a probability $p$, typically in $(0,1)$.
When this occurs, we apply recursively 
\cref{lem:giantComponent_Linear} (or \ref{lem:giantComponent_Branching}) to this substructure.
After a random and almost surely finite number of iterations, we find a giant substructure
of a different type.
In the permutation case, this idea can be used to find the limiting permuton
of $\mathcal T_0$; see an example in \cref{ssec:couteau_suisse_union}.

\subsection{Several macroscopic substructures}\label{ssec:SeveralSubstructures}
We now describe a framework where several macroscopic substructures appear:
we assume that the generating series $ T_1,\dots, T_r$ have singularities which are simple poles 
and that $F$ is a polynomial.
Writing $F$ as a sum of monomials decomposes 
$\mathcal T_0$ into a disjoint union of subfamilies, 
one corresponding to each monomial.
We therefore focus on the case where $F$ is a monomial.

  We assume that the generating series $T_1,\dots,T_r$ have singularities which are simple poles, \emph{i.e.}, 
  \begin{equation}
    T_i(z) = \frac{\delta_i}{\rho - z} + \O(1).
    \label{eq:asymp_Ti_polar2}
  \end{equation}
  Assume in addition that 
  \begin{equation}
     F(z,T_1,\dots,T_r) = G(z) T_1 T_2\dots T_r, 
     \label{eq:F_Monome}
   \end{equation}
  where $G(z)$ is convergent at $\rho$; since there can be repetitions in the list $(T_1,\dots,T_r)$,
  this covers the case of a general monomial. 
  Let $\mathcal G$ be a combinatorial class with generating series $G$.

A structure in $\mathcal T_0$ can be identified with a list consisting of substructures in $\mathcal{G}, \mathcal T_1,\dots,\mathcal{T}_r$  (one structure from each class).

\begin{lemma}[Several macroscopic components: the monomial case]
\label{lem:SeveralComponentsMonomial}
   Let $\mathcal T_0, \mathcal T_1, \ldots, \mathcal T_r$ be combinatorial classes whose generating series have the same radius of convergence $\rho$ and are analytic on a $\Delta$-domain.
   We assume that $T_0= F(z,T_1,\dots,T_r)$ and \cref{eq:asymp_Ti_polar2,eq:F_Monome} hold.
We mark a set of $k$ elements, taken uniformly at random,  
   in a uniform random $\mathcal T_0$-structure of size $n$, and denote by $\ell_i$ ($1\leq i \leq r$) the (random) number of marked elements lying in the $\mathcal{T}_i$-substructure. 

Then $(\ell_1,\dots,\ell_r)$ is asymptotically uniformly distributed in the set $\{\ell_1 +\dots +\ell_r =k\}$.
\end{lemma}

\begin{proof}
  From the assumptions \eqref{eq:asymp_Ti_polar2} and \eqref{eq:F_Monome}, we get that,
  for $z$ in a $\Delta$-neighborhood of $\rho$,
  \[T_0(z) = G(\rho) \frac{\delta_1\dots \delta_r}{(\rho-z)^r}  + \O\bigg(\frac1{(\rho-z)^{r-1}}\bigg). \]
  By singular differentiation, on a $\Delta$-neighborhood of $\rho$, we have
  \[T_0^{(k)}(z) = G(\rho) \frac{(r+k-1)!}{(r-1)!} \frac{\delta_1\dots \delta_r}{(\rho-z)^{r+k}} + \O\bigg(\frac1{(\rho-z)^{r+k-1}}\bigg). \]
  Similarly, 
  \[T_{i}^{(\ell_i)}(z) = \frac{\ell_i!\ \delta_i}{(\rho- z)^{\ell_i+1}} 
  + \O\bigg(\frac1{(\rho-z)^{\ell_i}}\bigg). \]
  Combining both equations, we can write 
   \begin{align*}
G(\rho)\sum_{\ell_1 +\dots +\ell_r =k} \binom{k}{\ell_1,\dots,\ell_r} \prod_{i=1}^r T_{i}^{(\ell_i)}(z)\hspace{-3cm}&\\
&= G(\rho)\sum_{\ell_1 +\dots +\ell_r =k} \binom{k}{\ell_1,\dots,\ell_r} \prod_{i=1}^r\left( \frac{\ell_i!\, \delta_i}{(\rho- z)^{\ell_i+1}}  \right) + \O\bigg(\frac1{(\rho-z)^{r+k-1}}\bigg)\\
&= G(\rho) \frac{\delta_1\dots \delta_r}{(\rho-z)^{r+k}} \left(\sum_{\ell_1 +\dots +\ell_r =k} k!\right)
\O\bigg(\frac1{(\rho-z)^{r+k-1}}\bigg)\\
&= T_0^{(k)}(z) + \O\bigg(\frac1{(\rho-z)^{r+k-1}}\bigg),
\end{align*}
where in the last line we used that the number of $(\ell_1,\dots,\ell_r)$ such that $\ell_1 +\dots +\ell_r =k$ is $\binom{k+r-1}{r-1}$. 

By the transfer theorem, we obtain (for $\ell_1 +\dots +\ell_r =k$), 
\begin{align*}
[z^n]\left(G(z)\binom{k}{\ell_1,\dots,\ell_r} \prod_{i=1}^r T_{i}^{(\ell_i)}(z)\right) &\sim  G(\rho)\binom{k}{\ell_1,\dots,\ell_r} \frac{n^{r+k-1}}{\rho^{n+k+r}} \frac{1}{(k+r-1)!}\left(\prod_{i=1}^r \ell_i!\delta_i\right)\\
 &\sim  G(\rho) \frac{n^{r+k-1}}{\rho^{n+k+r}} \frac{k!}{(k+r-1)!}\left(\prod_{i=1}^r \delta_i\right).
\end{align*}
The right-hand side does not depend on $\ell_i$'s. Summing over the $\binom{k+r-1}{r-1}$ possible values for the $\ell_i$'s we obtain
$$
[z^n]\left(T_0^{(k)}(z)\right) \sim G(\rho) \frac{n^{r+k-1}}{\rho^{n+k+r}} \frac{1}{(r-1)!} \left(\prod_{i=1}^r\delta_i\right).
$$
Recall that we consider a uniform random structure $\mathbf{t}_n$ of size $n$ in $\mathcal T_0$ with a uniform set of $k$ marked elements.
Let $E^{(n)}_{\ell_1,\dots,\ell_r}$ denote the event that for every $1\leq i\leq r$, exactly $\ell_i$ of these marked elements lie in the $\mathcal T_i$-substructure.
Its probability can be computed by
$$
\proba (E^{(n)}_{\ell_1,\dots,\ell_r}) = 
\frac{[z^n]\left(G(z)\binom{k}{\ell_1,\dots,\ell_r} \prod_{i=1}^r T_{i}^{(\ell_i)}(z)\right)}{[z^n]\left(T_0^{(k)}(z)\right)} \to \frac{1}{\binom{k+r-1}{r-1}}.
$$
This concludes the proof.
\end{proof}
We now discuss briefly the more general case
where $T_0=F(z,T_1,\dots,T_r)$, with $F$ a polynomial in $T_1,\dots,T_r$ (not necessarily a monomial)
with coefficients converging at $z=\rho$ (the $T_i$'s are still assumed to have a simple pole in $\rho$).
Each monomial has a pole at the singularity, whose multiplicity equals the degree of the monomial.
Therefore, only monomials of maximal degree contribute to the limit.
We will use this principle to determine permuton limits of some families of permutations
in two different cases.
\begin{itemize}
  \item 
An example with exactly one monomial of maximal degree (namely one monomial of degree $2$ and one of degree $1$)
is given in \cref{ssec:compound}.
  \item 
When there are several monomial of maximal degree,
a random element in $\mathcal T_0$ belongs asymptotically with positive probability
to each of the classes corresponding to these monomials.
We will see an example of this kind of behavior in \cref{ssec:couteau_suisse_union}.
\end{itemize}

\subsection{Examples}

\subsubsection{Four classes $\TTT$ with a single strongly connected component pointing to $\TTT$}

We consider the $X$-class already analyzed in \cref{sec:ClasseX_debut,sec:ClasseX}.
As  explained in \cref{sec:ClasseX}, we can use \cref{Th:linearCase} to prove
that all critical classes except for $\mathcal T_0$, namely $\mathcal T_3$,
 $\mathcal T_4$, $\mathcal T_6$ and  $\mathcal T_7$, converge to an $X$-permuton.
We can prove that $\mathcal T_0$ has the same limit using \cref{lem:giantComponent_Linear} 
instead of the little trick used in \cref{sec:ClasseX}.
Indeed, the first equation of the specification \eqref{eq:SpecifClasseX} expresses 
$\mathcal T_0$ as a linear combination of $\mathcal T_3$,
 $\mathcal T_4$, $\mathcal T_6$ and  $\mathcal T_7$ (the coefficients involving subcritical classes). 
 Moreover, all series $\mathcal T_3$,
 $\mathcal T_4$, $\mathcal T_6$ and  $\mathcal T_7$  have a simple pole at $\rho = 1 - \sqrt{2}/2$.  
 Therefore, by \cref{lem:giantComponent_Linear}, with probability tending to $1$,
a uniform random tree in $\mathcal T_0$ has a giant substructure in either $\mathcal T_3$,
 $\mathcal T_4$, $\mathcal T_6$ or $\mathcal T_7$.
 Since the latter all tend to an $X$-permuton (with the same parameters), so does $\mathcal T_0$.
 \medskip

Similarly, we can replace our previous trick by \cref{lem:giantComponent_Linear} 
for the classes discussed in \cref{sec:ClasseXTilde,sec:ClasseV,sec:PinPerm}.

\subsubsection{A class with many strongly connected components}
\label{ssec:couteau_suisse_union}

The example that we consider now is the class 
\[\TTT= \Av(2413,3142,2314,3241,21453,45213).\]
This class is not substitution-closed and contains no simple permutation.

For this class, we obtain\footnote{See the \href{http://mmaazoun.perso.math.cnrs.fr/pcfs/} {companion Jupyter notebook}  \texttt{examples/Union.ipynb}} a specification with 13 families $\TTT=\TTT_0, \dots, \TTT_{9}, \TTT_{11},\TTT_{12},\TTT_{13}$ (the family $\TTT_{10}$ being empty, see \cref{rk:ClasseUnion} in Appendix).
The corresponding system on series can be explicitly solved,
showing that all series except $T_1$ and $ T_{11}$ are critical and have a common square-root singularity.
The complete specification and the explicit solution of the associated system
can be found in \cref{Ex:ClasseUnion}.
The dependency graph restricted to the critical $\mathcal T_i$ is shown in \cref{fig:DependendyGraphUnion}
and has nine strongly connected components.
\begin{figure}[htbp]
\includegraphics[width=9cm]{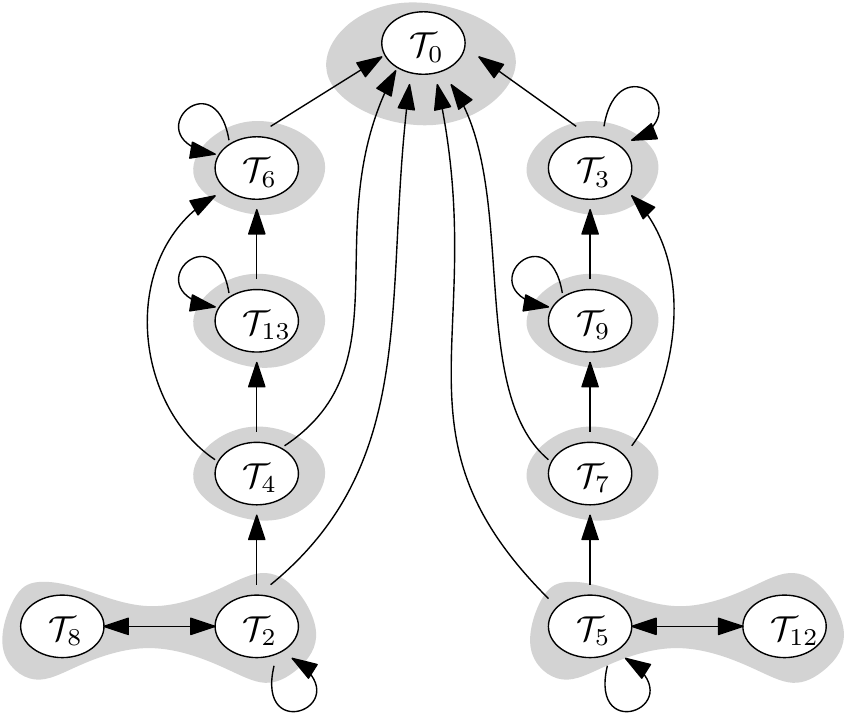}
\caption{The subgraph restricted to critical families $\mathcal T_i$, for the specification~\eqref{eq:SpecifClasseUnion} of the class $\Av(2413,3142,2314,3241,21453,45213)$. It has nine strongly connected components.}
\label{fig:DependendyGraphUnion}
\end{figure}

\begin{remark}
This example has been built on purpose to show a graph $G^\star$ with many strongly connected components. 
This has been ensured by considering the class $\Av(213) \cup \Av(231)$, for which it is easy to check that the basis is 
$\{2413,3142,2314,3241,21453,45213\}$ given above. 
We are aware that studying this class \emph{via} its tree-specification (given in Appendix) is neither the most natural nor the simplest thing to do. 
Our goal with this example is to illustrate that, even without the knowledge of the simple ``union'' structure of our class, 
our approach would still work. 
\end{remark}
\medskip

We now determine the limiting permuton of a uniform random permutation in $\mathcal T$,
using the specification; see \cref{fig:SimusUnion} for a simulation.
\begin{proposition}
 A uniform random permutation in the class $\Av(2413,3142,2314,3241,21453,45213)$
 converges in distribution to the random permuton,
which is the diagonal with probability $1/2$ and the antidiagonal with probability $1/2$. 
\end{proposition}
\begin{figure}[htbp]
	\begin{tabular}{ccc}
		\includegraphics[width = 0.2\linewidth]{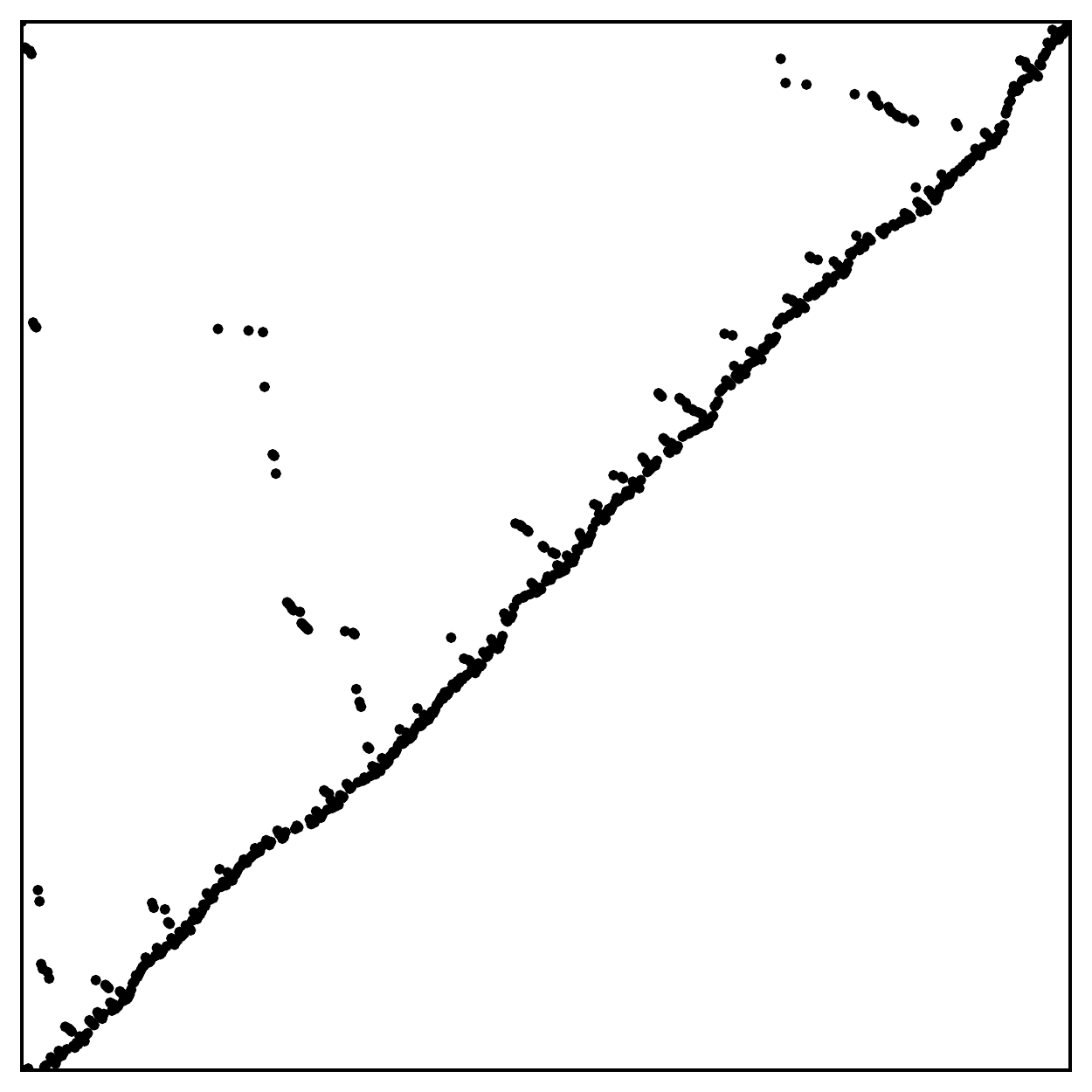} & \includegraphics[width = 0.2\linewidth]{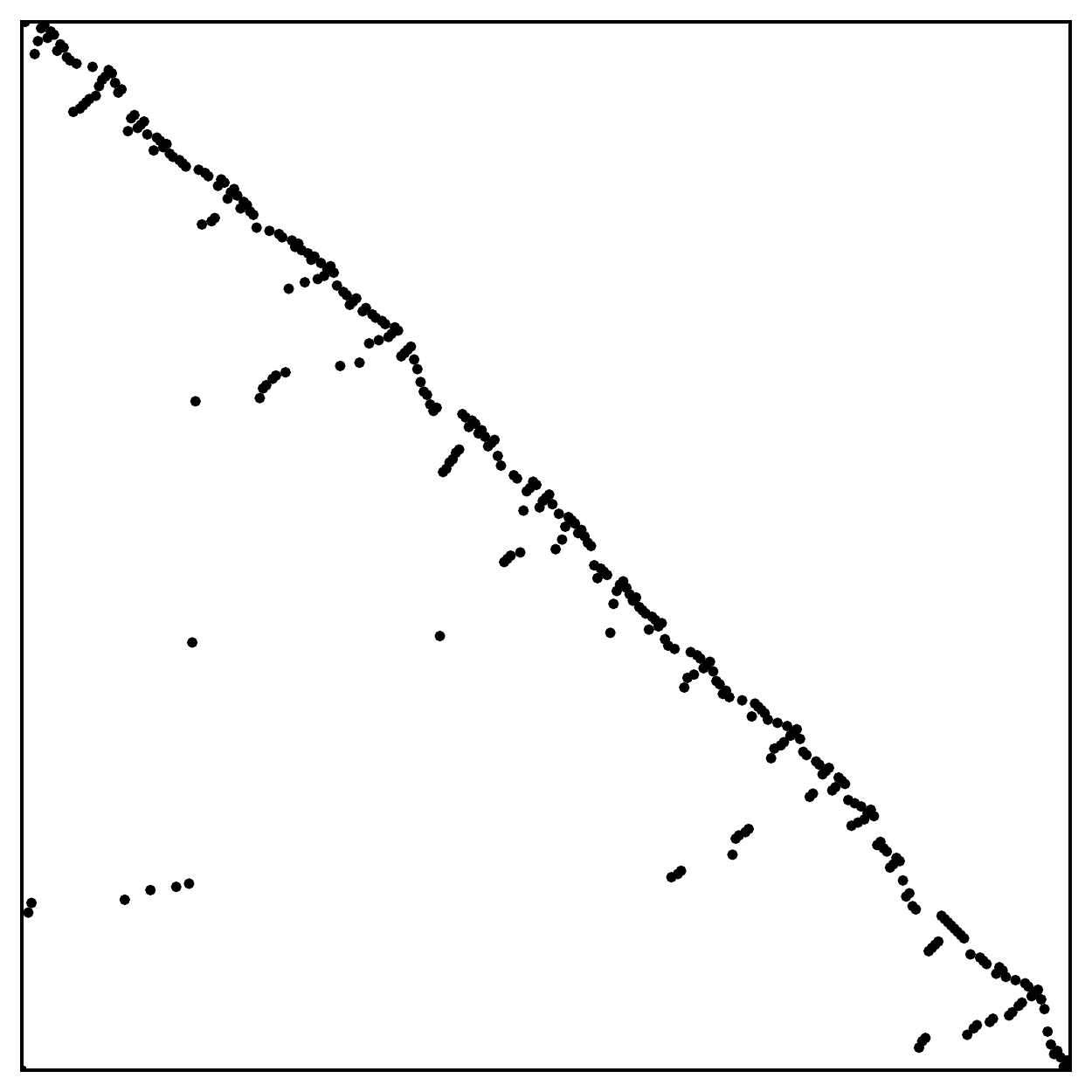} &\includegraphics[width = 0.2\linewidth]{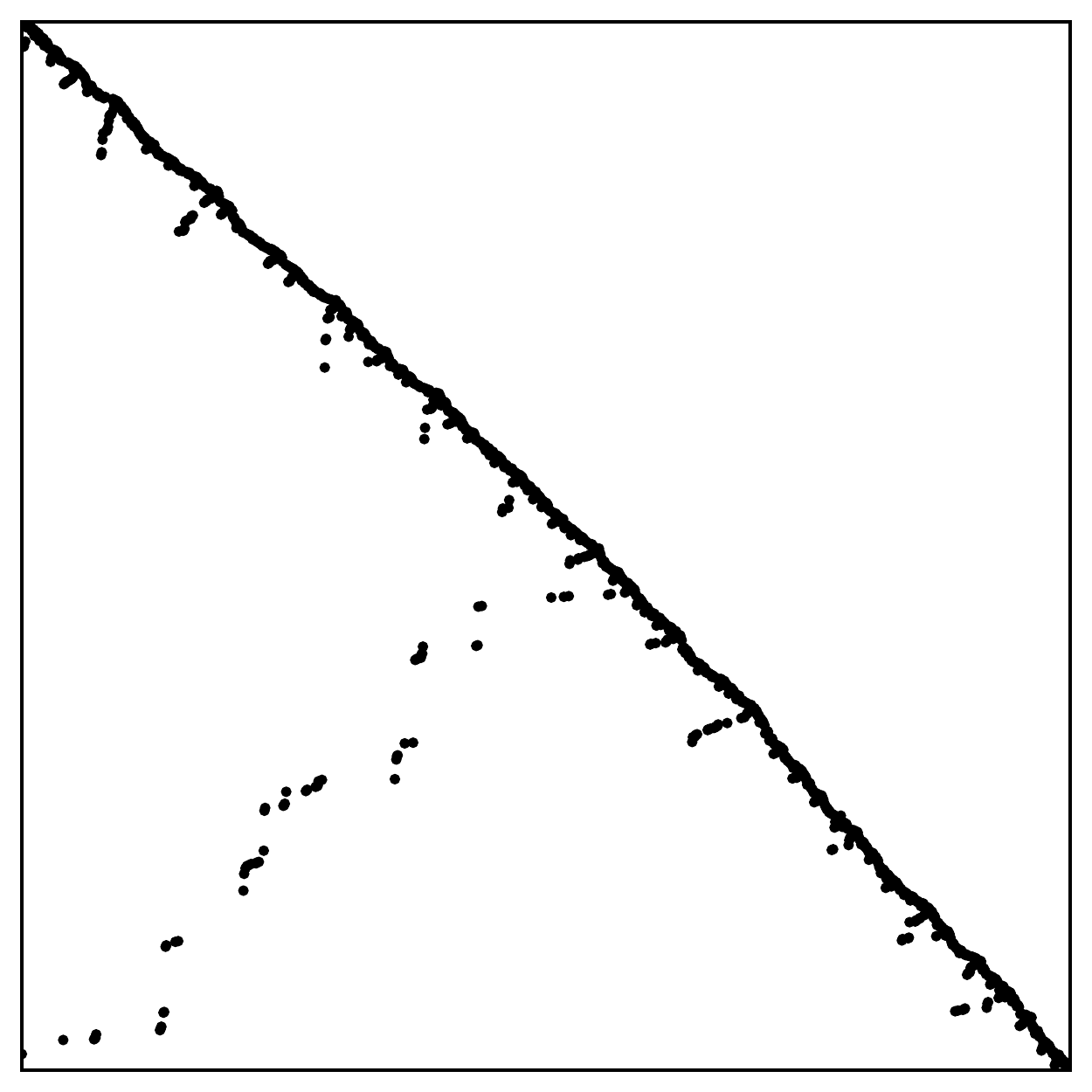}  
	\end{tabular}
	\caption{Three large permutations in $\mathcal T$, drawn uniformly at random.
	\label{fig:SimusUnion}}
\end{figure}
\begin{proof}
The strategy is to proceed step by step, determining the limiting permuton of uniform random permutations in each of the critical families,
navigating in the dependency graph of \cref{fig:DependendyGraphUnion} from bottom to top.

We first consider the strongly connected component $\{\mathcal T_2,\mathcal T_8\}$.
Taking the equations for $\mathcal T_1$, $\mathcal T_2$ and $\mathcal T_8$ in the specification \eqref{eq:SpecifClasseUnion} for $\mathcal T$
given in \cref{Ex:ClasseUnion}, we have a specification for $\mathcal T_2$.
This restricted specification satisfies Hypothesis (SC) and is essentially branching.
We can therefore apply \cref{Th:branchingCase} (the other hypotheses are straightforward to check)
and we get that a uniform random permutation in $\mathcal T_2$ converge to a biased Brownian separable permuton with some parameter $p$ in $[0,1]$.
Since the only quadratic term in the system is $\oplus[\mathcal T_8,\mathcal T_2]$, which corresponds to a $\oplus$ node,
we have $p_+=1$, which means that the limit is in fact the main diagonal of $[0,1]^2$.

We now consider $\mathcal T_4$. It is given by the equation $\mathcal T_4=\ominus[\mathcal T_1,\mathcal T_2]$.
The family $\mathcal T_1$ is subcritical, while $\mathcal T_2$ has a square-root singularity in $\rho$
(as easily seen on the explicit expression given in \cref{Ex:ClasseUnion}).
Applying \cref{lem:giantComponent_Branching}, we know that a uniform random permutation of $\mathcal T_4$ has 
a giant substructure in $\mathcal T_2$, and therefore, also converges to the diagonal permuton.

Moving on to $\mathcal T_{13}$, it is given by the equation
\[
\mathcal T_{13}= \oplus[\mathcal T_{4},\mathcal T_{13}]\uplus \oplus[\mathcal T_{1},\mathcal T_{13}]\uplus \oplus[\mathcal T_{4},\mathcal T_{11}]\uplus \ominus[\mathcal T_{1},\mathcal T_{13}]. 
\]
An important difference with the equation of $\mathcal T_4$ is that it involves also $\mathcal T_{13}$ itself on the right-hand side.
We can still apply \cref{lem:giantComponent_Branching} and conclude
that a uniform random permutation of $\mathcal T_{13}$ has a giant substructure in either $\mathcal T_4$ or $\mathcal T_{13}$.
Iterating this argument (see the discussion at the end of \cref{Sec:SubTree}),
after a finite number of steps, we find a giant substructure of type $\mathcal T_4$.
We conclude that a uniform random permutation in $\mathcal T_{13}$ has the same limiting permuton
as one in $\mathcal T_4$, \emph{i.e.}~ the diagonal permuton.
With the exact same reasoning, we prove that a uniform random permutation in $\mathcal T_6$ also converges to the diagonal permuton 
(which appears here as the Brownian separable permuton of parameter $p_+=0$).

On the other hand, and following the same steps, we show that a uniform random permutation in any of the classes $\mathcal T_5$, $\mathcal T_7$, $\mathcal T_9$ and $\mathcal T_3$
converges to the antidiagonal permuton.

Finally, we consider $\mathcal T_0$. It is given by the equation
  \[\mathcal T_{0}= \{ \bullet \} \uplus \oplus[\mathcal T_{1},\mathcal T_{2}]\uplus 
  \oplus[\mathcal T_{1},\mathcal T_{3}]\uplus \oplus[\mathcal T_{4},\mathcal T_{2}]\uplus \ominus[\mathcal T_{1},\mathcal T_{5}]\uplus \ominus[\mathcal T_{1},\mathcal T_{6}]\uplus \ominus[\mathcal T_{7},\mathcal T_{5}]. \]
In the above equation $T_1$ is convergent in $\rho$
and all other classes are critical (with square-root singularities).
By \cref{lem:giantComponent_Branching}, a uniform random permutation in $\mathcal T_0$
contains a giant substructure of type $\mathcal T_{\bm j}$, where ${\bm j}$ follows asymptotically some distribution on $\{2,3,4,5,6,7\}$.
For each $j_0$ in this set, we denote $p_{j_0}=\mathbb P(\bm j=j_0)$.
We can then conclude that a uniform random permutation in $\mathcal T_0$ converges in distribution to the random permuton,
which is the diagonal with probability $p_+:= p_2+p_4+p_6$ and the antidiagonal with probability $p_-:=p_3+p_5+p_7$.
Using the explicit expression of the $p_j$'s in \cref{lem:giantComponent_Branching} or observing the symmetry,
we see that $p_+=p_-=1/2$.
\end{proof}

\subsubsection{A ``compound'' class}
\label{ssec:compound}

Our goal here is to illustrate the emergence of several macroscopic substructures in the limit, 
as described in \cref{ssec:SeveralSubstructures}. 
To this effect, we consider the class $\mathcal{C}$ which can be defined as the downward closure of $\oplus[\mathcal{X},\mathcal{X}]$, 
where $\mathcal{X}$ denotes the $X$-class (see \cref{sec:ClasseX_debut,sec:ClasseX}). 
This class has no simple permutation and has therefore a tree-specification. 
We explain below an easy way to construct one such specification. 
However the obtained specification does not satisfy Hypothesis (SC) (p.\pageref{Hyp:StronglyConnected}). 
We explain here how to determine nevertheless the limiting permuton
of a uniform random permutation in $\mathcal C$.

We first define the limiting permuton.
\begin{definition}
  Let $\bm U$ be a uniform random variable in $[0,1]$.
  We construct the random permuton $\bm\mu^{\oplus[X,X]}$ as follows:
  \begin{itemize}
    \item on $[0,\bm U] \times [0,\bm U]$, we take a rescaled copy of $\mu^X_{(\frac14, \frac14, \frac14, \frac14)}$, \emph{i.e.}
      \[\bm\mu^{\oplus[X,X]} \big([\bm Ua,\bm Ub] \times [\bm Uc,\bm Ud]\big) = \bm U \cdot \mu^X_{(\frac14, \frac14, \frac14, \frac14)}\big([a,b] \times [c,d]\big);\]
   \item similarly, on $[1-\bm U,1] \times [1-\bm U,1]$, we take a rescaled copy of $\mu^X_{(\frac14, \frac14, \frac14, \frac14)}$;
    \item $\bm\mu^{\oplus[X,X]} \big([0,\bm U] \times [1-\bm U,1] \big)= \bm\mu^{\oplus[X,X]} \big([1-\bm U,1] \times [0,\bm U]\big) = 0$.
  \end{itemize}
\end{definition}

We now describe the distribution of the permutation constructed from $k$ random points in this permuton.
\begin{lemma}
  Let $(\bm \ell_1,\bm \ell_2)$ be a uniform random variable in the set $\{(\ell_1,\ell_2) \in \mathbb Z_{\ge 0}^2:\ell_1+\ell_2=k\}$.
  Conditionally on $(\bm \ell_1,\bm \ell_2)$, we take $\pi_i$ (for $i$ in $\{1,2\}$) to be independent random permutations
  distributed as $\Perm_{\ell_i}(\mu^X_{(\frac14, \frac14, \frac14, \frac14)})$.
  Then
  \[ \Perm_k(\bm\mu^{\oplus[X,X]}) \stackrel{(d)}= \oplus[\pi_1,\pi_2].\]
  \label{lem:subpermutations_XX}
\end{lemma}
\begin{proof}
Denote as in \cref{ssec:permutons_intro}  $(\xx_{1},\yy_{1}),\dots, (\xx_{k},\yy_{k})$ the coordinates of the $k$ i.i.d. points drawn with distribution $\bm\mu^{\oplus[X,X]}$ in order to define $\Perm_k(\bm\mu^{\oplus[X,X]})$. It suffices to notice that 
$$
\mathrm{card}\{1\leq i\leq k; \xx_i\leq \bm U\} 
$$
is uniformly distributed in $\{0,1,\dots,k\}$. Moreover, conditionally on $\bm U$ and on the event $\{\xx_i<\bm U\}$, $\xx_i$ is uniform in $(0,\bm U)$. Therefore the permutation induced by points $\{(\xx_{i},\yy_{i}); \xx_i\leq \bm U\}$ (resp. $> \bm U$) has the same distribution as $\pi_1$ (resp. $\pi_2$).
We conclude that the permutation induced by the whole set $\{(\xx_{i},\yy_{i}); 1 \le i \le k\}$
has the same distribution as $\oplus[\pi_1, \pi_2]$,
which is what we wanted to prove.
\end{proof}

We can now state and prove our convergence result, illustrated in \cref{Fig:exemple_XX}.
\begin{proposition}
  \label{prop:compound}
 Let $\mathcal C$ be the downward closure of $\oplus[\mathcal{X},\mathcal{X}]$ and $\bm \sigma_n$ be a uniform random permutation of size $n$ in $\mathcal C$.
 Then $\bm \sigma_n$ converges in distribution to the random permuton $\bm\mu^{\oplus[X,X]}$.
\end{proposition}

\begin{figure}[htbp] 
	\centering
	\includegraphics[width = 0.3\linewidth]{simu_XX_n242.pdf} \hspace{0.5cm} \includegraphics[height = 0.29\linewidth]{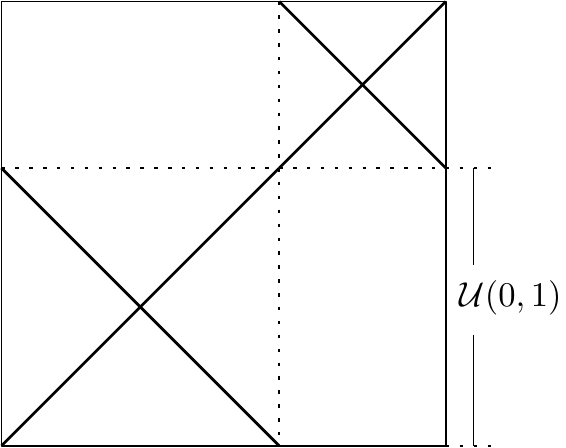}
	\caption{Left: A simulation of a uniform permutation of size 242 in $\mathcal C$. Right: The limiting permuton, as predicted by \cref{prop:compound} ($\mathcal U(0,1)$ stands for the uniform distribution on $(0,1)$).
		\label{Fig:exemple_XX}}
\end{figure}
\begin{proof}
Clearly, $\mathcal{C}$ can be written as $\mathcal{X} \cup \oplus[\mathcal{X},\mathcal{X}]$, 
but this equation is essentially ambiguous, hence does not fit in the tree-specification framework. 
Instead, writing that 
$$\mathcal{C} = \mathcal{X}^{\nonp} \uplus \oplus[\mathcal{X}^{\nonp},\mathcal{X}]$$
provides an unambiguous description of $\mathcal{C}$ (for the definition of $\mathcal{X}^{\nonp}$, see the third item in \cref{dfn:classesGeneralized}). 

We can therefore build a specification for $\mathcal{C}$, starting from that of the $X$-class,
\cref{eq:SpecifClasseX} (p.\pageref{eq:SpecifClasseX}).
Note that the families $\mathcal{X}$ and $\mathcal{X}^{\nonp}$ 
correspond
to $\mathcal{T}_0$ and $\mathcal{T}_1 \uplus \mathcal{T}_4$ in specification~\eqref{eq:SpecifClasseX}, respectively.
A specification for $\mathcal{C}$ can thus be obtained from 
the specification~\eqref{eq:SpecifClasseX} of the $X$-class, 
by adding to it the two equations
\begin{align}
 \mathcal C &= \mathcal X^{\nonp} \uplus \oplus[\mathcal X^{\nonp},\mathcal{T}_{0}] ;
 \label{eq:CompoundClass}\\
 \mathcal X^{\nonp} &= \mathcal{T}_1 \uplus \mathcal{T}_4.
\end{align}
These equations are not exactly of the form required in tree-specifications, 
but are easily modified to achieve a proper tree-specification.
The above form is however practical to apply the tools of this section.
In particular, we see that the series of $\mathcal X^{\nonp}$ and $\mathcal C$
both have the same radius of convergence $\rho$ as
the critical series of specification~\eqref{eq:SpecifClasseX}
(namely $\mathcal{T}_0$, $\mathcal{T}_3$, $\mathcal{T}_4$, $\mathcal{T}_6$ and $\mathcal{T}_7$)

We recall from \cref{sec:ClasseX} that a uniform random permutation
in any of the critical classes ($\mathcal{T}_0$, $\mathcal{T}_3$, $\mathcal{T}_4$, $\mathcal{T}_6$ and $\mathcal{T}_7$)
converges to the centered $X$-permuton.
We then note that $\mathcal{X}^{\nonp}$ is the disjoint union of a subcritical class and the critical class $\mathcal T_4$.
Therefore a uniform permutation in $\mathcal{X}^{\nonp}$ behaves asymptotically as one in $\mathcal T_4$,
and also converges to the centered $X$-permuton $\mu^X_{(\frac14, \frac14, \frac14, \frac14)}$.

We now focus on $\mathcal C = \mathcal X^{\nonp} \uplus \oplus[\mathcal X^{\nonp},\mathcal{T}_{0}]$.
The generating series of $\mathcal X^{\nonp}$ has a simple pole at $\rho$
(this follows from $T_4$ having a simple pole at $\rho$, see the equations p.\cref{eq:SpecifClasseX}).
On the contrary, the generating series of $\oplus[\mathcal X^{\nonp},\mathcal{T}_{0}]$
has a double pole at $\rho$, since both $\mathcal X^{\nonp}$ and $\mathcal{T}_{0}$
have a simple pole.
Using the transfer theorem, and up to multiplicative constants, 
the coefficients of the generating series of $\mathcal X^{\nonp}$ and $\oplus[\mathcal X^{\nonp},\mathcal{T}_{0}]$ behave asymptotically as
$\rho^{-n}$ and $n\, \rho^{-n}$ respectively.
Therefore a uniform random permutation of size $n$ in $\mathcal C$
is, with probability tending to $1$, in $\oplus[\mathcal X^{\nonp},\mathcal{T}_{0}]$.

Let us take a uniform random set of $k$ elements in a uniform random permutation $\bm \si_n$ in $\mathcal C$,
or equivalently, in $\oplus[\mathcal X^{\nonp},\mathcal{T}_{0}]$.
Then the number $\bm\ell_1$ (resp. $\bm\ell_2$) of these elements that are in the $\mathcal X^{\nonp}$- 
(resp. $\mathcal{T}_{0}$-)substructure is random.
Since the series of $\mathcal X^{\nonp}$ and $\mathcal{T}_{0}$ have both simple poles at $\rho$,
we can apply \cref{lem:SeveralComponentsMonomial} and $(\bm \ell_1,\bm \ell_2)$ is uniformly distributed
on the set $\{\ell_1+\ell_2=k\}$.
Since the permuton limit of elements in $\mathcal X^{\nonp}$ is $\mu^X_{(\frac14, \frac14, \frac14, \frac14)}$,
the $\bm \ell_1$ elements in the $\mathcal X^{\nonp}$-substructure induce a pattern $\bm \pi_1$,
which is asymptotically distributed
like $\Perm_{\bm \ell_1}(\mu^X_{(\frac14, \frac14, \frac14, \frac14)})$.
Similarly the $\bm \ell_2$ elements in the $\mathcal{T}_{0}$-substructure 
induce a pattern $\bm \pi_2$, 
which is asymptotically distributed
like $\Perm_{\bm \ell_2}(\mu^X_{(\frac14, \frac14, \frac14, \frac14)})$.

Comparing with \cref{lem:subpermutations_XX}, the pattern $\oplus[\bm \pi_1,\bm \pi_2]$ induced by the $k$ random elements in $\bm \si_n$
is asymptotically distributed as $\Perm_k(\bm\mu^{\oplus[X,X]})$.
We conclude with \cref{thm:randompermutonthm} that a uniform random permutation $\bm \si_n$ in $\mathcal C$
converges towards $\mu^{\oplus[X,X]}$.
\end{proof}

\appendix %
\section{Complex analysis toolbox} \label{sec:complex_analysis}

\subsection{Transfer theorem}

We start by defining the notion of $\Delta$-domain.
\begin{definition}[$\Delta$-domain and $\Delta$-neighborhood]\label{Def:DeltaDomaine}
A domain $\Delta$ is a {\em $\Delta$-domain at $1$} if there exist two real numbers  $R>1$ and $0<\phi<\tfrac{\pi}{2}$ such that
$$\Delta=\{z \in \mathbb{C} \mid |z|<R,\, z\neq 1,
|\arg(z-1)|>\phi\}.$$ By extension, for a complex number $\rho \neq
0$, a domain is a {\em $\Delta$-domain at $\rho$} if it the image by
the mapping $z\rightarrow \rho z$ of a $\Delta$-domain at $1$.  A
{\em $\Delta$-neighborhood} of $\rho$ is the intersection of a
neighborhood of $\rho$ and a $\Delta$-domain at $\rho$.
\end{definition}
We will make use of the following family of $\Delta$-neighborhoods: for $\rho \neq 0 \in \mathbb C$, $0<r<|\rho|$, $\varphi <\pi/2$, set $\Delta(\varphi,r,\rho) = \{z\in \mathbb C, |\rho-z|<r, \arg(\rho - z)>\varphi\}$.

When a function $A$ is analytic on a $\Delta$-domain at its radius of convergence $\rho$,
the asymptotic behavior of its coefficients is closely related to
the behavior of the function near the $\rho$.

The following theorem is a corollary of \cite[Theorem VI.3 p.~390]{Violet}.

\begin{theorem}[Transfer Theorem]\label{thm:transfert}
  Let $A$ be an analytic function whose radius of convergence is $\rho_A$. Assume moreover that $A$ is analytic on a $\Delta$-domain $\Delta$ at $\rho_A$, $\delta$ be an arbitrary real number in $\mathbb{R} \setminus \mathbb{Z}_{\geq 0}$ 
  and $C_A$ a constant possibly equal to $0$.
  
  Suppose $A(z) = (C_A + o(1))(1-\tfrac{z}{\rho_A})^{\delta}$ when $z$ tends to $\rho_A$ in $\Delta$.
  Then the  coefficient of $z^n$ in $A$, denoted by $[z^n]A(z)$ satisfies 
  $$[z^n]A(z) = (C_A +o(1)) \frac{1}{\rho_A^{n}} \ \frac{n^{-(\delta+1)}}{\Gamma(-\delta)},$$
  where $\Gamma$ is the gamma function.
\end{theorem}

\subsection{Generalities for systems of functional equations}
In this section and the subsequent one, we look at vectors of nonnegative series $\mathbf Y = (Y_1,\ldots Y_c)$ that satisfy systems of equations of the form
\begin{equation}
	\mathbf Y(z) = \mathbf \Phi(z,\mathbf Y(z)), 
	\label{eq:annex_generic_system}
\end{equation}
where $\mathbf \Phi(z,\mathbf y) = (\Phi_1(z,\mathbf y),\ldots ,\Phi_c(z,\mathbf y))$ is a vector of multivariate power series of $(z,\mathbf y)$ with nonnegative integer coefficients.

\begin{definition}\ 
The system $\mathbf \Phi$ is \emph{strongly connected} if the directed graph on $\{1,\ldots,c\}$ given by $j\to i$ whenever $\frac {\partial \Phi_i}{\partial y_j}$ is nonzero, is strongly connected. 
\end{definition}
This assumption guarantees that all series $Y_1,\ldots,Y_c$ have a common radius of convergence $\rho$ (see \cref{lem:MonotonieDesRho} for example).

\medskip
\subsection{Linear systems}
In this section we assume that $\mathbf \Phi$ is a linear function of its second argument, in the sense that \cref{eq:annex_generic_system} reduces to $\mathbf Y(z) = \mathbb M(z)\mathbf Y(z) + \mathbf V(z)$ where $\mathbf V(z) = \mathbf \Phi(z,\mathbf 0)$ and the $c \times c$-matrix $\mathbb M$ is the Jacobian of $\mathbf \Phi$ in its second argument.
Note that under the linear assumption $\mathbb M$ does not depend on $\mathbf y$.

The following proposition is an adaptation of known results: it extends Theorem V.7 (p.342) and Lemma V.1 (p.346) in \cite{Violet} (which establish that, when $\mathbb M(z)=z\mathbb M$,  then $\rho$ is a simple pole of $(\Id - \mathbb M(z))^{-1}$ and this quantity tends to $C / (z-\rho)$  where $C$ is a rank $1$ matrix), and Lemma 2 in \cite{BanderierDrmota} (where $\mathbb M(z)$ is a matrix with polynomial coefficients in $z$, but constants corresponding to dominating terms of the asymptotic behavior are not computed). The proof is mostly adapted from this last reference.

\begin{proposition}\label{Prop:AsymptotiqueLinear}
Consider the following system $$\mathbf Y(z) = \mathbb M(z)\mathbf Y(z) + \mathbf V(z)$$ where $\mathbf V(z) = \mathbf \Phi(z,\mathbf 0)$, $\mathbf \Phi$ being a linear function of its second argument, and the $c \times c$-matrix $\mathbb M$ is  the Jacobian of $\mathbf \Phi$ in its second argument.
Assume that the system is strongly connected, that all entries of $\mathbb M(z)$ and $\mathbf V(z)$ are series with nonnegative coefficients, that $\mathbb M(0) = \mathbf 0$ and that $\mathbf V$ is nonzero.

Then the unique solution $\mathbf Y(z) = (\Id - \mathbb M(z))^{-1}\mathbf V(z)$  is a formal power series with nonnegative coefficients. Moreover the common radius of convergence $\rho$ of the entries of $\mathbf Y$ is finite, and the following assertions are equivalent:
	\begin{enumerate}
		\item There exists $t\geq 0$ strictly smaller than the radius of convergence of all entries of $\mathbb M$ and $\mathbf V$, such that $\det(\Id - \mathbb M(t)) = 0$;
		\item The radius of convergence of all entries of $\mathbb M$ and $\mathbf V$ is strictly larger than $\rho$. 
	\end{enumerate}
	If they hold, then $\rho>0$ and
	\begin{enumerate}
		\setcounter{enumi}{2}
		\item $\rho$ is also the common radius of convergence of all entries of $(\Id-\mathbb M)^{-1}$;
		\item $\mathbb M(\rho)$ is an irreducible matrix with Perron eigenvalue $1$. We denote by $\mathbf u$ and $\mathbf v$ the corresponding left and right positive eigenvectors normalized so that $\transpose{\mathbf u}\mathbf{v}=1$ ;
		\item $(\Id-\mathbb M)^{-1}$ and $\mathbf Y(z)$ are analytic on a $\Delta$-neighborhood of $\rho$, and as $z\to \rho$, denoting coefficient-wise asymptotic equivalence by $\sim$,
	\begin{equation}
		(\Id - \mathbb M(z))^{-1} \sim \left(\frac {1} {\transpose{\mathbf u}\mathbb M'(\rho)\mathbf v}\right) \frac {\mathbf v\transpose{\mathbf u}} {\rho - z} .
		\label{eq:annexe_linear_asymp_inverse}
	\end{equation}
		Consequently,
		\begin{equation}
		\mathbf Y(z)  \sim 
		\left(\frac {\transpose{\mathbf u} \mathbf V(\rho)} {\transpose{\mathbf u}\mathbb M'(\rho) \mathbf v }\right) 
		\frac {\mathbf v} {\rho - z}.
		\label{eq:annexe_linear_asymp_Y}
		\end{equation}
\end{enumerate}
	Moreover, if the g.c.d. of the periods of the series in $\mathbb M$ is $1$, then there are no other singularities on the circle of convergence for the series in
	$\mathbf Y$ and $(\Id-\mathbb M)^{-1}$, and those series are analytic on a $\Delta$-domain at $\rho$.
\end{proposition}
We start with a lemma that is used in the subsequent proof.
\begin{lemma}\label{lem:rcv}
  Let $\mathbb M$ be an irreducible matrix whose coefficients are series in $z$ with nonnegative coefficients.
	Assume $\mathbb M(0)= \mathbf 0$, then $\Id - \mathbb M(z)$ is invertible around $z=0$ 
	and all the coefficients of $(\Id - \mathbb M(z))^{-1}$ are positive analytic functions with the same radius of convergence.
\end{lemma}

\begin{proof}
	The invertibility of $\Id - \mathbb M(z)$ near zero follows from the fact that the spectral radius of $\mathbb M(z)$ is continuous in $z$ and $\mathbb M(0)= \mathbf 0$.
	
	Fix $1\leq i,j,l \leq c$. By the irreducibility condition, there exists $k$ such that $\mathbb M(z)^k_{i,j} \neq 0$. Moreover
	$$
	(\Id - \mathbb M(z))^{-1}=\Id+\mathbb M(z)+\dots + \mathbb M(z)^{k-1} +\mathbb M(z)^{k}(\Id - \mathbb M(z))^{-1}.
	$$
	As a result, $((\Id - \mathbb M(z))^{-1})_{i,l}$ depends positively on $((\Id - \mathbb M(z))^{-1})_{j,l}$.
	Since $\mathbb M(z)^{k}(\Id - \mathbb M(z))^{-1}= (\Id - \mathbb M(z))^{-1}\mathbb M(z)^{k}$, it also implies that $((\Id - \mathbb M(z))^{-1})_{l,j}$ depends positively on $((\Id - \mathbb M(z))^{-1})_{l,i}$.
        Denote $\rho_{ij}$, the radius of convergence of $((\Id - \mathbb M(z))^{-1})_{i,j}$ for all $i,j$. Then we have
        for all $i, j, k, l$,  $\rho_{ij}\leq \rho_{il}\leq \rho_{kl}$.
        Hence all entries of $(\Id - \mathbb M(z))^{-1}$ have the same radius of convergence.
\end{proof}
We move on to the proof of  \cref{Prop:AsymptotiqueLinear}.
\begin{proof}
  The uniqueness of the solution $\mathbf Y(z)$ directly comes from the resolution of the equation defining  $\mathbf Y(z)$.
  Since the system is strongly connected, then the matrix $\mathbb M$ is irreducible, \emph{i.e} for $i,j$, there is a positive integer $k$ such that $(\mathbb M^k)_{i,j}>0$ and my hypothesis $\mathbb M(0)=\mathbf 0$.
Therefore by \cref{lem:rcv} all entries of $(\Id-\mathbb M(z))^{-1}$ are nonzero series, with the same radius of convergence. As a result, since moreover the entries of $\mathbf V(z)$ are series with nonnegative coefficients, all entries of $\mathbf Y(z) = (\Id - \mathbb M(z))^{-1}\mathbf V(z)$ have the same radius of convergence.
	
	By Perron-Frobenius theorem, the spectral radius $\lambda(t) = \SpectralRadius_{\mathbb M(t)}$, called {\em the Perron eigenvalue}, is a simple eigenvalue of $\mathbb M(t)$ and forms a continuous and strictly increasing function of $t$ on $[0, R_{\mathbb M})$, where $ R_{\mathbb M}$ is the smallest radius of convergence of the entries of $\mathbb M$.
	
	Now assume statement ii). Since $\mathbf V(z)$ has nonnegative entries and a radius of convergence strictly larger than $\rho$, then  $\rho$ is necessarily the common radius of convergence of all the entries of $(\Id-\mathbb M(z))^{-1}$. If $\lambda(\rho)<1$, then $\Id -\mathbb M(z)$ would be analytically invertible around $\rho$ thanks to the comatrix formula, since the entries of $\mathbb M$ are analytic near $\rho$. But this negates Pringsheim's theorem \cite[Theorem IV.6 p.240]{Violet}. As a result $\lambda(\rho)\geq 1$ which implies statement i).
	
	Conversely assume statement i). Then $\alpha = \inf\{t\geq 0,  \lambda(t)=1\}$ is well-defined. Since $\lambda(0) = 0$, then $\alpha>0$, and by continuity, $\lambda(\alpha) = 1$.	Since the coefficients of $\mathbb M$ are series with nonnegative coefficients, then for $\mathbb |z|<\alpha$, $|\mathbb M(z)| \leq \mathbb M(|z|)$ coefficient-wise, hence $\SpectralRadius_{\mathbb M(z)} <1$.
	Because furthermore the radius of convergence of $\mathbb M$ and $\mathbf V$ is larger than $\alpha$, then $(\Id - \mathbb M(z))^{-1}$ and $\mathbf Y$ are defined and analytic on $D(0,\alpha)$ and $\rho \geq \alpha$. We will now compute their asymptotics as $z\to\alpha$. They will turn up to be divergent, which will imply $\alpha = \rho$ and hence statement ii).
	
	By hypothesis, the Perron eigenvalue of $\mathbb M(\alpha)$ is $1$. Denote by $\mathbf u$ and $\mathbf v$ the corresponding left and right positive eigenvectors normalized so that $\transpose{\mathbf u}\mathbf v=1$.
	Let $\mathbb P$ be a Jordanization basis for $\mathbb M(\alpha)$, so that $\mathbb P^{-1} \mathbb M(\alpha) \mathbb P = \diag(1, \mathbb J)$, where $\mathbb J$ is some $(c-1)\times(c-1)$ Jordan matrix that does not admit the eigenvalue $1$. (We write $\diag(A,B)$ for the block-diagonal concatenation of two square matrices $A,B$.)

	Necessarily $\mathbb P\mathbf \mathbf e_1 = \mathbf v$. Moreover, $\transpose{\mathbf e_1}\mathbb P^{-1}$ is a left eigenvector of $\mathbb M$ and $\transpose{\mathbf e_1}\mathbb P^{-1}\mathbf v = 1$.
	Therefore $\transpose{\mathbf u} = \transpose{\mathbf e_1}\mathbb P^{-1}$.
	
We also have that 
$$
\mathbb P^{-1} (\Id_c -\mathbb M(\alpha)) \mathbb P = \diag(0, \Id_{c-1}-\mathbb J)
$$
where $\Id_{d}$ is the identity matrix of size $d$.
Of course $\det( \Id_{c-1}-\mathbb J  ) \neq 0$. Recall that $\mathbb M$ is analytic at $\alpha$. Hence as $z\to \alpha$,
	\[
	\mathbb P^{-1} (\Id_c -\mathbb M(z)) \mathbb P
	=  \left[\begin{matrix} C (\alpha - z) + o(\alpha - z) & \mathcal{O}(\alpha - z)\\
	\mathcal{O}(\alpha - z) & (\Id_{c-1}-\mathbb J) + \mathcal{O}(\alpha - z)
	\end{matrix}\right],
	\]
where $C = (\mathbb P^{-1}\mathbb M'(\alpha)\mathbb P)_{11} = \transpose{\mathbf e_1} \mathbb P^{-1}\mathbb M'(\alpha)\mathbb P\mathbf e_1 = \transpose{\mathbf u} \mathbb M'(\alpha) \mathbf v$, $\mathbb M'(z)$ being the component-wise derivative of $\mathbb M(z)$.
	This last quantity is positive since $\mathbf u$ and $\mathbf v$ have positive coefficients and $\mathbb M'(\alpha)$ is a nonnegative matrix and is not equal to zero. 
	Now we deduce that
	\[
	\det(\Id_c - \mathbb M(z)) = C \det( \Id_{c-1}-\mathbb J)(\alpha - z) + o(\alpha - z).
	\]
	This implies that we can find a neighborhood $B(\rho,\epsilon)$ of $\rho$ such that $(\Id - \mathbb M(z))^{-1}$ can be analytically continued on $B(\rho,\epsilon) \setminus \{\rho\}$.
	We also estimate the transpose of the cofactor matrix as follows:	
	\[
	\Com(\mathbb P^{-1} (\Id_c -\mathbb M(z)) \mathbb P)^t = \left[\begin{matrix} \det( \Id_{c-1}-\mathbb J)+\mathcal{O}(\alpha - z) & \mathcal{O}(\alpha - z)\\
	\mathcal{O}(\alpha - z) & \mathcal{O}(\alpha - z)
	\end{matrix}\right].	\]
	Now we can estimate the inverse of our matrix:
	\[
	(\mathbb P^{-1} (\Id_c -\mathbb M(z)) \mathbb P)^{-1} = \frac {\Com(\mathbb P^{-1} (\Id_c -\mathbb M(z)) \mathbb P)^t}{	\det(\Id_c - \mathbb M(z))} \sim \frac {1} {C(\alpha - z)} \left[\begin{matrix} 
	1 + o(1)&  o(1)\\
	o(1) & o(1)
	\end{matrix}\right]
	\]
	And 
	\[(\Id_c -\mathbb M(z))^{-1} =  \frac {1} {C(\alpha - z)} \mathbb P^{-1}\left(\left[\begin{matrix} 1 &  0\\
	0 & 0
	\end{matrix}\right] + o(1)\right)\mathbb P = \frac {\mathbf v\transpose{\mathbf u}+o(1)}{C(\alpha - z)}.\]
	Consequently the entries are divergent series at $z=\alpha$, therefore $\alpha=\rho$. This gives the asymptotics in \cref{eq:annexe_linear_asymp_inverse} for $(\Id_c -  \mathbb M(z))^{-1}$ near $\rho$. Multiplying by $\mathbf V(z)$, which is analytic at $z=\rho$, gives \cref{eq:annexe_linear_asymp_Y}.

	We are left to show that the aperiodicity condition implies that there is no other singularity than $\rho$ on the circle of convergence for $(\Id_c -  \mathbb M(z))^{-1}$.
	Let $z\neq \rho$, $|z|=\rho$. We just need to show that $(\Id_c -  \mathbb M(z))$ is invertible. Since we only have positive series, we have the coefficient-wise inequality $|\mathbb M(z)| \leq \mathbb M(\rho)$. Since the g.c.d. of the periods of the coefficients of $\mathbb M$ is $1$, it follows from the Daffodil lemma \cite[Lem. IV.1]{Violet} the inequality is strict in at least one coefficient. Then from Perron-Frobenius theorem we know that $\SpectralRadius_{|\mathbb M(z)|}< \SpectralRadius_{\mathbb M(\rho)} = 1$. Using $\SpectralRadius_{\mathbb M}\leq \SpectralRadius_{|\mathbb M|}$ we conclude on the invertibility of $(\Id -  \mathbb M(z))$ around $z$.
	
	The existence of a $\Delta$-domain at $\rho$ follows from a classic compactness argument.	
\end{proof}

\smallskip

\subsection{Nonlinear systems and Drmota-Lalley-Woods theorem}~
	In this section we state and prove a version of the classical Drmota-Lalley-Woods theorem.
	In a classical form \cite[Theorem VII.6, p.489]{Violet}, it entails that polynomial, irreducible and nonlinear tree-specifications
	lead to a common square-root singularity for all series.
	Our result (\cref{Thm:DLW}) is based on a version by Drmota \cite[Theorem 2.33]{Drmota},
	which is stated for analytic specifications, under a suitable analyticity condition.
	We explicitly computed the constants of the square-root term $\sqrt{\rho - z}$ for the tree series,
	along with asymptotics written as a rank one matrix times $(\rho-z)^{-1/2}$ for the natural transfer matrix associated to the system.
	
	The version of Drmota considers series with an additional counting parameter, which we dropped as it is not needed for our purposes. Also, the combinatorial assumptions on the system that ensure uniqueness of the solution differ from ours, as will be discussed in the proof of \cref{Thm:DLW}.
	
\begin{theorem}\label{Thm:DLW}
	Consider the following system:
	\begin{equation}
	 \tag{\ref{eq:annex_generic_system}}
	 \mathbf Y(z) = \mathbf \Phi(z,\mathbf Y(z)), 	 
	\end{equation} where $\mathbf \Phi(z,\mathbf y) = (\Phi_1(z,\mathbf y),\ldots ,\Phi_c(z,\mathbf y))$ is a vector of multivariate power series of $(z,\mathbf y)$ with nonnegative integer coefficients.
	We consider the Jacobian matrix of the system in its second argument:
	\begin{equation*}
	\label{eq:annex_branching_def_M}
	\mathbb M(z,\mathbf y) = \operatorname{Jac}_{\mathbf y} \mathbf \Phi(z,\mathbf y), \quad \emph{i.e. }
	M_{i,j}(z,\mathbf y) = \frac {\partial \Phi_i(z,\mathbf y)} {\partial y_j}, 1\leq i,j \leq c.
	\end{equation*}
        Assume that 
	\begin{enumerate}
	\item $\mathbf \Phi(0,\mathbf 0) = 0$,  $\mathbb M(0,\mathbf 0)$ is the zero matrix
           and $\mathbf \Phi(z,\mathbf 0)$ is nonzero.
         	\item $\mathbf \Phi$ is not linear in its second argument,
		\item $\mathbf \Phi$ is strongly connected.
	\end{enumerate}
        
Then there is a unique solution $\mathbf Y$ of \eqref{eq:annex_generic_system} in the ring of formal power series with no constant term. All its entries have nonnegative coefficients and the same radius of convergence $\rho<\infty$ and the entries of $\mathbf Y(\rho)$ are finite.

The two following assertions are then equivalent:
\begin{enumerate}
	\setcounter{enumi}{3}
	\item There exists $(z,\mathbf y)$ in the region of convergence of $\mathbf \Phi$, such that $\mathbf y = \Phi(z,\mathbf y)$ and $\mathbb M(z,\mathbf y)$ has dominant eigenvalue $1$.
	\item $(\rho,\mathbf Y(\rho))$ belongs to the interior of the region of convergence of $\mathbf \Phi$.
\end{enumerate}

And if these conditions hold, then $\rho>0$ and
	\begin{enumerate}
		\setcounter{enumi}{5}
		\item all entries of $\mathbf Y$ and $(\Id - \mathbb M(z,\mathbf Y(z)))^{-1}$ have radius of convergence $\rho$ and are analytic on a $\Delta$-neighborhood of $\rho$.
		\item $\mathbb M(\rho, \mathbf Y(\rho))$ is an irreducible matrix with Perron eigenvalue $1$.
	\end{enumerate}
	Denote by $\mathbf u$ and $\mathbf v$ the left and right eigenvectors of $\mathbb M(\rho, \mathbf Y(\rho))$ for the eigenvalue 1, chosen positive and normalized so that $\transpose{\mathbf u}\mathbf{v}=1$. 
Let
\begin{equation*}
\forall \, 1 \leq i,j,j'\leq c, \quad  H_{i,j,j'}(z)= \frac {\partial \Phi_i}{\partial y_j \partial y_{j'}}(z,\mathbf y)\bigg|_{\mathbf y = \mathbf Y(z)}\quad  \text{ and } \quad 
\mathbf U(z) =\frac {\partial  \mathbf \Phi}{\partial z}(z,\mathbf y)\bigg|_{\mathbf y = \mathbf Y(z)}
\end{equation*}
	Defining the following positive constants,
		\begin{equation*}
		\beta = \sqrt{\transpose{\mathbf u} \mathbf U(\rho)}, \quad 
		Z = \frac 1 2 \sum_{i,j,j'\in I^\star} u_iv_jv_{j'} H_{i,j,j}(\rho), \quad \zeta = \sqrt{Z},
		\end{equation*}
	we then have the following asymptotics near $\rho$:
	\begin{align}
	\mathbf Y(z) &= \mathbf Y(\rho) - \frac{ \beta \mathbf v}{ \zeta} \sqrt{\rho - z} + o(\sqrt{\rho - z}),
	\label{eq:annexe_AsympYBranching}\\
		\mathbf Y'(z) &\sim \frac{ \beta \mathbf v}{2 \zeta \sqrt{\rho - z}},
	\label{eq:annexe_AsympY'Branching}\\ 
	(\Id - \mathbb M(z,\mathbf Y(z)))^{-1} &\sim \frac{ \mathbf v\transpose{\mathbf u}}{2 {\beta \zeta} \sqrt{\rho - z}}.
	\label{eq:annexe_AsympInverseBranching}
		\end{align}	
	Finally if all series $Y_i(z)$ are aperiodic, then $\rho$ is the unique dominant singularity of the $Y_i$'s and of the series in $(\Id - \mathbb M(z,\mathbf Y(z)))^{-1}$, and these series are analytic on a $\Delta$-domain at $\rho$.
\end{theorem}

\begin{proof}
	First let us show that hypothesis i) implies existence and uniqueness of a solution with no constant term. Because $\mathbf \Phi$ itself has no constant term, the map $\mathbf Y \to \mathbf \Phi(z,\mathbf Y)$ sends the ring of series with no constant term to itself. Moreover, since there are no monomials of degree 1 involving just one $y_i$, this is a contraction mapping.  Therefore by the fixed point theorem a solution  exists and it is nonzero because of the assumption $\mathbf \Phi(z,0) \neq \mathbf 0$.
	
All entries of $\mathbf Y$ have the same radius of convergence. Indeed, iterating $\mathbf \Phi$ enough and using Hypotheses ii) and iii), we get that each $Y_i$ depends positively and nonlinearly on every other $Y_j$'s.
More precisely for each $Y_i$, there exist  $c>0$ and $k\geq 0$ such that $cz^kY_i^2$ is coefficient-wise dominated by $Y_i$. Hence $Y_i$ cannot be a polynomial, so $\rho<\infty$, and $Y_i(\rho)$ must be finite.
	
	For $0\leq t\leq \rho$, let us now set $\lambda(t) = \SpectralRadius_{\mathbb M(t, \mathbf Y(t))}$. By Perron-Frobenius theorem, this is an increasing, continuous function. We will show that statement v) implies statement iv).
	Assume that $\Phi$ is analytic at $(\rho, \mathbf Y(\rho))$, and suppose that the $\lambda(\rho)<1$. Then $\det(\Id - M(\rho, \mathbf Y(\rho))\neq 0$, and the analytic implicit function theorem would imply that $\mathbf Y$ could be continued on a neighborhood of $\rho$.
	Thanks to Pringsheim's theorem \cite[Thm IV.5]{Violet}, this is in contradiction with the fact that $\rho$ is the radius of convergence of $\mathbf Y$.
	Hence the $\lambda(\rho)\geq1$, and there exists $z_0\leq\rho$ such that  $\lambda(z_0)=1$ as stated in iv).

For the rest of the proof, we assume statement iv). We apply Theorem 2.33 of \cite{Drmota}. The hypotheses of this theorem are all verified, except (in our notation) $\mathbf \Phi(0,\mathbf y) = 0$, which we replaced by the weaker one $\mathbb M(0,\mathbf 0) = 0$. In the proof of Drmota, this hypothesis was only used to guarantee the uniqueness of the solution $\mathbf Y$ as a formal power series in $z$. However as we saw, our set of hypotheses still guarantees uniqueness of the solution, when restricted to series with no constant term. As a result, Theorem 2.33 of \cite{Drmota} guarantees that $z_0 = \rho$ (hence statement v)), and that $\mathbf Y$ can be continued on a $\Delta$-neighborhood of $\rho$. It also implies that there exists a positive vector $\mathbf c$ such that the following asymptotics holds:
	\begin{equation}
		\mathbf Y(z) = \mathbf Y(\rho) - (\mathbf c + o(1)) \sqrt{\rho-z}.
		\label{eq:Tstar_branching_nonexplicit_constant}
	\end{equation}

Since $\lambda(\rho) = 1$, the radius of convergence of $(\Id - \mathbb M(z,\mathbf Y(z)))^{-1}$ is at least $\rho$.
We will now compute the precise asymptotics of $(\Id - \mathbb M(z,\mathbf Y(z)))^{-1}$ and $\mathbf Y(z)$ when $z$ is near $\rho$. The fact that $(\Id - \mathbb M(z,\mathbf Y(z)))^{-1}$ can be analytically continued on a $\Delta$-neighborhood of $\rho$ will be obtained as a byproduct of this derivation.
	
Let us denote $\mathbb A = \mathbb M(\rho,\mathbf Y(\rho))$. This is an irreducible nonnegative matrix with Perron eigenvalue $1$. As in the linear case, the Perron-Frobenius theorem provides corresponding left and right positive eigenvectors $\mathbf u$ and $\mathbf v$ normalized so that $\transpose{\mathbf u}\mathbf v=1$. Let also $\mathbb P$ be a Jordanization basis for $\mathbb A$, so that $\mathbb P^{-1} \mathbb A \mathbb P = \diag(1, \mathbb T)$, and $\mathbb T$ is some Jordan matrix with spectral radius less than $1$. Necessarily $\mathbb P \mathbf e_1 = \mathbf v$ and $\transpose{\mathbf u} = \transpose{\mathbf e_1}\mathbb P^{-1}$.
	
	We get that $\mathbb P^{-1} (\Id_c -\mathbb A) \mathbb P = \diag(0, \Id_{c-1}-\mathbb T)$, and $\det(\Id_{c-1}-\mathbb T) \neq 0$. Recall that each coefficient of the matrix $\mathbb M(z,\mathbf y)$ is analytic at $(\rho, \mathbf Y(\rho))$. Hence as $z\to \rho$,
	\begin{align*}
	M_{i,j}(z,\mathbf Y(z)) = M_{i,j}(\rho,\mathbf Y(\rho)) 
	&- \frac {\partial M_{i,j}}{\partial z}(\rho,\mathbf Y(\rho)) (\rho-z)(1+o(1)) \\
	&- \sum_{j'=1}^c\frac {\partial M_{i,j}}{\partial
		y_{j'}}(\rho,\mathbf Y(\rho))( Y_{j'}(\rho) -  Y_{j'}(z))(1+o(1))
	\end{align*}
	The second term , which is linear, is dominated by the third one, whose square-root behavior is given by \cref{eq:Tstar_branching_nonexplicit_constant}. Also, we have 
	\[\frac {\partial  M_{i,j}}{\partial
		y_{j'}}(z,\mathbf Y(z)) = \frac {\partial  \Phi_i}{\partial
		y_{j}\partial
		y_{j'}}(z,\mathbf Y(z))= H_{i,j,j}(z).\]
        Note that the nonlinearity of $\mathbf \Phi$ implies that at least one of the series $H_{i,j,j}$ is nonzero.
        
	Collecting everything we get the following asymptotics near $\rho$ for the matrix $\mathbb M (z,\mathbf Y(z))$:
	\begin{equation*}
	M_{i,j}(z,\mathbf Y(z))= A_{ij} - \sqrt{\rho-z}\sum_{j'=1}^c H_{i,j,j}(\rho)c_{j'} +o(\sqrt{\rho-z}).
	\end{equation*}
	Hence as $\rho \to z$, we have the following asymptotics written in block-decomposition:
	\[
	\mathbb P^{-1} (\Id -\mathbb M(z,\mathbf Y(z)))\mathbb P
	=  \left[\begin{matrix} (C +o(1)) \sqrt{\rho-z}  & \mathcal{O}(\sqrt{\rho-z})\\
	\mathcal{O}(\sqrt{\rho-z}) & (\Id_{c-1}-\mathbb T) + \mathcal{O}(\sqrt{\rho-z})
	\end{matrix}\right],
	\]
	where 
	\[
	C 
	= \lim_{z\to\rho}(\mathbb P^{-1}\frac{\mathbb A-\mathbb M(z,\mathbf Y(z))}{\sqrt{\rho-z}}\mathbb P)_{11} 
	=\lim_{z\to\rho} \transpose{\mathbf u} \frac{\mathbb A
		-\mathbb M(z,\mathbf Y(z))}{\sqrt{\rho-z}} v 
	= \sum_{i,j,j'\in I^\star} u_iv_jc_{j'}H_{i,j,j}(\rho).\]
	We then proceed as in the linear case. The asymptotic estimate of the determinant near $\rho$
        \[
	\det(\Id_c - \mathbb M(z)) = C \det( \Id_{c-1}-\mathbb T)\sqrt{\rho - z} + o(\sqrt{\rho - z}).
	\]
shows it does not vanish on a punctured neighborhood of $\rho$. Hence $(\Id - \mathbb M(z,\mathbf Y(z)))$ is invertible on a (possibly smaller) $\Delta$-neighborhood of $\rho$. Then using the comatrix formula for the inverse, we obtain
	\begin{equation}
	(\Id - \mathbb M(z,\mathbf Y(z)))^{-1}\sim \frac{ \mathbf v\transpose{\mathbf u}}{C \sqrt{\rho - z}}.
	\label{eq:AsympTstarflecheBranching_nonexplicit_constant}
	\end{equation}
	We proceed to transfer this asymptotics into asymptotics for $\mathbf Y'(z)$. Differentiation of the relation \eqref{eq:annex_generic_system} yields
	\begin{align*}
	\mathbf Y' (z)
	&= \frac{\partial \mathbf \Phi}{\partial z}(z,\mathbf y)\Big|_{\mathbf y = \mathbf Y(z)}+ \mathrm{Jac}_{\mathbf y} \mathbf \Phi (z,\mathbf y)\Big|_{\mathbf y = \mathbf Y(z)} \cdot \mathbf Y'(z) \\
	&= \mathbf U(z) + \mathbb M(z,\mathbf Y(z))\mathbf Y'(z).
	\end{align*}
         Note that Hypotheses i) and iii) imply  that $\mathbf U(z)$ is nonzero too.
	Hence
	\begin{equation}\mathbf Y'(z) = (\Id - \mathbb M(z,\mathbf Y(z)))^{-1} \mathbf U(z).
	\label{eq:lien_U_Tstar'}
	\end{equation}
	Now, since $\mathbf U$ is convergent at $\rho$, with \cref{eq:AsympTstarflecheBranching_nonexplicit_constant}, we obtain
	\begin{equation}\mathbf Y'(z) \sim \frac{ \transpose{\mathbf u} \mathbf U(\rho)}{C }\frac{\mathbf v}{\sqrt{\rho - z}} = \frac{\beta^2}{C}\frac{\mathbf v}{\sqrt{\rho - z}} .
	\label{eq:AsympTstar'Branching_nonexplicit_constant}
	\end{equation}
	Since $\mathbf Y$ is analytic on a $\Delta$-neighborhood at $\rho$, singular differentiation \cite[Thm VI.8]{Violet} of \cref{eq:Tstar_branching_nonexplicit_constant} yields
	\[\mathbf Y'(z) \sim \frac {\mathbf c} {2\sqrt{\rho-z}}.\]
	We can identify the constants in the two expressions and get 
	$\mathbf c = \frac {2\beta^2}{C} \mathbf v$, which can be reinjected in the definition of $C$, yielding 
	$C^2 = 2\beta^2 \sum_{i,j,j'\in I^\star} u_iv_jv_{j'}H_{i,j,j}(\rho) = 4 \beta^2 Z $ and then $C = 2\beta\zeta$.
	Substituting this value for $C$ into \cref{eq:Tstar_branching_nonexplicit_constant,eq:AsympTstarflecheBranching_nonexplicit_constant,eq:AsympTstar'Branching_nonexplicit_constant} yields the desired asymptotics.
		
	We shall now show that there is no other singularity on the circle of convergence under the aperiodicity condition, in a similar fashion to the linear case. Let $z\neq \rho$ be such that $|z|=\rho$. By the Daffodil lemma \cite[Lem. IV.1]{Violet}, we have $|\mathbf Y(z)|<\mathbf Y(\rho)$. Hence $\SpectralRadius_{\mathbb M(z, \mathbf Y(z))} \leq \SpectralRadius_{\mathbb M(|z|, |\mathbf Y(z)|)} < \SpectralRadius_{\mathbb M(\rho, \mathbf Y(\rho))}=1$. By the analytic implicit function theorem \cite[Thm B.6]{Violet}, this implies that $\mathbf Y$ is analytic near $z$. And $(\Id-\mathbb M(w, \mathbf Y(w))$ is then invertible near $z$. The existence of a $\Delta$-domain at $\rho$ once again follows from a classic compactness argument.
\end{proof}

\section{Details on the examples}\label{Sec:AppendixExamples}
~\\

\subsection{The class $\Av(2413,3142,2314,3241,21453,45213)$}\label{Ex:ClasseUnion}

The algorithm of \cite{BBPPR} gives for this class a specification\footnote{See the \href{http://mmaazoun.perso.math.cnrs.fr/pcfs/} {companion Jupyter notebook}  \texttt{examples/Union.ipynb}} with 14 equations, for families $\TTT=\TTT_0, \dots, \TTT_{13}$. 
The family $\TTT_{10}$ is however empty, as we will explain in \cref{rk:ClasseUnion} below. 
Removing it from the obtained specification yields the following one: 

{\small
\begin{equation}
\begin{cases}
\begin{array}{rl}
  \mathcal T =
  \mathcal T_{0}= &\{ \bullet \} \uplus \oplus[\mathcal T_{1},\mathcal T_{2}]\uplus \oplus[\mathcal T_{1},\mathcal T_{3}]\uplus \oplus[\mathcal T_{4},\mathcal T_{2}]\uplus \ominus[\mathcal T_{1},\mathcal T_{5}]\uplus \ominus[\mathcal T_{1},\mathcal T_{6}]\uplus \ominus[\mathcal T_{7},\mathcal T_{5}]\\
\mathcal T_{1}= &\{ \bullet \} \\
\mathcal T_{2}= &\{ \bullet \} \uplus \oplus[\mathcal T_{8},\mathcal T_{2}]\uplus \ominus[\mathcal T_{1},\mathcal T_{2}]\\
\mathcal T_{3}= &\oplus[\mathcal T_{1},\mathcal T_{3}]\uplus \ominus[\mathcal T_{7},\mathcal T_{9}]\uplus \ominus[\mathcal T_{1},\mathcal T_{9}]\uplus \ominus[\mathcal T_{7},\mathcal T_{11}]\\
\mathcal T_{4}= &\ominus[\mathcal T_{1},\mathcal T_{2}]\\
\mathcal T_{5}= &\{ \bullet \} \uplus \oplus[\mathcal T_{1},\mathcal T_{5}]\uplus \ominus[\mathcal T_{12},\mathcal T_{5}]\\
\mathcal T_{6}= &\oplus[\mathcal T_{4},\mathcal T_{13}]\uplus \oplus[\mathcal T_{1},\mathcal T_{13}] \uplus \oplus[\mathcal T_{4},\mathcal T_{11}]\uplus \ominus[\mathcal T_{1},\mathcal T_{6}]\\
\mathcal T_{7}= &\oplus[\mathcal T_{1},\mathcal T_{5}]\\
\mathcal T_{8}= &\{ \bullet \} \uplus \ominus[\mathcal T_{1},\mathcal T_{2}]\\
\mathcal T_{9}= &\oplus[\mathcal T_{1},\mathcal T_{9}]\uplus \ominus[\mathcal T_{7},\mathcal T_{9}]\uplus \ominus[\mathcal T_{1},\mathcal T_{9}]\uplus \ominus[\mathcal T_{7},\mathcal T_{11}]\\
\mathcal T_{11}= &\{ \bullet \} \uplus \oplus[\mathcal T_{1},\mathcal T_{11}]\uplus \ominus[\mathcal T_{1},\mathcal T_{11}]\\
\mathcal T_{12}= &\{ \bullet \} \uplus \oplus[\mathcal T_{1},\mathcal T_{5}]\\
\mathcal T_{13}= &\oplus[\mathcal T_{4},\mathcal T_{13}]\uplus \oplus[\mathcal T_{1},\mathcal T_{13}]\uplus \oplus[\mathcal T_{4},\mathcal T_{11}]\uplus \ominus[\mathcal T_{1},\mathcal T_{13}]. 
\end{array}
\end{cases}
\label{eq:SpecifClasseUnion}
\end{equation}
}

\begin{remark}~\label{rk:ClasseUnion}
In the specification obtained from the algorithm of~\cite{BBPPR} (not displayed), 
the family abbreviated $\mathcal T_{10}$ is actually $\mathcal T_{(213,231)}$, 
which consists of permutations of the class $\mathcal T$ forced to contain the patterns $213$ and $231$. 
From the characterization of $\mathcal T $ as $\Av(213) \cup \Av(231)$, it is clear $\mathcal T_{10}$ has to be empty. 
The algorithm of~\cite{BBPPR} is however not able to detect this simplification, 
and we had to perform this simplification by hand.
\end{remark}

Translating this specification into a system on the corresponding series, and solving this system, we get 
\[
\begin{cases}
\begin{array}{rl}
 T =  T_{0}= & \tfrac{-3z^2 - 2z\sqrt{-4z + 1} + 4z + \sqrt{-4z + 1} - 1}{z(2z - 1)}\\
 T_{1}= & z\\ %
 T_{2}=T_5= &\tfrac{-2z -\sqrt{-4z + 1} + 1}{2z}\\
 T_{3}= T_6= T_{9}=T_{13}= &\tfrac{-z^2 - z\sqrt{-4z + 1} + 2z + \sqrt{-4z + 1}/2 - 1/2}{z(2z - 1)}\\
 T_{4}=T_7= &-z - \tfrac{\sqrt{-4z + 1}}{2} + \tfrac{1}{2}\\
 T_{8}=T_{12}= &-\tfrac{\sqrt{-4z + 1}}{2} + \tfrac{1}{2}\\
 T_{11}= &\tfrac{-z}{2z - 1}
\end{array}
\end{cases}
\]

The dominant singularity is of square-root type, coming from $\sqrt{-4z + 1}$.
All series above except $T_1$ and $ T_{11}$ are critical, with radius of convergence  $\rho=1/4$. 
Due to the presence of (for instance) the term $T_4 T_2$ in the equation for $T_0$, 
the specification~\eqref{eq:SpecifClasseUnion} is essentially branching. 
Its dependency graph restricted to the critical $\mathcal T_i$ is shown in \cref{fig:DependendyGraphUnion}
(p.\pageref{fig:DependendyGraphUnion})
and has nine strongly connected components.
From this specification and this system, 
we obtained the limiting permuton of this class in \cref{ssec:couteau_suisse_union}.
\smallskip

\subsection{The class $\Av(2413,3142, 2143,34512)$}\label{sec:ClasseXTildeAnnexe}

The specification for this class that we obtain applying the algorithm of~\cite{BBPPR} is\footnote{See the \href{http://mmaazoun.perso.math.cnrs.fr/pcfs/} {companion Jupyter notebook}  \texttt{examples/AsymmetricX.ipynb}}
{\footnotesize
\begin{equation}
\begin{cases}
\begin{array}{rl}
\mathcal T = \mathcal T_{0}= &\{ \bullet \} \uplus \oplus[\mathcal T_{1},\mathcal T_{2}]\uplus \oplus[\mathcal T_{1},\mathcal T_{3}]\uplus \oplus[\mathcal T_{4},\mathcal T_{2}]\uplus \ominus[\mathcal T_{5},\mathcal T_{6}]\uplus \ominus[\mathcal T_{5},\mathcal T_{7}]\uplus \ominus[\mathcal T_{8},\mathcal T_{6}]\\
\mathcal T_{1}= &\{ \bullet \} \\
\mathcal T_{2}= &\{ \bullet \} \uplus \oplus[\mathcal T_{1},\mathcal T_{2}]\\
\mathcal T_{3}= &\oplus[\mathcal T_{1},\mathcal T_{3}]\uplus \oplus[\mathcal T_{4},\mathcal T_{2}]\uplus \ominus[\mathcal T_{5},\mathcal T_{6}]\uplus \ominus[\mathcal T_{5},\mathcal T_{7}]\uplus \ominus[\mathcal T_{8},\mathcal T_{6}]\\
\mathcal T_{4}= &\ominus[\mathcal T_{5},\mathcal T_{6}]\uplus \ominus[\mathcal T_{5},\mathcal T_{7}]\uplus \ominus[\mathcal T_{8},\mathcal T_{6}]\\
\mathcal T_{5}= &\{ \bullet \} \uplus \oplus[\mathcal T_{1},\mathcal T_{1}]\uplus \oplus[\mathcal T_{1},\mathcal T_{9}]\uplus \oplus[\mathcal T_{9},\mathcal T_{1}]\\
\mathcal T_{6}= &\{ \bullet \} \uplus \ominus[\mathcal T_{1},\mathcal T_{6}]\\
\mathcal T_{7}= &\oplus[\mathcal T_{1},\mathcal T_{2}]\uplus \oplus[\mathcal T_{1},\mathcal T_{3}]\uplus \oplus[\mathcal T_{4},\mathcal T_{2}]\uplus \ominus[\mathcal T_{10},\mathcal T_{6}]\uplus \ominus[\mathcal T_{10},\mathcal T_{7}]\uplus \ominus[\mathcal T_{1},\mathcal T_{7}]\uplus \ominus[\mathcal T_{8},\mathcal T_{6}]\\
\mathcal T_{8}= &\oplus[\mathcal T_{1},\mathcal T_{11}]\uplus \oplus[\mathcal T_{1},\mathcal T_{12}]\uplus \oplus[\mathcal T_{13},\mathcal T_{11}]\uplus \oplus[\mathcal T_{9},\mathcal T_{11}]\uplus \oplus[\mathcal T_{13},\mathcal T_{1}]\\
\mathcal T_{9}= &\ominus[\mathcal T_{1},\mathcal T_{6}]\\
\mathcal T_{10}= &\oplus[\mathcal T_{1},\mathcal T_{1}]\uplus \oplus[\mathcal T_{1},\mathcal T_{9}]\uplus \oplus[\mathcal T_{9},\mathcal T_{1}]\\
\mathcal T_{11}= &\oplus[\mathcal T_{1},\mathcal T_{2}]\\
\mathcal T_{12}= &\oplus[\mathcal T_{1},\mathcal T_{3}]\uplus \oplus[\mathcal T_{4},\mathcal T_{2}]\uplus \ominus[\mathcal T_{10},\mathcal T_{6}]\uplus \ominus[\mathcal T_{10},\mathcal T_{7}]\uplus \ominus[\mathcal T_{1},\mathcal T_{7}]\uplus \ominus[\mathcal T_{8},\mathcal T_{6}]\\
\mathcal T_{13}= &\ominus[\mathcal T_{10},\mathcal T_{6}]\uplus \ominus[\mathcal T_{10},\mathcal T_{7}]\uplus \ominus[\mathcal T_{1},\mathcal T_{7}]\uplus \ominus[\mathcal T_{8},\mathcal T_{6}].
\end{array}
\end{cases}
\label{eq:SpecifClasseXTilde}
\end{equation}
}

Solving the system on the series $(T_i)_{0 \leq i \leq 13}$ resulting from \cref{eq:SpecifClasseXTilde} gives 
\[
\begin{cases}
\begin{array}{rl}
 T =  T_{0}= &\tfrac{-z(z^3 - z^2 + 3z - 1)}{(z - 1)(z^3 - z^2 + 4z - 1)} \\
 T_{1}= &z \\
 T_{2}=T_6= &\tfrac{-z}{(z - 1)}\\
 T_{3}=T_7= &\tfrac{z^2}{(z - 1)(z^3 - z^2 + 4z - 1)}\\
 T_{4}= &\tfrac{z^2(z - 1)}{(z^3 - z^2 + 4z - 1)}\\
 T_{5}= &\tfrac{-z(z^2 + 1)}{(z - 1)}\\
 T_{8}= &\tfrac{z^3(z^3 - z^2 + 3z + 1)}{(z - 1)(z^3 - z^2 + 4z - 1)}\\
 T_{9}=T_{11}= &\tfrac{-z^2}{(z - 1)}\\
 T_{10}= &\tfrac{-z^2(z + 1)}{(z - 1)}\\
 T_{12}= &\tfrac{z^3(z^2 - z + 4)}{(z - 1)(z^3 - z^2 + 4z - 1)}\\
 T_{13}= &\tfrac{z^3(z^2 + 2)}{(z - 1)(z^3 - z^2 + 4z - 1)}.
\end{array}
\end{cases}
\]
The critical series are $T_0, T_3, T_4, T_7, T_8, T_{12}$ and $T_{13}$. Their common root $\rho$ is the only real root of the polynomial $z^3 - z^2 + 4z - 1$, namely
\[\rho=-\tfrac{(7/2 + 3\sqrt{597}/2)^{1/3}}{3} + \tfrac{1}{3} + \tfrac{11}{3(7/2 + 3\sqrt{597}/2)^{1/3}} \approx 0.26272.\]
It follows that the specification~\eqref{eq:SpecifClasseXTilde} is essentially linear. 
The dependency graph shows that the critical series are organized into two strongly connected components, 
one of which consists of the class $\TTT_0$ alone.
However, as for the $X$-class (see \cref{sec:ClasseX}), $\TTT_0 = \TTT_3 \uplus \{12\dots n \mid n \geq 1\}$
and we study the specification where the equation for $\TTT_0$ is removed. Again similarly to the $X$-class, the limit of a uniform random permutation of size $n$ in $\TTT_3$ will also be the limit of a uniform random permutation in $\TTT_0$.

From the specification we are able to compute the matrices $\mathbb M^\star$, $\mathbb D^{\gauche,+},\dots,\mathbb D^{\droite,-}$. Namely, 
\[
\mathbb M^\star(z) = \left(\begin{array}{cccccc}
	z & -\frac{z}{z - 1} & -\frac{z^{3} + z}{z - 1} & -\frac{z}{z - 1} & 0 & 0 \\
	0 & 0 & -\frac{z^{3} + z}{z - 1} & -\frac{z}{z - 1} & 0 & 0 \\
	z & -\frac{z}{z - 1} & z - \frac{z^{3} + z^{2}}{z - 1} & -\frac{z}{z - 1} & 0 & 0 \\
	0 & 0 & 0 & 0 & z & z - \frac{z^{2}}{z - 1} \\
	z & -\frac{z}{z - 1} & z - \frac{z^{3} + z^{2}}{z - 1} & -\frac{z}{z - 1} & 0 & 0 \\
	0 & 0 & z - \frac{z^{3} + z^{2}}{z - 1} & -\frac{z}{z - 1} & 0 & 0
\end{array}\right),
\]
\[
\mathbb D^{\gauche,+}=\left(\begin{array}{rrrrrr}
	1 & 0 & 0 & 0 & 0 & 0 \\
	0 & 0 & 0 & 0 & 0 & 0 \\
	1 & 0 & 0 & 0 & 0 & 0 \\
	0 & 0 & 0 & 0 & 1 & 0 \\
	1 & 0 & 0 & 0 & 0 & 0 \\
	0 & 0 & 0 & 0 & 0 & 0
\end{array}\right),
\qquad 
\mathbb D^{\gauche,-}=\left(\begin{array}{rrcrrr}
	0 & 0 & -\frac{3 \, z^{2} + 1}{z - 1} + \frac{z^{3} + z}{{\left(z - 1\right)}^{2}} & 0 & 0 & 0 \\
	0 & 0 & -\frac{3 \, z^{2} + 1}{z - 1} + \frac{z^{3} + z}{{\left(z - 1\right)}^{2}} & 0 & 0 & 0 \\
	0 & 0 & -\frac{3 \, z^{2} + 2 \, z}{z - 1} + \frac{z^{3} + z^{2}}{{\left(z - 1\right)}^{2}} + 1 & 0 & 0 & 0 \\
	0 & 0 & 0 & 0 & 0 & 0 \\
	0 & 0 & -\frac{3 \, z^{2} + 2 \, z}{z - 1} + \frac{z^{3} + z^{2}}{{\left(z - 1\right)}^{2}} + 1 & 0 & 0 & 0 \\
	0 & 0 & -\frac{3 \, z^{2} + 2 \, z}{z - 1} + \frac{z^{3} + z^{2}}{{\left(z - 1\right)}^{2}} + 1 & 0 & 0 & 0
\end{array}\right)
\]
\[
\mathbb D^{\droite,+}=\left(\begin{array}{rcrrrc}
	0 & -\frac{1}{z - 1} + \frac{z}{{\left(z - 1\right)}^{2}} & 0 & 0 & 0 & 0 \\
	0 & 0 & 0 & 0 & 0 & 0 \\
	0 & -\frac{1}{z - 1} + \frac{z}{{\left(z - 1\right)}^{2}} & 0 & 0 & 0 & 0 \\
	0 & 0 & 0 & 0 & 0 & -\frac{2 \, z}{z - 1} + \frac{z^{2}}{{\left(z - 1\right)}^{2}} + 1 \\
	0 & -\frac{1}{z - 1} + \frac{z}{{\left(z - 1\right)}^{2}} & 0 & 0 & 0 & 0 \\
	0 & 0 & 0 & 0 & 0 & 0\end{array}\right)
        \]
        
\[ \text{and } \qquad 
\mathbb D^{\droite,-}=\left(\begin{array}{rrrcrr}
	0 & 0 & 0 & -\frac{1}{z - 1} + \frac{z}{{\left(z - 1\right)}^{2}} & 0 & 0 \\
	0 & 0 & 0 & -\frac{1}{z - 1} + \frac{z}{{\left(z - 1\right)}^{2}} & 0 & 0 \\
	0 & 0 & 0 & -\frac{1}{z - 1} + \frac{z}{{\left(z - 1\right)}^{2}} & 0 & 0 \\
	0 & 0 & 0 & 0 & 0 & 0 \\
	0 & 0 & 0 & -\frac{1}{z - 1} + \frac{z}{{\left(z - 1\right)}^{2}} & 0 & 0 \\
	0 & 0 & 0 & -\frac{1}{z - 1} + \frac{z}{{\left(z - 1\right)}^{2}} & 0 & 0
\end{array}\right)
 \]
By performing the computations in the field $\mathbb Q(\rho)$,
we are able to compute those matrices at $z=\rho$. We verify that the dominant eigenvalue of $\mathbb M^\star(\rho)$ is 1 and compute the corresponding left and right eigenvectors.
and the vector $\mathbf p$:
\[\mathbf p = \frac 1{ 597} \left(51 \rho^{2} + 42 \rho + 105, 51 \rho^{2} + 42 \rho + 105, -113 \rho^{2} + 24 \rho + 259, 11 \rho^{2} - 108 \rho + 128\right).\]
A numerical approximation gives
\[ \mathbf p \approx (0.200258808255625,0.200258808255625,0.431332891374616,0.168149492114135).\]

Those numbers are algebraic of degree $3$ since $\rho$ is.

\subsection{The V-shape: $\Av(2413,1243,2341,531642,41352)$}\label{sec:ClasseVAnnexe}

The specification for this class that we obtain applying the algorithm of~\cite{BBPPR}\footnote{See the \href{http://mmaazoun.perso.math.cnrs.fr/pcfs/} {companion Jupyter notebook}  \texttt{examples/V.ipynb}} is 
\begin{equation*}
\begin{cases}
\begin{array}{rl}
\mathcal T_{0}= &\{ \bullet \} \uplus \oplus[\mathcal T_{1},\mathcal T_{2}]\uplus \oplus[\mathcal T_{1},\mathcal T_{3}]\uplus \oplus[\mathcal T_{4},\mathcal T_{2}]\uplus \ominus[\mathcal T_{5},\mathcal T_{0}]\uplus 3142[\mathcal T_{1},\mathcal T_{1},\mathcal T_{1},\mathcal T_{6}]\\
\mathcal T_{1}= &\{ \bullet \} \uplus \ominus[\mathcal T_{7},\mathcal T_{1}]\\
\mathcal T_{2}= &\{ \bullet \} \uplus \oplus[\mathcal T_{7},\mathcal T_{2}]\\
\mathcal T_{3}= &\oplus[\mathcal T_{8},\mathcal T_{2}]\uplus \ominus[\mathcal T_{9},\mathcal T_{6}]\\
\mathcal T_{4}= &\ominus[\mathcal T_{10},\mathcal T_{11}]\uplus \ominus[\mathcal T_{10},\mathcal T_{1}]\uplus \ominus[\mathcal T_{7},\mathcal T_{11}]\uplus 3142[\mathcal T_{1},\mathcal T_{1},\mathcal T_{1},\mathcal T_{6}]\\
\mathcal T_{5}= &\{ \bullet \} \uplus \oplus[\mathcal T_{1},\mathcal T_{1}]\uplus 3142[\mathcal T_{1},\mathcal T_{1},\mathcal T_{1},\mathcal T_{1}]\\
\mathcal T_{6}= &\{ \bullet \} \uplus \oplus[\mathcal T_{12},\mathcal T_{2}]\uplus \ominus[\mathcal T_{9},\mathcal T_{6}]\\
\mathcal T_{7}= &\{ \bullet \} \\
\mathcal T_{8}= &\ominus[\mathcal T_{9},\mathcal T_{6}]\\
\mathcal T_{9}= &\{ \bullet \} \uplus \oplus[\mathcal T_{1},\mathcal T_{7}]\\
\mathcal T_{10}= &\oplus[\mathcal T_{1},\mathcal T_{1}]\uplus 3142[\mathcal T_{1},\mathcal T_{1},\mathcal T_{1},\mathcal T_{1}]\\
\mathcal T_{11}= &\oplus[\mathcal T_{1},\mathcal T_{2}]\uplus \oplus[\mathcal T_{1},\mathcal T_{3}]\uplus \oplus[\mathcal T_{4},\mathcal T_{2}]\uplus \ominus[\mathcal T_{10},\mathcal T_{11}]\uplus \ominus[\mathcal T_{10},\mathcal T_{1}]\uplus \ominus[\mathcal T_{7},\mathcal T_{11}]\uplus 3142[\mathcal T_{1},\mathcal T_{1},\mathcal T_{1},\mathcal T_{6}]\\
\mathcal T_{12}= &\{ \bullet \} \uplus \ominus[\mathcal T_{9},\mathcal T_{6}]
\end{array}
\end{cases}
\end{equation*}
and the solutions of the associated system are
\begin{equation*}
\begin{cases}
\begin{array}{rl}
	T_0 &= -\frac{z^{7} - 7 \, z^{6} + 20 \, z^{5} - 28 \, z^{4} + 20 \, z^{3} - 7 \, z^{2} + z}{2 \, z^{7} - 13 \, z^{6} + 37 \, z^{5} - 62 \, z^{4} + 59 \, z^{3} - 32 \, z^{2} + 9 \, z - 1} \\
	T_1 =T_2 = T_9 &= -\frac{z}{z - 1} \\
	T_3 &= -\frac{z^{2}}{z^{3} - 4 \, z^{2} + 4 \, z - 1} \\
	T_4 &= \frac{z^{8} - 4 \, z^{7} + 11 \, z^{6} - 13 \, z^{5} + 8 \, z^{4} - 2 \, z^{3}}{2 \, z^{7} - 13 \, z^{6} + 37 \, z^{5} - 62 \, z^{4} + 59 \, z^{3} - 32 \, z^{2} + 9 \, z - 1} \\
	T_5 &= \frac{z^{5} - 2 \, z^{4} + 4 \, z^{3} - 3 \, z^{2} + z}{z^{4} - 4 \, z^{3} + 6 \, z^{2} - 4 \, z + 1} \\
	T_6 &= -\frac{z^{2} - z}{z^{2} - 3 \, z + 1} \\
	T_7 &= z \\
	T_8 &= \frac{z^{2}}{z^{2} - 3 \, z + 1} \\
	T_{10} &= \frac{2 \, z^{4} - 2 \, z^{3} + z^{2}}{z^{4} - 4 \, z^{3} + 6 \, z^{2} - 4 \, z + 1} \\
	T_{11} &= \frac{z^{8} - 5 \, z^{7} + 10 \, z^{6} - 14 \, z^{5} + 11 \, z^{4} - 5 \, z^{3} + z^{2}}{2 \, z^{8} - 15 \, z^{7} + 50 \, z^{6} - 99 \, z^{5} + 121 \, z^{4} - 91 \, z^{3} + 41 \, z^{2} - 10 \, z + 1} \\
	T_{12} &= \frac{z^{3} - 2 \, z^{2} + z}{z^{2} - 3 \, z + 1}
\end{array}
\end{cases}
\end{equation*}
The critical series are $T_0, T_4$ and $T_{11}$, 
whose radius of convergence $\rho$ is the only real root of the polynomial 
\[ 2z^5 - 7z^4 + 14z^3 - 13z^2 + 6z - 1. \] 
The graph of critical series is not strongly connected: $\{\TTT_4,\TTT_{11}\}$ forms a connected component which does not involve $\TTT_0$, hence we can study the specification where $\TTT_0$ is removed. 
It is essentially linear, verifies Hypotheses (SC) and (RC), and involves aperiodic subcritical series. Hence \cref{Th:linearCase} applies and there exists a parameter $\mathbf p$
such that uniform random permutations of size $n$ in
either $\TTT_4$ or $\TTT_{11}$ converges to the $X$-permuton with parameter $\mathbf p$. 

Furthermore, we know from the design of the algorithm of~\cite{BBPPR} that
all families appearing in the system are included in $\TTT_0$,
in particular $\TTT_{11} \subseteq \TTT_0$.
A quick computer-assisted computation (done in the companion notebook)
shows that $T_0-T_{11}=z/(1-z)$,
\emph{i.e.}, for each $n$, there is exactly one permutation of size $n$ in $\TTT_{0} \setminus \TTT_{11}$.
Hence, uniform random permutations of size $n$ in $\TTT_0$
also converge to the $X$-permuton with parameter $\mathbf p$.

We now turn to the computation of the parameter $\mathbf p$,
using \cref{eq:DefProbaCaterpillar}.
From the specification we directly compute
\begin{equation*}
\begin{array}{ll}
\mathbb M^\star(z) =
\left(\begin{array}{rr}
0 & z + \frac{2 \, z^{4} - 2 \, z^{3} + z^{2}}{z^{4} - 4 \, z^{3} + 6 \, z^{2} - 4 \, z + 1} \\
-\frac{z}{z - 1} & z + \frac{2 \, z^{4} - 2 \, z^{3} + z^{2}}{z^{4} - 4 \, z^{3} + 6 \, z^{2} - 4 \, z + 1}
\end{array}\right),
&\mathbb D^{\gauche,+} = \mathbb D^{\droite,-} = \mathbb O, \\
\mathbb D^{\gauche,-}
=\left(\begin{array}{rr}
0 & z + \frac{2 \, z^{4} - 2 \, z^{3} + z^{2}}{z^{4} - 4 \, z^{3} + 6 \, z^{2} - 4 \, z + 1} \\
-\frac{z}{z - 1} & z + \frac{2 \, z^{4} - 2 \, z^{3} + z^{2}}{z^{4} - 4 \, z^{3} + 6 \, z^{2} - 4 \, z + 1}
\end{array}\right),
&\mathbb D^{\droite,+}=
\left(\begin{array}{rr}
0 & 0 \\
-\frac{1}{z - 1} + \frac{z}{{\left(z - 1\right)}^{2}} & 0
\end{array}\right).
\end{array}
\end{equation*}
This implies that $p_{\gauche}^+ =  p_{\droite}^-=0$, hence $p_{\droite}^+ = 1 - p_{\gauche}^-$. As a result, the associated $X$-permuton will degenerate into a V shape based at the point $(p_{\gauche}^-,0)$. We can now perform computations in $\mathbb Q(\rho)$ to obtain that $p_{\gauche}^-=-\frac{192}{599} \rho^{4} + \frac{600}{599} \rho^{3} - \frac{1119}{599} \rho^{2} + \frac{1507}{1198} \rho + \frac{343}{599}$. This algebraic number is the only real root of the polynomial
\[19168 z^{5} - 86256 z^{4} + 155880 z^{3} - 141412 z^{2} + 64394 z - 11773\]
and a numerical evaluation gives $p_{\gauche}^- \approx 0.818632668576995$.

\subsection{The class of pin-permutations}\label{sec:PinPermAnnexe}

The recursive description given in \cite{PinPerm} can be translated into a tree-specification
as in \cref{dfn:tree_specification}.

As in \cite{PinPerm}, we denote by (see \cite{PinPerm} for the definitions):
\begin{itemize}
  \item $\mathcal S$ the set of all pin-permutations;
  \item $\mathcal E^+$ (resp. $\mathcal E^-$) the set of increasing (resp. decreasing) oscillations;
  \item $\mathcal N^{+}$ (resp. $\mathcal N^{-}$) the set
    of pin-permutations that are not increasing (resp. decreasing) oscillations,
and whose root is not $\oplus$ (resp. $\ominus$);
\item $\mathcal T_{\mathcal E^+}$ (resp. $\mathcal T_{\mathcal E^-}$)
  the set of direct sums of at least two increasing (resp. decreasing) oscillations;
\item $\mathcal T_{\mathcal E^+,\mathcal N^+}$ (resp. $\mathcal T_{\mathcal E^-,\mathcal N^-}$)
  the set of direct sums of at least two permutations, one being in $\mathcal N^+$,
  the others in $\mathcal E^+$ (resp. $\mathcal N^-$ and $\mathcal E^-$);
  \item $\Si$ the set of simple pin-permutations $\alpha$
    and $\Si^\star$ the set of pairs $(\alpha,a)$ where $\alpha$ is in $\Si$
    and $a$ an active point of $\alpha$;
  \item $\QE^+$ (resp. $\QE^-$) the set of triples $(\beta,m,a)$,
    where $\beta$ is an increasing (resp. decreasing) quasi-oscillation
    and $m$ and $a$ are its main and auxiliary substitution points, respectively.
\end{itemize}
The set of (marked) simple permutations 
$\Si^\star$, $\Si$, $\QE^+$ and $\QE^-$ in the above list
are characterized and enumerated in \cite{PinPerm}. 
\smallskip

Then there is a tree-specification for the following 19 families:
$\mathcal S$, $\mathcal S \backslash \{1\}$,
$\mathcal E^+$, $\mathcal E^+ \backslash \{1\}$, $\mathcal E^+ \backslash \{1,21\}$,
$\mathcal E^-$, $\mathcal E^- \backslash \{1\}$, $\mathcal E^- \backslash \{1,12\}$,
$\mathcal N^+$, $\mathcal N^-$,
 $\mathcal T_{\mathcal E^+}$, $\mathcal T_{\mathcal E^-}$,
 $\mathcal T_{\mathcal E^+,\mathcal N^+}$, $\mathcal T_{\mathcal E^-,\mathcal N^-}$,
 $\mathcal T_{\mathcal E^+}^\star:=\mathcal T_{\mathcal E^+} \setminus \{12,132,213\}$,
 $\mathcal T_{\mathcal E^-}^\star:=\mathcal T_{\mathcal E^-} \setminus \{21,231,312\}$,
 $\{12\}$, $\{21\}$ and $\{1\}$. 

 Below are the equations for $\mathcal S$, $\mathcal T_{\mathcal E^+}$, $\mathcal T_{\mathcal E^+,\mathcal N^+}$
 and $\mathcal N^+$. Some other follow by symmetry or by excluding small permutations.
 \begin{equation}
   \label{eq:system_Pin}
 \begin{cases}
 \begin{array}{rl}
   \mathcal S&=\{\bullet\} \uplus \mathcal T_{\mathcal E^+} \uplus \mathcal T_{\mathcal E^+,\mathcal N^+}
   \uplus \mathcal T_{\mathcal E^-} \uplus \mathcal T_{\mathcal E^-,\mathcal N^-}\\
   &\qquad \uplus\, \biguplus_{\alpha \in \Si} \alpha[1,\dots,1] 
   \uplus\, \biguplus_{(\alpha,i) \in \Si^\star} \alpha[1,\dots,1,\mathcal{S}\setminus \{1\} ,1,\dots,1]\\
   &\qquad \uplus \biguplus_{(\beta,m,a) \in \QE^+} \beta[1,\dots,1,\mathcal{S}\setminus \{1\},1,\dots,1,12,1,\dots,1]\\
   &\qquad
   \uplus \biguplus_{(\beta,m,a) \in \QE^-} \beta[1,\dots,1,\mathcal{S}\setminus \{1\},1,\dots,1,21,1,\dots,1]
   \\
   \mathcal T_{\mathcal E^+} &= \oplus[\mathcal E^+,\mathcal E^+] 
        \uplus \oplus[\mathcal E^+,\mathcal T_{\mathcal E^+}]\\
        \mathcal T_{\mathcal E^+,\mathcal N^+} &= \oplus[\mathcal N^+,\mathcal E^+] 
        \uplus \oplus[\mathcal N^+,\mathcal T_{\mathcal E^+}]
\uplus \oplus[\mathcal E^+,\mathcal N^+] 
        \uplus \oplus[\mathcal E^+,\mathcal T_{\mathcal E^+,\mathcal N^+}]\\ 
   \mathcal N^+&=
   \mathcal T^\star_{\mathcal E^-} \uplus \mathcal T_{\mathcal E^-,\mathcal N^-}\\
   &\qquad \uplus\, \biguplus_{\alpha \in \Si \setminus \mathcal E^+} \alpha[1,\dots,1] 
   \uplus\, \biguplus_{(\alpha,i) \in \Si^\star} \alpha[1,\dots,1,\mathcal{S}\setminus \{1\} ,1,\dots,1]\\
   &\qquad \uplus \biguplus_{(\beta,m,a) \in \QE^+} \beta[1,\dots,1,\mathcal{S}\setminus \{1\},1,\dots,1,12,1,\dots,1]\\
   &\qquad
   \uplus \biguplus_{(\beta,m,a) \in \QE^-} \beta[1,\dots,1,\mathcal{S}\setminus \{1\},1,\dots,1,21,1,\dots,1]
  \end{array}
 \end{cases}
\end{equation}
Finally, the families $\mathcal E^+$ and $\mathcal E^-$ are explicit sets of permutations, 
each consisting of $1$ permutation of size $1$, $1$ permutation of size $2$, and $2$ permutations of each size $n \geq 3$.

The corresponding system is solved explicitly in \cite{PinPerm}.
The critical families are 
$\mathcal S$, $\mathcal S \backslash \{1\}$,
$\mathcal N^+$, $\mathcal N^-$,
 $\mathcal T_{\mathcal E^+,\mathcal N^+}$, $\mathcal T_{\mathcal E^-,\mathcal N^-}$.
 From the equations, we see that the system is essentially linear.
 Here is the dependency graph of the system restricted to critical families.
\begin{figure}[thbp]
\begin{center}
 \includegraphics[width=8cm]{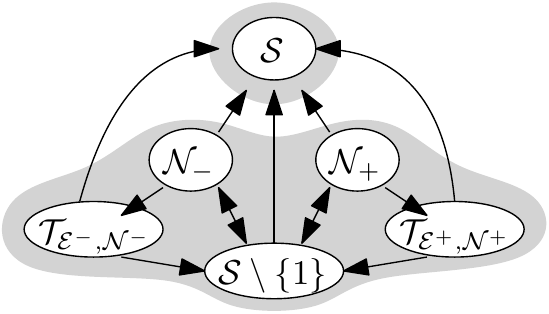}
\end{center}
\caption{The dependency graph of the pin-permutations class.}
\label{Fig:DependencyGraphPin}
\end{figure}

 As in other essentially linear examples, we observe that there are two strongly connected
 components, one constituted of $\mathcal S$ alone.
 The other one contains the family $\mathcal S \setminus \{1\}$, whose asymptotics
 is equivalent to that of $\mathcal S$.
 
 As this specification has infinitely many simple permutations, we need to argue that Hypothesis (RC) holds.
 It is easily observed from the equations that all entries of $\mathbf V^\star$ and $\mathbf M^\star$ are polynomials in the subcritical series
 and in the series $\Si$, $\Si^\star$, $\QE^+$, $\QE^-$ counting the families of simple permutations
 appearing in \eqref{eq:system_Pin}.
 It is shown in \cite{PinPerm} that the latter series are all analytic at the radius of convergence
 of $\mathcal S$, implying (RC).
 
 Moreover, the aperiodicity is clear, so that we can apply \cref{Th:linearCase}
 to the tree-specification without the class $\mathcal S$ and its equation.
 We conclude that a uniform random permutation of size $n$ in $\mathcal S \setminus \{1\}$
 (or equivalently in $\mathcal S$)
 tends to the $X$-permuton with some parameters $p_+^{\gauche},p_+^{\droite},p_-^{\gauche},p_-^{\droite}$.
 Since the class $\mathcal S$ has all symmetries of the square,
 we know without computation that 
 $p_+^{\gauche}=p_+^{\droite}=p_-^{\gauche}=p_-^{\droite}=1/4$.

\subsection{A non-degenerate essentially branching class}\label{sec:Exemple_branching_annexe}
We consider the class $\TTT$ of permutations avoiding the patterns $31452$ and $41253$
whose standard tree has nodes labeled only by $\oplus$, $\ominus$ and $3142$.
This class has the following tree-specification\footnote{See the \href{http://mmaazoun.perso.math.cnrs.fr/pcfs/} {companion Jupyter notebook}  \texttt{examples/Branching.ipynb}}:
\[
\begin{cases}
\begin{array}{rl}
\TTT = \mathcal T_{0}= &\{ \bullet \}\ \uplus\ \oplus[\mathcal T_{1},\mathcal T_{0}]\ \uplus\ \ominus[\mathcal T_{2},\mathcal T_{0}]\ \uplus\ 3142[\mathcal T_{0},\mathcal T_{3},\mathcal T_{3},\mathcal T_{0}]\\
\mathcal T_{1}= &\{ \bullet \}\ \uplus\ \ominus[\mathcal T_{2},\mathcal T_{0}]\ \uplus\ 3142[\mathcal T_{0},\mathcal T_{3},\mathcal T_{3},\mathcal T_{0}]\\
\mathcal T_{2}= &\{ \bullet \}\ \uplus\ \oplus[\mathcal T_{1},\mathcal T_{0}]\ \uplus\ 3142[\mathcal T_{0},\mathcal T_{3},\mathcal T_{3},\mathcal T_{0}]\\
\mathcal T_{3}= &\{ \bullet \}\ \uplus\ \ominus[\mathcal T_{4},\mathcal T_{3}]\\
\mathcal T_{4}= &\{ \bullet \}
\end{array}
\end{cases}
\]
Clearly, $T_4= z$ and $T_3 = \frac{z}{1-z}$.
Since $\mathcal T_0$ contains the separable permutations, the radius of convergence of $T_0$ is smaller than 1.
Hence $T_3$ and $T_4$ are subcritical.
Moreover, $T_0, T_1$ and $T_2$ form a connected component of the dependency graph.
Thus $T_0, T_1$ and $T_2$ are critical and Hypothesis (SC) is satisfied.
In addition, $\mathcal T_0$ and thus all $\mathcal T_i$ contain finitely many simple permutations, so that 
Hypothesis (AR) holds from \cref{obs:polynomial}.
One can see that the specification is essentially branching (\emph{e.g.}, the equation of $T_0$ involves a factor $T_1T_0$).
Finally, $T_3 = \frac{z}{1-z}$ is aperiodic.
We can therefore apply \cref{Th:branchingCase}:
there exists some parameter $p_+$ such that the limiting permuton of $\mathcal T_0$ is the Brownian separable permuton of parameter $p_+$.

We move on to the computation of the parameter $p_+$. We did not explicitly solve the system, but rather reduced it to a cubic equation in $T_0$, and, playing with Cardano's formulas, obtained that the radius of convergence $\rho$ of $T_0$ is the only real root of the equation 
\begin{equation*}
-4z^9 + 41z^8 - 230z^7 + 507z^6 - 582z^5 + 403z^4 - 186z^3 +58z^2 - 12z + 1
\end{equation*}
while the values of the critical series at the radius of convergence can be expressed in terms of $\rho$ as follows:
\begin{equation*}
T_0(\rho) = \frac{-21\rho^5 + 30\rho^4 + 12\rho^3 - 33\rho^2 + 15\rho - 3}{18\rho^5 - 78\rho^4 + 102\rho^3 - 66\rho^2 + 24\rho - 6},\quad T_1(\rho) = T_2(\rho) = \frac{T_0(\rho)}{1+T_0(\rho)}.
\end{equation*}

We obtain directly from the specification
\begin{equation*}
\mathbb M^\star(z,y_0,y_1,y_2) =
\left(\begin{array}{rrr}
y_1 + y_2 + 2y_0(\tfrac{z}{1-z})^2&y_0&y_0\\
y_2 + 2y_0(\tfrac{z}{1-z})^2&0&y_0\\
y_1+ 2y_0(\tfrac{z}{1-z})^2&y_0&0
\end{array}\right),
\end{equation*}
and
\begin{equation*}
E^+_{i,j,j'} = \begin{cases}  1 &\text{ if } i\in\{0,2\}, j=1, j'=0\\ 
0 & \text{ otherwise.} \end{cases}
\end{equation*}

\begin{equation*}
E^-_{i,j,j'} = \begin{cases} 1 &\text{ if } i\in\{0,1\}, j=2, j'=0\\ 
T_3^2 = (\tfrac {z}{1-z})^2 &\text{ if } i\in\{0,1,2\}, j=j'=0\\ 
0 & \text{ otherwise.} \end{cases}
\end{equation*}

We can now perform computations in $\mathbb Q(\rho)$ to find the dominant eigenvectors of the matrix $\mathbb M^\star(\rho,T_0(\rho),T_1(\rho),T_2(\rho))$ and use \cref{eq:DefParametrePermutonBrownien} to compute $p_+$. We get that $p_+ \approx 0.474869237650240$ is the only real root of the polynomial
\[z^{9} - 3 z^{8} + \frac{232819}{62348} z^{7} - \frac{78093}{31174} z^{6} + \frac{243697}{249392} z^{5} - \frac{54293}{249392} z^{4} + \frac{24529}{997568} z^{3} - \frac{125}{62348} z^{2} + \frac{45}{62348} z - \frac{2}{15587}.\]

\subsection*{Acknowledgments}
MB is partially supported by the Swiss National Science Foundation, under grant number 200021-172536. LG is partially supported by Grant ANR-14-CE25-0014 (ANR GRAAL).

\end{document}